\documentclass[graybox,envcountsame,envcountsect]{svmult}
\usepackage{type1cm}        
\usepackage{makeidx}         
\makeindex		     
\usepackage{graphicx}        
\usepackage{multicol}        
\usepackage[bottom]{footmisc}

\usepackage{newtxtext}       %
\usepackage{newtxmath}       

\RequirePackage[all,cmtip]{xy}
\newdir^{ (}{{}*!/-5pt/\dir^{(}}
\newdir_{ (}{{}*!/-5pt/\dir_{(}}
\RequirePackage{xspace,url,bbm}

\usepackage[pagebackref,bookmarksopen]{hyperref}
\hypersetup{
	colorlinks=true,
	linkcolor=blue,
	citecolor=magenta,
	urlcolor=cyan,
}
\usepackage{bookmark} 

\makeatletter
\providecommand*{\toclevel@titlech}{0} 
\edef\toclevel@authorch{\the\numexpr\toclevel@titlech+1} 
\makeatother

\spnewtheorem{notation}[theorem]{Notation}{\itshape}{\rmfamily}
\spnewtheorem{choice}[theorem]{Choice}{\itshape}{\rmfamily}

\newcommand\sfrac[2]{{#1}/{#2}}
\newcommand\Psfrac[2]{({#1}/{#2})}
\newcommand\loccit{loc.\kern3pt cit.\kern-.5pt\xspace}
\newcommand\cf{see\kern.3em}
\newcommand\Cf{See\kern.3em}
\newcommand\eg{e.g.\kern.3em}
\newcommand\ie{i.e.,\ }
\newcommand\resp{\text{resp.}\kern.3em}
\newcommand\moins{\smallsetminus}
\newcommand\cswh{\widehat}
\newcommand\cswt{\widetilde}
\newcommand\csov{\overline}
\newcommand\csemptyset{\varnothing}
\newcommand{\csring}{A}

\newcommand{\lcr}{[\![}
\newcommand{\rcr}{]\!]}
\newcommand{\lpr}{(\!(}
\newcommand{\rpr}{)\!)}

\newcommand\CC{\mathbb{C}}

\newcommand\Gm{\mathbb{G}_\mathrm{m}}

\newcommand\NN{\mathbb{N}}
\newcommand\PP{\mathbb{P}}
\newcommand\QQ{\mathbb{Q}}
\newcommand\RR{\mathbb{R}}
\newcommand\ZZ{\mathbb{Z}}

\newcommand\bma{\boldsymbol{a}}
\newcommand\bme{\boldsymbol{e}}
\newcommand\bD{\boldsymbol{D}}
\newcommand\bL{\boldsymbol{L}}
\newcommand\bP{\boldsymbol{P}}
\newcommand\bR{\boldsymbol{R}}

\newcommand\cA{\mathcal{A}}
\newcommand\cC{\mathcal{C}}
\newcommand\cD{\mathcal{D}}
\newcommand\cE{\mathcal{E}}
\newcommand\cF{\mathcal{F}}
\newcommand\cG{\mathcal{G}}
\newcommand\cH{\mathcal{H}}
\newcommand\cK{\mathcal{K}}
\newcommand\cL{\mathcal{L}}
\newcommand\cN{\mathcal{N}}
\newcommand\cO{\mathcal{O}}
\newcommand\cS{\mathcal{S}}
\newcommand\cT{\mathcal{T}}

\newcommand{\cHom}{\mathop{\cH\!om}\nolimits}

\newcommand\scL{\mathscr{L}}

\DeclareMathOperator{\Can}{Can}
\DeclareMathOperator{\coker}{Coker}
\DeclareMathOperator{\EMHS}{\mathsf{EMHS}}
\newcommand{\sEMHS}{{\scriptscriptstyle\mathrm{EMHS}}}
\DeclareMathOperator{\id}{Id}
\DeclareMathOperator{\im}{Im}
\DeclareMathOperator{\image}{image}
\DeclareMathOperator{\MHM}{\mathsf{MHM}}
\DeclareMathOperator{\Perv}{Perv}
\DeclareMathOperator{\sPerv}{sPerv}
\DeclareMathOperator{\reel}{Re}
\DeclareMathOperator{\rk}{rk}
\DeclareMathOperator{\Var}{Var}

\newcommand{\sfp}{\mathsf{p}}
\newcommand{\catD}{\mathsf{D}}
\newcommand{\Mod}{\mathsf{Mod}}
\newcommand{\varh}{\mathsf{P}}
\newcommand{\pairing}{\mathsf{Q}}
\newcommand\bun{\boldsymbol{1}}

\newcommand\gS{\mathfrak{S}}

\newcommand{\an}{\mathrm{an}}
\newcommand{\rb}{\mathrm{b}}
\newcommand{\rc}{\mathrm{c}}
\newcommand{\rd}{\mathrm{d}}
\newcommand{\dR}{{\mathrm{dR}}}
\newcommand{\rme}{\mathrm{e}}
\newcommand{\gr}{\mathrm{gr}}
\newcommand{\irr}{\mathrm{irr}}
\newcommand{\rrd}{\mathrm{rd}}
\newcommand{\rmod}{\mathrm{mod}\,}
\newcommand{\rp}{\mathrm{p}}

\newcommand{\coH}{\mathrm{H}}\let\bH\coH

\newcommand{\cswtj}{\cswt\jmath}
\newcommand{\sfi}{\texttt{i}}
\newcommand{\BM}{{\scriptscriptstyle\mathrm{BM}}}
\newcommand{\cc}{{\CC\textup{-c}}}
\newcommand{\tS}{{{}^\mathrm{t}\!S}}

\newcommand{\wAfu}{{\cswt{\mathbb{A}}}^{\!1}}
\newcommand{\Af}[1]{\mathbb{A}^{\!#1}}

\newcommand{\p}[1]{{}^{\scriptscriptstyle\rp}#1}
\newcommand{\hp}[1]{\raisebox{6pt}{$\scriptscriptstyle\rp$}#1}
\newcommand{\pcH}{\p\cH{}}
\newcommand{\psip}{\p\psi}\let\ppsi\psip
\newcommand{\phip}{\p\mkern-1mu\phi}\let\pphi\phip
\newcommand{\pring}{\p\!\csring}
\newcommand{\pCC}{\hp\CC}
\newcommand{\pQQ}{\hp\QQ}

\newcommand{\pR}{\hp\!R}

\newcommand\rond{\mathaccent"7017}
\newcommand{\bbullet}{{\scriptscriptstyle\bullet}}
\newcommand{\cbbullet}{{\raisebox{1pt}{$\bbullet$}}}
\newcommand{\csvee}{{\scriptscriptstyle\vee}}
\newcommand{\csperp}{{\scriptscriptstyle\perp}}

\newcommand{\quand}{\quad\text{and}\quad}

\newcommand{\csto}{\mathchoice{\longrightarrow}{\rightarrow}{\rightarrow}{\rightarrow}}
\newcommand\mto{\mathchoice{\longmapsto}{\mapsto}{\mapsto}{\mapsto}}
\newcommand\hto{\mathrel{\lhook\joinrel\csto}}
\newcommand\To[1]{\mathchoice{\xrightarrow{\textstyle\kern4pt#1\kern3pt}}{\stackrel{#1}{\longrightarrow}}{}{}}
\newcommand{\isom}{\stackrel{\sim}{\longrightarrow}}

\newcommand\csbigoplus{\mathop{\textstyle\bigoplus}\displaylimits}
\newcommand\csbigsqcup{\mathop{\textstyle\bigsqcup}\displaylimits}

\begin{document}

\title{Singularities of Functions: a Global Point of View}
\author{Claude Sabbah}
\institute{Claude Sabbah \at{CMLS, CNRS, École polytechnique, Institut Polytechnique de Paris, 91128 Palaiseau cedex, France}
\email{Claude.Sabbah@polytechnique.edu}}
%
%
\maketitle


\vspace*{-2\baselineskip}
\section*{Contents}
\setcounter{minitocdepth}{2}
\dominitoc

\abstract*{This text surveys cohomological properties of pairs $(U,f)$ consisting of a smooth complex quasi-projective variety $U$ together with a regular function on~it. On the one hand, one tries to mimic the case of a germe of holomorphic function in its Milnor ball and, on the other hand, one takes advantage of the algebraicity of~$U$ and $f$ to apply technique of algebraic geometry, in particular Hodge theory. The monodromy properties are expressed by means of tools provided by the theory of linear differential equations, by mimicking the Stokes phenomenon. In the case of tame functions on smooth affine varieties, which is an algebraic analogue of that of a holomorphic function with an isolated critical point, the theory simplifies much and the formulation of the results are nicer. Examples of such tame functions are exhibited.}

\abstract{This text surveys cohomological properties of pairs $(U,f)$ consisting of a smooth complex quasi-projective variety $U$ together with a regular function on~it. On the one hand, one tries to mimic the case of a germ of holomorphic function in its Milnor ball and, on the other hand, one takes advantage of the algebraicity of~$U$ and $f$ to apply technique of algebraic geometry, in particular Hodge theory. The monodromy properties are expressed by means of tools provided by the theory of linear differential equations, by mimicking the Stokes phenomenon. In the case of tame functions on smooth affine varieties, which is an algebraic analogue of that of a holomorphic function with an isolated critical point, the theory simplifies much and the formulation of the results are nicer. Examples of such tame functions are exhibited.}

\section{Introduction}
Milnor's theory of singularities of complex hypersurfaces \cite{Milnor68} focuses on the properties of a germ of a holomorphic function $f:(\CC^{d},0)\csto(\CC,0)$ at a critical point $0$, which produces a singularity of the germ of hypersurface $(f^{-1}(0),0)$. One realizes the germ $f$ as a holomorphic function defined on some neighborhood of $0$, usually taken as an open ball (a \index{Milnor!ball}Milnor ball) in some local coordinate system, and the properties considered do not depend on the radius of the ball, provided it is ``sufficiently small'' (a condition that can be made precise in terms of a Thom-Whitney stratification of this neighborhood). See \cite{L-NB-S20} for details. A more global aspect of the theory arises when considering unfoldings of a germ $(f,0)$: having fixed a Milnor ball for a germ $(f,0)$, this ball is usually not a Milnor ball for the nearby functions of the unfolding (see \cite{Ebeling20}). Nevertheless, some ``smallness'' properties remain valid.

In this text, we consider the global setting of a regular function $f:U\csto\Af1$ on a smooth complex quasi-projective variety $U$ of dimension $d$. Particular cases are those for which~$U$ is affine: \eg $U$ is the affine space $\Af{d}$, the complex torus $(\Gm)^{d}$ (where $\Gm$ is the multiplicative group $\Af1\moins\{0\}$), a product of those, or the complement of a hypersurface in these examples, like the complement of an arrangement of hyperplanes in an affine space.

Comparing with Milnor's situation, we can say in a figurative way that the singularity of $f$ that is of interest for us is ``the point at infinity on $U$'' and that its limit Milnor ball is $U$ itself. This approach suggests to analyze the map $f$ ``from inside'', that is, by means of a family of balls (for a suitable metric) of larger and larger radius. Many choices of such metrics exist however. On the other hand, in order to analyze the map $f$ ``from outside'', one can ``fill the hole at infinity'' by choosing a projectivization of $f$, that is, a complex quasi-projective variety~$X$ together with a projective morphism $f_X:X\csto\Af1$, such that $U$ is Zariski dense in~$X$ and~$f$ is the restriction of $f_X$ to $U$. Such a projectivization is also far from unique.

A vast literature is devoted to the analysis of the critical points at infinity and their influence (or the influence of their absence) on the topology of the map $f$, particularly when~$f$ is a polynomial mapping. Furthermore, several examples have highlighted unexpected behaviors. As it is impossible to list all papers related to this subject, we refer the reader to \cite{A-L-M00b,Tibar07} and the references therein.

In the case of an \emph{isolated critical point} of a germ of holomorphic function, the following (co)homological properties have been obtained:
\begin{enumerate}
\item
The reduced cohomology (with $\ZZ$-coefficients) of the Milnor fiber is non zero only in degree $d-1$; it is a free $\ZZ$-module whose rank $\mu$ is the \emph{Milnor number} of $f$ at its critical point; it cannot be zero at a critical point.
\item
This cohomology comes with an automorphism $T$ (monodromy), with an intersection pairing and a Seifert pairing, the latter being nondegenerate. \item
The corresponding homology groups are either the reduced homology of the Milnor fiber in degree $d-1$ (\index{vanishing cycles}vanishing cycles) or the homology of the \index{Milnor!ball}Milnor ball relative to the Milnor fiber (Lefschetz thimbles). The topology that gives rise to these homological objects is well understood, including in its symplectic aspects. The relations between the objects considered above, and others, are efficiently summarized in \cite{Hertling05}.

\item
The cohomology of the Milnor fiber underlies a $\ZZ$-mixed \index{Hodge!mixed -- structure}Hodge structure with \index{filtration!weight --}weight filtration determined by the nilpotent part of the monodromy (\cf \cite{Steenbrink22}). This structure provides us with numerical invariants: the spectrum and the spectral pairs.

\item
These Hodge invariants can also be recovered by means of the motivic Milnor fiber introduced by Denef and Loeser (\cf \cite{D-L01} and also \cite{Steenbrink22}).
\item
A major object that comes into play in the analytic theory for computing Hodge invariants is the \index{lattice!Brieskorn --}\emph{Brieskorn lattice} introduced by Brieskorn \cite{Brieskorn70} in order to obtain an algebraic expression of some of the topological invariants above. There is also a natural nondegenerate pairing on the Brieskorn lattice, with values in the ring of formal power series of one variable, called the higher residue pairing, introduced by K.\,Saito \cite{KSaito83} (\cf \cite{Pham85b} and \cite{Hertling05} for the relations with the topological pairings).
\end{enumerate}

In the case of a regular function $f\in\cO(U)\moins\CC$, having isolated critical points is not a sufficient condition to realize similar properties. One should also impose the absence of critical points at infinity: this is the \emph{tameness property} that we will discuss with details in Section \ref{sec:tameness}. We aim at defining a \index{vanishing cycles}vanishing cycle $\ZZ$-module that shares the properties above. Supplementary data occur, describing the relations between the various isolated critical points of $f$: the monodromy of the ``vanishing cohomology'' is enhanced with a supplementary set of data, called the \emph{Stokes matrices}. Such data already show up in the semi-local situation mentioned at the beginning. The term ``Stokes matrices'' is justified by the fact that such matrices are Stokes matrices of a suitable linear system of differential equations of one variable with an irregular singularity. This was already observed by F.\,Pham, with the title of his paper \cite{Pham85b}: ``La descente des cols par les onglets de Lefschetz avec vues sur Gauss-Manin''.\footnote{``The descent of the passes (saddle points) by Lefschetz thimbles with views on Gauss-Manin.''}

The general case of a possibly non tame regular function (with non-isolated critical points), while not being as simple as the tame case, shows many interesting properties. One cannot expect the absence of $\ZZ$-torsion in the various cohomology groups that occur, so we mainly consider the case of rational or complex coefficients. The exponentially twisted de~Rham complex plays an important role for obtaining Hodge-theoretic properties. Another approach, in the spirit of the motivic arc-space approach of Denef-Loeser, has been developed by Raibaut in \cite{Raibaut12}: the focus is on the motivic Milnor fiber of $f$ at $\infty$ and the various invariants that one can deduce from the computation of this object

In contrast with the case of the germs of holomorphic functions, the case of pairs $(U,f)$ that we are considering has strong relations to several other areas of algebraic geometry and arithmetic that we will very briefly mention in this text. On the one hand, the theory of exponential Nori motives, developed in \cite{F-J17} by Fresán and Jossen, makes it possible, when~$U$ and $f$ are defined over a subfield $\mathbbm{k}$ of $\CC$, to relate Hodge properties of $(U,f)$, in particular such as those defined in Section \ref{subsec:IrrMHS}, to arithmetic properties of $f$.

On the other hand, mirror symmetry for Fano manifolds aims at associating to any smooth quasi-projective Fano manifold a Landau-Ginzburg model, that is realized as a regular function on another quasi-projective manifold. The theory has been developed in \cite{K-K-P14}, where various conjectures have been settled. In the version of this correspondence involving the Frobenius manifolds introduced by Dubrovin (\cite{Dubrovin96}), the role of the \index{Hodge!irregular -- structure}irregular Hodge structure produced by the Brieskorn lattice is essential (\cf for example \cite{D-S02a,D-S02b,Douai07}). There is also a categorical approach to such a correspondence (\cf\cite{K-K-P08}), that has motivated a categorical realization of the twisted de~Rham complex, as proposed by Shklyarov in~\cite{Shklyarov17}.

All over the text, we use the following notation.

\begin{notation}\label{nota}\mbox{}
\begin{enumerate}
\item\label{nota1}
The exponential map identifies $\RR/2\pi\ZZ$ with $S^1$ by $\theta\mto\rme^{\sfi\theta}$ (with $\sfi=\sqrt{-1}$). We~will abuse notation by writing $\theta\in S^1$.
\item\label{nota2}
Let $Z$ be a topological space and $V$ be an open subset of $Z$. For a sheaf~$\cF$ on~$Z$, the notation $\cF_V$ refers to the sheaf on $Z$ obtained by restricting $\cF$ to the open set~$V$ and then extend it by zero to $Z$.
\item
The base ring $\csring$ is either $\ZZ$ or a field, usually $\QQ$ or $\CC$. For most arguments, a commutative ring with finite global homological dimension would suffice (\cf\cite[Conv.\,3.0]{K-S90}).
\end{enumerate}
\end{notation}

\section{Topological/cohomological properties}\label{sec:topprop}
In this text, we consider a regular function $f:U\csto\Af1$ from a connected smooth complex quasi-projective variety $U$ of dimension $d$ to the affine line $\Af1$. This set of data is also denoted by $f\in\cO(U)$. Let us emphasize that, unless otherwise stated, we~work with algebraic varieties endowed with their Zariski topology and, for example, we will denote by~$U^\an$ the underlying complex manifold. Some results below do not need the smoothness assumption on $U$ to hold, but we will not try to state them in the most general setting.

The \index{set!critical --}critical set of $f$ is a Zariski closed subset of $U$ whose image by $f$ is a finite subset $C=C(f)$ of $\Af1$ (by Sard's lemma in the algebraic context), called the \index{set!of critical values}set of critical values of~$f$. A fortiori, the restriction $f:U\moins f^{-1}(C)\csto\Af1\moins C$ does not have any critical point.

A \index{projectivization}\emph{projectivization} of $f$ is a \emph{projective morphism} $f_X$ from a (possibly not smooth) quasi-projective variety $X$ to $\Af1$ such that $U$ is a Zariski dense open subset of $X$ and $f$ is the restriction of $f_X$ to $U$. Assume $U$ is a locally closed subspace of some projective space $\PP^N$. A natural projectivization of $f$ is obtained by considering the Zariski closure $X$ in $\PP^N\times\Af1$ of the graph $\Gamma(f)\subset U\times\Af1$ and by defining $f_X$ as induced by the second projection. By resolving singularities of $X$ (they lie in $X\moins U$), we can choose, if needed, $X$ to be smooth and also $X\moins U$ to be a divisor with normal crossings.

\subsection{The fibration theorem}

\begin{theorem}\label{th:Bf}
There exists a smallest finite set $B(f)\subset\Af1$, called the \index{set!bifurcation --}\emph{bifurcation set of $f$}, such that the map $f^\an:(U\moins f^{-1}(B(f)))^\an\csto\CC\moins B(f)$ is a $C^\infty$ fibration.
\end{theorem}

\begin{proof}[Sketch]
One chooses a projectivization $f_X:X\csto\Af1$ of $f$. A theorem of \hbox{Hironaka} \cite{Hironaka77} (\cf also \cite[Th.\,6.4.5]{L-NB-S20}) asserts that the underlying complex analytic space~$X^\an$ admits a Whitney stratification satisfying Thom's $A_{f_X}$ condition. In fact, in~the quasi-projective context, the strata can be chosen quasi-projective and connected, so that the $f_X$-image of a stratum is either a point or a Zariski dense open set of~$\Af1$. The construction can be made in order to ensure that $U\moins f^{-1}(B)$ is a stratum, for some finite set $B$ of $\Af1$. Thom's first isotopy lemma implies that the restriction of~$f_X$ to each stratum $S$ induces a $C^\infty$ fibration $S^\an\csto f_X(S)^\an$. Applying this to the stratum $U\moins f^{-1}(B)$ yields the assertion for some (possibly non minimal) finite set~$B$. Minimality can then obviously be achieved.
\end{proof}

The set $B(f)$ contains the set $C=C(f)$ of critical values of $f$ : otherwise, for some $c\in C$, the singular fiber $f^{-1}(c)^\an$ would be a $C^\infty$ manifold; this is not possible, according to \cite[Rem.\,p.\,13]{Milnor68}.

The set $B(f)\moins C$ consists of \index{set!of atypical critical values}\emph{atypical critical values of $f$}. It is in general not empty. For example (\cf \cite[Ex.\,3.4]{A-L-M00b}), the polynomial $f:\Af{3}\csto\Af1$ defined by $f(x,y,z)=x(1+xyz)$ has no critical point, but $B(f)\neq\csemptyset$.

Determining $C$ from $f$ is an algebraic computation. On the other hand, determining $B(f)$ is much harder. A topological criterion has been provided in \cite{H-L84} for $f:\Af3\csto\Af1$ in terms of the topological Euler characteristic of the fibers of $f$. Other results in this direction can be found in \cite{A-L-M00b} and the references therein.

On the other hand, we will focus on the case where $B(f)=C$ and, more precisely, where~$f_X$ has ``no critical point at infinity'', called the \index{tameness}\emph{tameness property}, in~Section~\ref{sec:tameness}.

\subsection{Cohomological tools}\label{subsec:cohomtools}
We aim at defining a space of global \index{vanishing cycles!global --}vanishing cycles for a regular function \hbox{$f:U\csto\Af1$}. For that purpose, we will make use of some of the tools explained in \cite{M-S22} (a good reference is also \cite{Dimca04}, and the reference book is \cite{K-S90}), that we recall here for convenience. \emph{In this section, the varieties are equipped with their analytic topology, and we will omit the exponent $\an$ for the sake of simplicity.}

\subsubsection*{Algebraically constructible complexes}
Our base ring $\csring$ is either $\ZZ$ (a PID) or $\QQ,\RR,\CC$ (a field). Given a complex quasi-projective variety $Y$ (which is possibly singular), an \emph{algebraically \index{complex!algebraically constructible --}constructible sheaf of $\csring$-modules} on $Y$ is a sheaf $\cF$ of $\csring$-modules on~$Y$ for which there exists a stratification of $Y$ by locally closed connected smooth quasi-projective subvarieties $(Y_a)_{a\in\cA}$ such that, for each $a\in\cA$, the sheaf $\cF|_{Y_a}$ is a \index{sheaf!locally constant --}locally constant sheaf of $\csring$-modules of finite type. We implicitly assume (\cf \eg \cite{Trotman20}) that a stratification is locally finite and satisfies the frontier condition. A priori, it is not assumed to satisfy any regularity condition, but one can always refine it to do so. These sheaves form a category for which the morphisms are all morphisms of sheaves (it is a \emph{full} subcategory of the category of sheaves of $\csring$-modules). This category is \emph{abelian}: this statement amounts to the property that the kernel and cokernel of a sheaf morphism between locally constant sheaves of $\csring$-modules of finite type are of the same type.

One can then form the bounded derived category of this abelian category, which is equivalent to the full subcategory $\catD^\rb_\cc(Y,\csring)$ of the bounded derived category of sheaves of $\csring$-modules whose objects consists of bounded complexes having constructible cohomology.

Given a morphism $f:Y\csto Y'$ between quasi-projective varieties, the derived \index{pushforward}pushforward functors $\bR f_*$ and $\bR f_!$ (\index{pushforward!with proper support}pushforward with proper support) are defined from $\catD^\rb_\cc(Y,\csring)$ to $\catD^\rb_\cc(Y',\csring)$ as well as the pullback functors $f^{-1}$ and $f^!$ from $\catD^\rb_\cc(Y',\csring)$ to $\catD^\rb_\cc(Y,\csring)$. \index{duality!Poincaré-Verdier --}Poincaré-Verdier duality $\bD$ is a contravariant equivalence of categories from $\catD^\rb_\cc(Y,\csring)$ to itself, and is its own quasi-inverse, that is, it~satisfies $\bD\circ\bD\simeq\id$.

\subsubsection*{Perverse and strongly perverse complexes}
The category $\catD^\rb_\cc(Y,\csring)$ is naturally equipped with a \index{t-structure}t-structure
\[
(\catD^{\rb,\leqslant0}_\cc(Y,\csring),\catD^{\rb,\geqslant0}_\cc(Y,\csring)).
\]
The heart of this t-structure is the abelian category of \index{sheaf!perverse --}perverse ``sheaves'' $\Perv(Y,\csring)$ (more accurately, \index{complex!perverse --}perverse complexes). For example, if $Y$ is smooth of dimension $n$, the shifted constant sheaf
\[
\pring_Y:=\csring_Y[n]
\]
belongs to $\Perv(Y,\csring)$.

If $\csring$ is a field, both terms of the t-structure are exchanged by Poincaré-Verdier duality~$\bD$, hence $\Perv(Y,\csring)$ is preserved by $\bD$. If $\csring$ is a PID, this property may not hold, and we are led to considering \index{complex!strongly perverse --}\emph{strongly} perverse complexes (\cf\cite[Def.\,10.2.50]{M-S22}). We denote by $\sPerv(Y,\csring)$ the corresponding full subcategory of $\Perv(Y,\csring)$, which is preserved by Poincaré-Verdier duality $\bD$ (it is equal to $\Perv(Y,\csring)$ if $\csring$ is a field). Be careful that it is not abelian, as for example the cokernel of a group morphism $\ZZ^p\csto\ZZ^q$ may acquire $\ZZ$-torsion.

\subsubsection*{Nearby and vanishing cycle functors}
We will make intensive use of these functors. Given a smooth quasi-projective variety $U$, a regular function $f:U\csto\Af1$, and a constructible complex $\cF\in\catD^\rb_\cc(U,\csring)$, one can associate to these data and each point $c\in\Af1$ a pair of objects of $\catD^\rb_\cc(U,\csring)$ supported on the fiber $f^{-1}(c)$, that we denote by $\ppsi_{f-c}(\cF)$ (\index{complex!of nearby cycles}complex of nearby cycles) and $\pphi_{f-c}(\cF)$ (\index{vanishing cycles!functor}\index{complex!of vanishing cycles}complex of vanishing cycles), where the upper index p indicates a shift by $-1$ with respect to the standard definition, see \cite[Def.\,10.4.4]{M-S22}, in order to obtain a better behavior with respect to Poincaré-Verdier duality, see \cite[Th.\,10.4.21]{M-S22}. These functors come equipped with a monodromy automorphism $T_c$ and a diagram of morphisms $(\Can_c,\Var_c)$ commuting with monodromy automorphisms $T_c$:
\begin{equation}\label{eq:canvar}
\xymatrix@C=2.5cm{
\ppsi_{f-c}(\cF)\ar@/^1pc/[r]^-{\Can_c}&\ar@/^1pc/[l]^-{\Var_c}\ar@/^1pc/[l]\pphi_{f-c}(\cF)
}
\end{equation}
and satisfying
\begin{equation}\label{eq:canvarT}
\begin{aligned}
\Var_c\circ\Can_c&=T_c-\id\text{ (on $\ppsi_{f-c}(\cF)$)},\\
\Can_c\circ\Var_c&=T_c-\id\text{ (on $\pphi_{f-c}(\cF)$)}.
\end{aligned}
\end{equation}
For example, the complex of vanishing cycles $\pphi_{f-c}(\pring_U)$ is supported on the critical locus of~$f$.

These functors $\ppsi_{f-c},\pphi_{f-c}$ are t-exact and preserve $\Perv(U,\csring)$ as well as $\sPerv(U,\csring)$ (\cf\cite[Rem.\,10.4.23]{M-S22}). In particular, $\pphi_{f-c}(\pring_U)$ is strongly perverse.

In dimension one, these functors can be used to characterize objects of $\sPerv$ with coefficients in $\ZZ$. For example, let $\Delta$ be a disc with coordinate $t$ and let $\Perv(\Delta,0;\ZZ)$ be the category of \index{sheaf!perverse --}perverse sheaves of $\ZZ$-modules on $\Delta$ which have singularities at $t=0$ at most.

\begin{lemma}\label{lem:sperv}
An object $\cF$ of $\Perv(\Delta,0;\ZZ)$ belongs to $\sPerv(\Delta,0;\ZZ)$ if and only if the $\ZZ$-modules $\ppsi_t(\cF)$ and $\pphi_t(\cF)$ are $\ZZ$-free.
\end{lemma}

\begin{proof}
Let $\cF$ be an object of $\Perv(\Delta,0;\ZZ)$. Assume that $\ppsi_t(\cF)$ and $\pphi_t(\cF)$ are $\ZZ$\nobreakdash-free, and let us prove that $\cF$ is an object of $\sPerv(\Delta,0;\ZZ)$. Freeness of $\ppsi_t(\cF)$ as a $\ZZ$\nobreakdash-mod\-ule is equivalent to the property that the restriction of~$\cF$ to the punctured disc~$\Delta^*$ is a \index{sheaf!locally constant --}locally constant sheaf of free $\ZZ$-modules. According to \cite[Prop.\,10.2.49]{M-S22}, we~are reduced to checking that $\cH^0(i_0^!\cF)$ is free, if~$i_0:\{0\}\hto\Delta$ denotes the inclusion. The distinguished triangle \cite[(10.90)]{M-S22} reads
\[
i_0^!\cF\csto\phip_t(\cF)\To{\Var_t}\psip_t(\cF)\To{+1}
\]
and since $\psip_t(\cF),\phip_t(\cF)$ are concentrated in degree zero, it leads to the long exact sequence
\[
0\csto\cH^0(i_0^!\cF)\csto\phip_t(\cF)\To{\Var_t}\psip_t(\cF)\csto\cH^1(i_0^!\cF)\csto0.
\]
If $\phip_t(\cF)$ is free, then so is $\cH^0(i_0^!\cF)$.

Conversely, let us assume that $\cF$ is an object of $\sPerv(\Delta,0;\ZZ)$. Then the local system (up to a shift) $\cF|_{\Delta^*}$ is free, and this implies freeness of $\psip_t(\cF)$ and therefore that of $\image(\Var_t)$. The assumption implies that $\cH^0(i_0^!\cF)$ is free, and therefore so is $\phip_t(\cF)$.
\end{proof}

\subsubsection*{Compatibility with pushforward}
A technique to obtain cohomological information about the map $f:U\csto\Af1$ is to analyze the pushforward complexes $\bR f_*\csring_U$ (ordinary pushforward) and $\bR f_!\csring_U$ (pushforward with proper support), which are both algebraically constructible, that~is, both are objects of $\catD^\rb_\cc(\csring_{\Af1})$. This is especially useful for analyzing global properties of nearby or \index{vanishing cycles!functor}vanishing cycles of $f$ along a fiber $f^{-1}(c)$. We know (\cf \cite[Prop.\,10.4.19]{M-S22}) that the functors $\psi_{f-c},\phi_{f-c}$ do commute with $\bR f_*,\bR f_!$ when $f$ is \emph{proper} (so that the latter functors are the same), but in general not otherwise. More precisely, let $f_X:X\hto\Af1$ be a projectivization of $f$ and let $j:U\hto X$ denote the open inclusion. Since derived pushforwards compose well, we have (the second isomorphism since $f_X$ is proper)
\[
\bR f_*\csring_U\simeq\bR f_{X*}(\bR j_*\csring_U),\quad \bR f_!\csring_U\simeq\bR f_{X*}(\bR j_!\csring_U).
\]
It follows that, considering vanishing cycles for example, we have
\[
\phi_{t-c}(\bR f_*\csring_U)\simeq \bR f_{X*}(\phi_{f_X-c}(\bR j_*\csring_U))
\]
in a way compatible with monodromy, but the right-hand side is in general distinct from $\bR f_*(\phi_{f-c}\csring_U)\simeq \bR f_{X*}(\bR j_*(\phi_{f-c}\csring_U))$. A special case where this commutation does take place nevertheless if the case of tame functions considered in detail in Section~\ref{sec:tameness}.

On the other hand, in order to exploit t-exactness of the shifted nearby and \index{vanishing cycles!functor}vanishing cycle functors $\ppsi_{f-c},\pphi_{f-c}$, it is suitable to consider the \index{sheaf!perverse cohomology --}perverse cohomology sheaves of the pushforward complexes:
\[
\pR^k\!f_*(\pring_U):=\pcH^k(\bR f_*(\pring_U))\quand\pR^k\!f_!(\pring_U):=\pcH^k(\bR f_!(\pring_U)),
\]
where we recall that $\pring_U=\csring_U[d]$, with $d:=\dim U$.

\begin{proposition}\label{prop:pushphi}
Let $f_X$ be a projectivization of $f$ with $X$. Then, for each $k$, we~have
\[
\pphi_{t-c}(\pR^k\!f_*\csring_U)\simeq \pR^k\!f_{X*}(\pphi_{f_X-c}(\bR j_*\pring_U)),
\]
and similar isomorphisms with the functors $\ppsi$, or $f_!$ with $\bR j_!\pring_U$.
\end{proposition}

\begin{proof}
Since $f_X$ is proper, we have identifications
\[
\pphi_{t-c}(\bR f_*\csring_U)\simeq\pphi_{t-c}(\bR f_{X*}(\bR j_*\csring_U))\simeq\bR f_{X*}(\pphi_{f_X-c}(\bR j_*\csring_U)).
\]
The result follows by taking the $k$-th perverse cohomology of both terms and using that $\pphi_{t-c}$ is t-exact.
\end{proof}

\section{Cohomologies attached to a pair \texorpdfstring{$(U,f)$}{Uf}}
\label{sec:cohUf}
A possible approach to associate with a regular function $f\in\cO(U)$ on a smooth quasi-projective variety $U$ some (co)homological invariants is to extend those existing for~$U$ itself, corresponding to the zero function on $U$.

On the topological side, to $U^\an$ is associated the singular homology and the singular homology with closed supports (Borel-Moore homology), as well as the singular cohomology and the singular cohomology with compact support (all with coefficients in $\csring$). When $\csring=\QQ$, various nondegenerate pairings relate these finite dimensional $\QQ$-vector spaces: intersection of cycles, integral of cohomology classes on cycles and cup product of cohomology classes.

On the de~Rham side, the algebraic de~Rham complex of $U$ leads to de~Rham cohomology and de~Rham cohomology with compact support. Grothendieck's comparison theorem identifies these complex vector spaces with their topological analogues with complex coefficients. If the variety $U$ is defined over $\QQ$ (\ie by means of equations with rational coefficients), the de~Rham complex is also defined over $\QQ$, as well as its cohomologies, so that the comparison isomorphism can be understood as comparing two different $\QQ$-structures on the same $\CC$-vectors space. Its matrix in $\QQ$-bases, usually consisting of transcendental numbers, is called the \emph{period matrix} of $U$ and can be realized by integrating a basis of rational differential forms over a basis of cycles.

On the other hand, the de~Rham cohomology can be filtered by $\CC$-subspaces in order to produce, together with the singular cohomology, a \index{Hodge!mixed -- structure}mixed Hodge structure (\cf \cite{DeligneHII}).

A unifying approach to these properties is obtained via the notion of motive, in~particular that of Nori motive (\cf \cite{H-MS17}).

Our aim in this section is to define similar homology and cohomology spaces, denoted generically by $\coH(U,f)$, with lower or upper decorations meaning ``Borel-Moore'', ``compact support'', ``de~Rham'', together with the various comparison isomorphisms and pairings, and with coefficients in $\csring$ for singular cohomology and~$\CC$ for de~Rham cohomology. A new set of decorations occurs, namely, \emph{rapid decay} and \emph{moderate growth}. We will also equip the de~Rham version with a pair of filtrations called the \emph{irregular Hodge filtration} and the weight filtration. Although this pair of filtrations does not form a mixed Hodge structure in the usual sense, in particular because the irregular Hodge filtration is indexed by rational numbers, this structure can be called an \index{Hodge!irregular -- structure}\emph{irregular mixed Hodge structure}, and the pairs consisting of the jumping indices of both filtrations are reminiscent of the spectral pairs occurring in the theory of isolated hypersurface singularities (\cf\cite[\S9.8]{Steenbrink22} and the reference therein). In~particular, a Thom-Sebastiani property holds in this context: given regular functions $f\in\cO(U)$ and $g\in\cO(V)$, this property relates the invariants of $f\boxplus g\in\cO(U\times V)$ with those of~$f$ and~$g$ (\cf\cite[\S3.4]{Bibi15}).

\subsection{Various expressions of the singular cohomology with growth conditions}\label{subsec:singcoh}

A \index{period}period of~$U$ is any integral $\int_\gamma\omega$, where $\omega$ is an algebraic differential form of some degree on $U$ and $\gamma$ is a cycle of suitable dimension on $U$. Such an integral may exist when~$\gamma$ is a Borel-Moore cycle (\ie a cycle with closed support), and the situation is controlled by the Poincaré-de~Rham duality theorem. The corresponding integral attached to a pair $(U,f)$ would be the integral $\int_\gamma\rme^{-f}\omega$, and, whatever the algebraic differential form $\omega$~is, it exists for Borel-Moore cycles~$\gamma$ as soon as the real part of~$f$ remains positive on the support of~$\gamma$ when $|f|$ is big enough. Indeed, the holomorphic form $\rme^{-f}\omega$ has then rapid decay along $\gamma$ when $|f|\csto\infty$, hence is integrable. These considerations lead us to define the \emph{moderate growth (co)homology} and, for duality purposes, the \emph{rapid decay (co)homology}, of~$(U,f)$. \emph{In the remaining part of Section~\ref{subsec:singcoh}, we consider the spaces with their analytic topology and we omit the exponent~$\an$.}

In other words, if $\omega$ is a closed algebraic $r$-form on $U$, we are led to considering that $\rme^{-f}\omega$ defines a class in the relative cohomology space $\coH^r(U,f^{-1}(t);\CC)$, where~$t$ is a positive real number ``large enough''. More intrinsically, we should consider the space $\lim_{t\csto+\infty}\coH^r(U,f^{-1}([t,+\infty));\CC)$. This space can be defined with coef\-ficients in~$\csring$ instead of~$\CC$, and we will give various different expressions for it, that we denote by $\coH^r(U,f;\csring)$.

Let us emphasize a way to simplify various arguments, which proves much useful. The idea is to simplify the space $U$ by replacing it with $\Af1$. The price to pay is making the sheaf~$\csring_U$ more complicated, by replacing it with a constructible complex of sheaves on~$\Af1$: we will consider the pushforward complexes $\bR f_*\csring_U$ and $\bR f_!\csring_U$, which have algebraic constructible cohomology on~$\Af1$. Let us also emphasize at this point that using \index{sheaf!perverse cohomology --}perverse cohomology sheaves instead of ordinary cohomology sheaves will simplify many arguments. We fix the following choice for $C\subset\Af1$ and $\rho\gg0$.

\begin{choice}\label{choice:Crho}\mbox{}
\begin{itemize}
\item
We fix a projectivization $f_X:X\csto\Af1$ of $f$ such that the complement $H=X\moins U$ is a divisor with normal crossings whose irreducible components are smooth. We~denote by $j:U\hto X$ the open inclusion.
\item
We fix a finite subset $C\subset\Af1$ such that the map $f_X:X\moins f_X^{-1}(C)\csto\Af1\moins C$ is smooth and, for each $x\in H\moins f_X^{-1}(C)$, the germ $f_{X,x}:((X,H),x)\csto(\CC,f(x))$ is isomorphic to the projection
\[
(\CC,f(x))\times ((\CC^{d-1},0),(H',0))\csto(\CC,f(x)),
\]
where $H'$ is a union of coordinate hyperplanes of $\CC^{d-1}$.
\item
We fix $\rho>0$ large enough so that the interior $\rond\Delta_\rho$ of a closed disc $\Delta_\rho$ in~$\Af1$ of radius $\rho>0$ contains $C$.
\end{itemize}
\end{choice}

\subsubsection*{Passage from \texorpdfstring{$U$ to $f^{-1}(\Delta_\rho)$}{U}}
With the previous choice, the complexes $\bR f_*\csring_U$ and $\bR f_!\csring_U$ have \index{sheaf!locally constant --}locally constant cohomology sheaves on $\Af1\moins C$.

\begin{lemma}
For $t=\rho$ and any $r\geqslant0$, the restriction morphism
\[
\coH^r(U,f^{-1}(\rho);\csring)\csto\coH^r(f^{-1}(\Delta_\rho),f^{-1}(\rho);\csring)
\]
is an isomorphism of relative cohomology spaces, and both spaces are independent of~$\rho$ up to natural isomorphisms.
\end{lemma}

\begin{proof}
In view of the exact sequence of relative cohomology, it is enough to show the restriction morphism $\coH^r(U;\csring)\csto\coH^r(f^{-1}(\Delta_\rho);\csring)$ is an isomorphism. By~pushing forward by $f$, we are led to showing a similar result for the hypercohomology of $\bR f_*\csring_U$ on $\Af1$. Let $j:\Af1\moins\Delta_\rho\hto\Af1$ denote the open inclusion. We are led to showing the vanishing of any hypercohomology space of the complex $\bR j_!j^{-1}(\bR f_*\csring_U)$. Since $j^{-1}(\bR f_*\csring_U)$ has \index{sheaf!locally constant --}locally constant cohomology on $\Af1\moins\Delta_\rho$ by our assumption on $\rho$, a simple induction reduces the question to the vanishing, for any $r\in\NN$, of the cohomology \hbox{$\coH^r(\Af1\moins\rond\Delta_\rho;j_!\cL)$} for a locally constant sheaf $\cL$ of $\csring$-modules on $\Af1\moins\Delta_\rho$. By~considering the radial projection $\Af1\moins\rond\Delta_\rho\csto\partial\Delta_\rho$, this follows from the vanishing of the relative cohomology $\coH^r([\rho,+\infty),\{\rho\};\csring)$ for each $r\in\NN$.
\end{proof}

\subsubsection*{Real oriented blow-up (1)}
As $f$ is a $C^\infty$ fibration over $\partial\Delta_\rho$ (Theorem \ref{th:Bf}), we can replace in the above formula the point $\rho$ with the closed half-circle $\rho\rme^{\sfi[-\pi/2,\pi/2]}\subset\partial\Delta_\rho$ and write the right-hand side as $\coH^r(f^{-1}(\Delta_\rho),f^{-1}(\rho\rme^{\sfi[-\pi/2,\pi/2]});\csring)$. We will interpret this relative cohomology as a cohomology with compact support. With the above choice of $X$, we can rewrite this relative cohomology as the relative hypercohomology of the complex $\bR j_*\csring_U$ on~$X$:
\[
\coH^r(f^{-1}(\Delta_\rho),f^{-1}(\rho\rme^{\sfi[-\pi/2,\pi/2]});\csring)\simeq\coH^r(f_X^{-1}(\Delta_\rho),f_X^{-1}(\rho\rme^{\sfi[-\pi/2,\pi/2]});\bR j_*\csring_U).
\]
We can however replace the complex $\bR j_*\csring_U$ by a single sheaf if we consider the \index{real oriented blow-up}\emph{real oriented blowing up} $\varpi:\cswt X(H)\csto X$ of the components of $H$ in $X$, in which polar coordinates are taken with respect to components of $H$: near each point of $H$, in local coordinates of~$X$ adapted to $H$, $\cswt X(H)$ has the corresponding polar coordinates normal to the components of~$H$; for $x$ belonging to exactly $\ell$ components of~$H$, we~have $\varpi^{-1}(x)\simeq(S^1)^\ell$. Denoting by $\cswtj:U\hto\cswt X(H)$ the inclusion, we have
\[
\bR\cswtj_*\csring_U\simeq \csring_{\cswt X(H)},
\]
and thus, setting $\cswt f_X=f_X\circ\varpi$,
\begin{multline*}
\coH^r(f_X^{-1}(\Delta_\rho),f_X^{-1}(\rho\rme^{\sfi[-\pi/2,\pi/2]});\bR j_*\csring_U)\\
\simeq\coH^r(\cswt f_X^{-1}(\Delta_\rho),\cswt f_X^{-1}(\rho\rme^{\sfi[-\pi/2,\pi/2]});\csring_{\cswt X(H)}).
\end{multline*}
\begin{figure}[htb]
\centerline{\includegraphics[scale=.4]{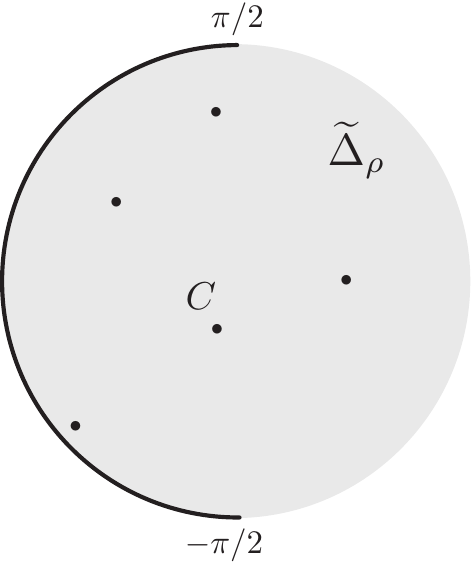}}
\caption{The semi-closed disk $\widetilde\Delta_\rho$}\label{fig:tDeltaR}
\end{figure}

Last, let $\cswt\Delta_\rho$ be the complement in $\Delta_\rho$ of closed interval $[-\pi/2,\pi/2]$ in its boundary like in Figure~\ref{fig:tDeltaR}.
We interpret the latter cohomology as the cohomology with compact support:
\[
\coH^r(\cswt f_X^{-1}(\Delta_\rho),\cswt f_X^{-1}(\rho\rme^{\sfi[-\pi/2,\pi/2]});\csring_{\cswt X(H)})\simeq\coH^r_\rc(\cswt f_X^{-1}(\cswt\Delta_\rho);\csring_{\cswt X(H)}).
\]
On noting the isomorphisms
\[
\bR\cswt f_{X*}\csring_{\cswt X(H)}\simeq\bR\cswt f_{X*}(\bR\cswtj_*\csring_U)\simeq\bR f_*\csring_U,
\]
we interpret this cohomology, denoted from now on as $\coH^r(U,f;\csring)$, as the hypercohomology
\[
\coH^r(U,f;\csring):=\coH^r_\rc(\cswt\Delta_\rho;\bR f_*\csring_U).
\]
It is natural to define the compact support analogue of $\coH^r(U,f;\csring)$ as
\[
\coH^r_\rc(U,f;\csring)=\coH^r_\rc(\cswt\Delta_\rho;\bR f_!\csring_U).
\]
We have thus proved:

\begin{lemma}\label{lem:coHUf}
There are canonical isomorphisms (for $\rho\gg0$)
\[
\coH^r(U,f;\csring)\simeq
\begin{cases}
\coH^r(f^{-1}(\Delta_\rho),f^{-1}(\rho\rme^{\sfi[-\pi/2,\pi/2]});\csring),\\[1pt]\coH^r(f_X^{-1}(\Delta_\rho),f_X^{-1}(\rho\rme^{\sfi[-\pi/2,\pi/2]});\bR j_*\csring_U),\\[3pt]
\coH^r_\rc(\cswt f_X^{-1}(\cswt\Delta_\rho);\csring_{\cswt X(H)}),\\[3pt]
\coH^r_\rc(\cswt\Delta_\rho;\bR f_*\csring_U).
\end{cases}
\]
and
\[
\coH^r_\rc(U,f;\csring)\simeq
\begin{cases}
\coH^r_\rc(f_X^{-1}(\Delta_\rho);\bR j_!\csring_U),\\[3pt]
\coH^r_\rc(\cswt f_X^{-1}(\cswt\Delta_\rho);\cswtj_!\csring_U),\\[3pt]
\coH^r_\rc(\cswt\Delta_\rho;\bR f_!\csring_U).
\end{cases}
\]
\end{lemma}

\subsubsection*{Perverse sheaves}
The \index{sheaf!perverse cohomology --}perverse cohomology sheaves $\pR^{r-d}f_*(\pring_U)$ and $\pR^{r-d}f_!(\pring_U)$ of $\bR f_*\csring_U$ and $\bR f_!\csring_U$ are suitable for expressing $\coH^r(U,f;\csring)$ and $\coH^r_\rc(U,f;\csring)$ because of the next proposition.

\begin{proposition}\label{prop:pervexact}
Let $\cF$ be an \index{sheaf!perverse --}$\csring$-perverse sheaf on $\Af1$ with singular set contained in $C\subset\rond\Delta_\rho$. Then $\coH^k_\rc(\cswt\Delta_\rho;\cF)=0$ for $k\neq0$ and the functor $\cF\mto\coH^0_\rc(\cswt\Delta_\rho;\cF)$ from $\Perv(\Af1;\csring)$ to $\Mod(\csring)$ is exact.
\end{proposition}

The vanishing statement is obtained as a consequence of the following vanishing lemma, which relies on the property already used that, for any semi-closed non-empty interval $[a,b)\subset\RR$ and any abelian group $\csring$, the cohomology $\coH^\cbbullet_\rc([a,b);\csring)$ vanishes (this cohomology is nothing but the relative cohomology $\coH^\cbbullet([a,b],\{b\};\csring)$). This follows \eg from a cohomology analogue of \cite[Lem.\,V.3.3]{Massey67}. The exactness statement follows by considering the hypercohomology long exact sequence associated to a short exact sequence of perverse sheaves.

\begin{lemma}\label{lem:vanishing}
Let $\cF$ be a perverse sheaf on $\Af1$ with singular points in $C$ and let $B\subset\Af1$ be a subset homeomorphic to a product $[a,b)\times[c,d]$, such that $C\cap\partial B=\csemptyset$. Then
\begin{enumerate}
\item
$\coH^k_\rc(B,\cF)=0$ for $k\neq0$,
\item
$\coH^0_\rc(B,\cF)=0$ if $C\cap B=\csemptyset$.\hfill\hbox{\rlap{$\sqcap$}$\sqcup$}
\end{enumerate}
\end{lemma}

Let us use the perverse shift convention: $\pring_U:=\csring_U[d]$ with $d=\dim U$. As~a consequence, applying the previous result to each \index{sheaf!perverse cohomology --}perverse cohomology sheaf $\pR^{r-d}f_*(\pring_U)$ and $\pR^{r-d}f_!(\pring_U)$, and due to the degeneration of the perverse Leray spectral sequence, we obtain:

\begin{corollary}\label{cor:Ufpush}
We have natural isomorphisms
\begin{align*}
\coH^r(U,f;\csring)&\simeq \coH^0_\rc(\cswt\Delta_\rho;\pR^{r-d}f_*(\pring_U)),\\
\coH^r_\rc(U,f;\csring)&\simeq \coH^0_\rc(\cswt\Delta_\rho;\pR^{r-d}f_!(\pring_U)).
\end{align*}
\end{corollary}

\subsubsection*{Real oriented blow-up (2)}
We now give expressions which are independent of $\rho\gg0$, but depend on the choice of a projectivization of $f$. As~these notions depend on the direction of reaching $f=\infty$, we are led to considering a full projectivization $\csov f:\csov X\csto\PP^1$ of $f$, and it is easier for future computations to assume that the complement $D=\csov X\moins U$ is a divisor with normal crossings whose irreducible components are smooth. We decompose $D=H\cup P$, with $P=\csov f{}^{-1}(\infty)$ (the support of the pole divisor of~$f$) and $H$ is the union of the remaining components. Note that $D\cap X=H\cap X$. In such a case, we~say that the full projectivization $\csov f:\csov X\csto\PP^1$ of~$f$ is \emph{good}.

In order to distinguish between directions when reaching the pole divisor~$P$, it~is convenient to consider the \index{real oriented blow-up}\emph{real oriented blow-up} of $\csov X$ along the components of~$P$. However, in order to treat the components $P$ and $H$ of $D$ on an equal footing, and to better understand Poincaré-Verdier duality, it is convenient to work with the real oriented blowing-up $\varpi:\cswt X=\cswt X(D)\csto \csov X$ of \emph{all irreducible components} of~$D$, in which polar coordinates are taken with respect to all components of~$D$: near each point of~$D$, in local coordinates of $\csov X$ adapted to $D$, $\cswt X$ has the corresponding polar coordinates normal to the components of~$D$; for $x$ belonging to exactly $\ell$ components of~$D$, we have $\varpi^{-1}(x)\simeq(S^1)^\ell$. Over $\Af1$, $\cswt X(D)$ restricts to $\cswt X(H\cap X)$ already considered above. This is summarized in the following diagram:
\begin{equation}\label{eq:diagrealblowup}
\begin{array}{c}
\xymatrix@C=1cm@R=.5cm{
&&&&\cswt X\ar@<-.5ex>[lld]_(.3)\varpi\ar[dd]^{\cswt f}|(.25)\hole&\ar@{_{ (}->}[l]\cswt P\ar[lld]\ar[dd]\\
U\ar[dd]_f\ar@{^{ (}->}[r]^-j& X\ar[dd]^{f_X}\ar@{^{ (}->}[r]^-\kappa&\csov X\ar[dd]^{\csov f}&\ar@{_{ (}->}[l]P\ar[dd] \\
&&&&\cswt\PP^1\ar@<-.5ex>[lld]_(.3)\varpi|(.5)\hole&\ar@{_{ (}->}[l]\partial\cswt\PP^1\ar[lld]\\
\Af1\ar@{=}[r]&\Af1\ar@{^{ (}->}[r]&\PP^1&\ar@{_{ (}->}[l]\infty
}
\end{array}
\end{equation}
where
\begin{itemize}
\item
all rectangles are Cartesian except, in general, both involving $(\varpi,\csov f,\cswt f,\varpi)$;
\item
$\cswt\PP^1$ is the real oriented blow-up of $\PP^1$ at infinity; it is homeomorphic to a closed disc with boundary $\partial\cswt\PP^1=S^1$ (directions at infinity);
\item
$\cswt X$ is a manifold with corners and $\varpi^{-1}(D)=:\partial\cswt X$ is its boundary; it~contains $\cswt P=\varpi^{-1}(P)=\cswt f^{-1}(\partial\cswt\PP^1)$ as a closed subset.
\end{itemize}
Let us make explicit the map $\cswt f$ near a point $\cswt x$ of $P$. There exist local coordinates $(x_1,\dots,x_d)$ of $\csov X$ near $x=\varpi(\cswt x)$ in which $D=\{x_1\cdots x_\ell=0\}$ and $f(x_1,\dots,x_d)=x_1^{-m_1}\cdots x_\ell^{-m_\ell}$, $m_j\geqslant0$. The corresponding coordinates are $\bigl((\rho_j,\rme^{\sfi\theta_j})_{j=1,\dots,\ell},x_{\ell+1},\dots,x_d\bigr)$, with $\varpi$ sending $(\rho_j,\rme^{\sfi\theta_j})$ to $x_j=\rho_j\rme^{\sfi\theta_j}$, and
\[
\cswt f((\rho_j,\rme^{\sfi\theta_j})_{j=1,\dots,\ell},x_{\ell+1},\dots,x_d)=\Bigl(\prod_{j=0}^\ell\rho_j^{-m_j}\Bigr)\cdot\exp\bigl(-\textstyle\sum_{j=0}^\ell m_j\theta_j\bigr).
\]

The dichotomy compact support\,/\,no support condition is now replaced with the dichotomy \index{rapid decay}rapid decay\,/\,\index{moderate growth}moderate growth. Let $\partial_\rmod\cswt\PP^1\subset S^1$ be the \emph{open} half-circle in the neighborhood of which which $\reel t$ is $>0$ (\ie $\rme^{-t}$ has moderate growth, equivalently, rapid decay), and let $\partial_{\exp}\cswt\PP^1$ be the complementary \emph{closed} half-circle (\index{exponential growth}exponential growth of $\rme^{-t}$). Let also $\partial_{\exp}\cswt X$ denote the pullback $\cswt f^{-1}(\partial_{\exp}\cswt\PP^1)$, which is a closed subset of $\cswt P$. We then consider the following \emph{open subsets} of the boundary~$\partial\cswt X$:
\begin{itemize}
\item
$\partial_\rmod\cswt X=\partial\cswt X\moins\partial_{\exp}\cswt X$ is the open subset of $\partial\cswt X$ in the neighborhood of which $\rme^{-f}$ has moderate growth (it contains $\varpi^{-1}(D\moins P)$),
\item
$\partial_\rrd\cswt X=\partial\cswt X\moins(\varpi^{-1}(H)\cup\partial_{\exp}\cswt X)$ is the open subset of $\partial\cswt X$ in the neighborhood of which $\rme^{-f}$ has rapid decay; it is contained in $\varpi^{-1}(P)$ and is also equal to $\cswt P\moins\partial_{\exp}\cswt X$; it also consists of points of $\cswt P$ in the neighborhood of which $\rme^{-f}$ has rapid decay, equivalently moderate growth.
\end{itemize}

\subsubsection*{Singular cohomology with growth conditions of the pair \texorpdfstring{$(U,f)$}{Uf}}
We consider the open subsets of $\cswt X$:
\[
\cswt U_\rmod=U\cup\partial_\rmod\cswt X\quand\cswt U_\rrd=U\cup\partial_\rrd\cswt X.
\]

\begin{definition}[{\cf\cite[(A.19)]{F-S-Y20b}}]\label{def:coHUf}
The singular cohomologies in degree $r$ of the pair $(U,f)$ are defined as singular cohomology spaces with compact support:
\begin{align*}
\coH^r(U,f;\csring):=\coH^r_\rc(\cswt U_\rmod;\csring)\quand\coH^r_{\rc}(U,f;\csring):=\coH^r_\rc(\cswt U_\rrd;\csring).
\end{align*}
\end{definition}

Since $\cswt X\moins\cswt U_\rmod=\partial_{\exp}\cswt X$, we can also interpret $\coH^r(U,f;\csring)$ as the relative cohomology space $\coH^r(\cswt X,\partial_{\exp}\cswt X;\csring)$. Similarly, we have
\[
\coH^r_{\rc}(U,f;\csring)\simeq \coH^r(\cswt X,\partial_{\exp}\cswt X\cup\varpi^{-1}(H);\csring).
\]

\begin{proposition}\label{prop:coHUf}
The expressions of $\coH^r(U,f;\csring)$ and $\coH^r_{\rc}(U,f;\csring)$ of Definition \ref{def:coHUf} are respectively naturally isomorphic to those of Lemma \ref{lem:coHUf}.
\end{proposition}

This shows in particular that the expressions of Definition \ref{def:coHUf} do not depend, up to natural isomorphisms, of the projectivization $\csov f$ of $f$. We first note that the analogue of Proposition \ref{prop:pervexact} holds, when setting $\wAfu_\rmod=\cswt\PP^1\moins \partial_{\exp}\cswt\PP^1$ and denoting by~$\alpha$ the open inclusion $\Af1\hto\wAfu_\rmod$:

\begin{corollary}\label{cor:pervexact}
Let $\cF$ be an $\csring$-perverse sheaf on $\Af1$. Then $\coH^k_\rc(\wAfu_\rmod;\cF)=0$ for $k\neq0$ and the functor $\cF\mto\coH^0_\rc(\wAfu_\rmod;\bR\alpha_*\cF)$ from $\Perv(\Af1;\csring)$ to $\Mod(\csring)$ is exact.
\end{corollary}

\begin{proof}
It is done by noticing that, for any constructible complex~$\cF$ on $\Af1$, the inclusion $\cswt\Delta_\rho\subset\wAfu_\rmod$ induces an isomorphism
\[
\coH^k_\rc(\wAfu_\rmod;\bR\alpha_*\cF)\simeq\coH^k_\rc(\cswt\Delta_\rho;\cF|_{\cswt\Delta_\rho})\quad\forall k\in\ZZ.
\]
\end{proof}

\begin{proof}[of Proposition \ref{prop:coHUf}]
We will consider the perverse $(r-d)$-cohomology of the pushforward complex $\bR f_*\csring_U$ and $\bR f_!\csring_U$. At this point, using perverse cohomology sheaves instead of ordinary cohomology sheaves will simplify many arguments.

From Corollary \ref{cor:pervexact} we deduce the expressions
\begin{equation}\label{eq:Ufpush}
\begin{aligned}
\coH^r(U,f;\csring)&\simeq \coH^0_\rc(\wAfu_\rmod;\bR\alpha_*\pR^{r-d}f_*(\pring_U)),\\
\coH^r_\rc(U,f;\csring)&\simeq \coH^0_\rc(\wAfu_\rmod;\bR\alpha_*\pR^{r-d}f_!(\pring_U)),
\end{aligned}
\end{equation}
which can be identified with those of Corollary \ref{cor:Ufpush}, proving thereby Proposition~\ref{prop:coHUf}.
\end{proof}

\subsubsection*{Global vanishing cycles}
We interpret $\coH^r(U,f;\csring)$ and $\coH^r_\rc(U,f;\csring)$ as $\csring$-modules of \index{vanishing cycles!global --}\emph{global vanishing cycles for $f$}. If we interpret $U$ as the global \index{Milnor!ball}Milnor ball and the cohomology of the fiber $f^{-1}(t)$ for $t\gg0$ as the nearby cohomology at $f=\infty$, the exact sequence
\[\let\csto\rightarrow
\cdots\csto\coH^{r-1}(U;\csring)\csto\coH^{r-1}(f^{-1}(t);\csring)\csto\coH^r(U,f^{-1}(t);\csring)\csto\coH^r(U;\csring)\csto\cdots
\]
can be interpreted as coming from the canonical morphism from nearby cycles at infinity to \index{vanishing cycles!at infinity}vanishing cycles at infinity. This interpretation is furthermore justified by the next proposition. Let us enumerate the ordered set $\reel(C)$ of real parts of the elements of $C$ in increasing order as $a_1<a_2<\cdots<a_\ell$.

\begin{proposition}\label{prop:HUfphifiltration}
Both $\coH^r(U,f;\csring)$ and $\coH^r_{\rc}(U,f;\csring)$ have a natural increasing filtration indexed by $\{a_1,\dots,a_\ell\}$ such that, for each $i\in\{1,\dots,\ell\}$, we have
\begin{align*}
\gr_{a_i}\coH^r(U,f;\csring)&\simeq\csbigoplus_{\substack{c\in C\\\reel(c)=a_i}}\bH^{r-1}\bigl(f_X^{-1}(c),\phi_{f_X-c}(\bR j_*\csring_U)\bigr),
\\
\gr_{a_i}\coH^r_\rc(U,f;\csring)&\simeq\csbigoplus_{\substack{c\in C\\\reel(c)=a_i}}\bH^{r-1}\bigl(f_X^{-1}(c),\phi_{f_X-c}(\bR j_!\csring_U)\bigr).
\end{align*}
\end{proposition}

\begin{proof}
Let us start with a perverse sheaf $\cF$ on $\Af1$ with singularities on~$C$. We~will prove that the cohomology space $\bH^0_\rc(\cswt\Delta_\rho;\cF)$ has a natural increasing filtration indexed by $\{a_1,\dots,a_\ell\}$ such that
\begin{equation}\label{eq:graiH0}
\gr_{a_i}\bH^0_\rc(\cswt\Delta_\rho;\cF)\simeq\phip_{c_i}\cF\quad\forall i=1,\dots,\ell.
\end{equation}
For the sake of simplicity, we assume that each $a_i$ is the real part of a unique $c_i\in C$ (otherwise, one replaces the map $\reel:C\csto\RR$ by the map $\reel+\varepsilon\im$ with $\varepsilon>0$ small). Let us decompose $\cswt\Delta_\rho$ as $\cswt\Delta_{\rho,<a_i}\sqcup\cswt\Delta_{\rho,\geqslant a_i}$, denoting thereby the intersection of~$\cswt\Delta_\rho$ with the open half-plane $\reel(t)<a_i$, respectively closed half-plane \hbox{$\reel(t)\geqslant a_i$}. These subsets have the same topology as $\cswt\Delta_\rho$, and we have a long exact sequence of cohomology with compact support
\begin{equation}\label{eq:H0exact}
\cdots\csto\bH^0_\rc(\cswt\Delta_{\rho,<a_i};\cF)\csto\bH^0_\rc(\cswt\Delta_\rho;\cF)\csto\bH^0_\rc(\cswt\Delta_{\rho,\geqslant a_i};\cF)\csto\cdots
\end{equation}
We first claim that $\bH^k_\rc(\cswt\Delta_\rho;\cF)=0$ for $k\neq0$. Indeed, from this exact sequence and by induction on $\#C$, it is enough to prove this when $\#C=1$. In such a case, if $\cswt\Delta_c$ is a small disc centered at $c\in C$ with a closed half-circle $\delta_c$ deleted on its boundary, we have $\bH^k_\rc(\cswt\Delta_c;\cF)\simeq\bH^k_\rc(\cswt\Delta_\rho;\cF)$. The cohomology long exact sequence associated with the open/closed decomposition $\Delta_c=\cswt\Delta_c\sqcup\delta_c$ is then identified with that associated with the shifted distinguished triangle of \cite[(10.87)]{M-S22}
\[
\ppsi_c\cF\csto\phip_c\cF\csto i_c^{-1}\cF[-1]
\]
and one concludes with the property that $\pphi_c\cF$ is concentrated in degree zero since it is perverse on a point $c$ (because the functor $\pphi_c$ is t-exact, \cf Rem.\,10.4.23 in \loccit).

It follows that \eqref{eq:H0exact} is a short exact sequence and thus $\bH^0_\rc(\cswt\Delta_{\rho,<a_{i+1}};\cF)\csto\bH^0_\rc(\cswt\Delta_\rho;\cF)$ is an inclusion, defining a filtration by the formula $\bH^0_\rc(\cswt\Delta_\rho;\cF)_{\leqslant a_i}:=\bH^0_\rc(\cswt\Delta_{\rho,<a_{i+1}};\cF)$. The similar exact sequence with $\cswt\Delta_\rho$ replaced with $\cswt\Delta_{\rho,<a_{i+1}}$ shows that $\gr_{a_i}\bH^0_\rc(\cswt\Delta_\rho;\cF)$ is identified with $\bH^0_\rc(\cswt\Delta_{\rho, [a_i,a_{i+1})};\cF)$. The latter space is isomorphic to $\bH^0_\rc(\cswt\Delta_{c_i};\cF)$, that we have already identified with $\phip_{c_i}\cF$. This concludes the proof of \eqref{eq:graiH0}.

Let us now come back to the statement of the proposition. The formulas of Corollary \ref{cor:Ufpush} lead us to consider the perverse cohomology sheaves $\cF=\pR^{r-d}f_*(\pring_U)$ and $\cF=\pR^{r-d}f_!(\pring_U)$, and to apply the previous result to them. However, we are faced to the problem of making the functor of \index{vanishing cycles!functor}vanishing cycles commute with pushforward by $f$ when $f$ is not proper. This question is solved by Proposition \ref{prop:pushphi}.
\end{proof}

\subsubsection*{The Poincaré pairing}
We now fix $\csring=\QQ$. In order to define the \index{pairing!Poincaré --}Poincaré pairing between the $\QQ$-vector spaces $\coH^r(U,f;\QQ)$ and $\coH^r_\rc(U,f;\QQ)$, we make precise that $\cswt U_\rmod$ is relative to~$f$, so~we denote it by $\cswt U_\rmod(f)$, and similarly $\cswt U_\rrd(f)$. Considering both~$f$ and $-f$, and thus growth properties of~$\rme^{-f}$ and~$\rme^f$, one checks that
\[
\cswt U_\rmod(f)\cap \cswt U_\rrd(-f)=U.
\]
Assume that $U$ is connected for simplicity. It follows that we have natural pairings for $r\in\{0,\dots,2d\}$:
\begin{equation}\label{eq:dualHUf}
\coH^r(U,f;\csring)\otimes \coH^{2d-r}_{\rc}(U,-f;\csring)\csto\coH^{2d}_\rc(U;\csring)\simeq \csring.
\end{equation}

\begin{proposition}[Poincaré-Verdier duality]\label{prop:PDUf}
If $\csring$ is a field, \eg $\csring=\QQ$, then the above pairings are nondegenerate.
\end{proposition}

\begin{proof}[Sketch]
One can give two proofs: one by computing on $\cswt X$ and the other one by computing on $\cswt\PP^1$ and using the commutation of Verdier duality with proper pushforward. We will sketch the first one. The point is to compute the dualizing complex on the manifold with corners $\cswt X$: one finds that it is the extension by zero to~$\cswt X$ of the shifted constant sheaf $\QQ_U[2d]$. Then the computation of the Verdier dual of the shifted sheaf $\pQQ_{\cswt U_\rmod(f)}$ (recall Notation \ref{nota}\eqref{nota2}) is seen to be isomorphic to $\pQQ_{\cswt U_\rrd(-f)}$. The result is obtained by applying Verdier duality to the cohomology of these sheaves.
\end{proof}

\subsubsection*{Weight filtration}
As a prelude to Hodge theory, let us consider the \index{filtration!weight --}weight filtration on the cohomology spaces $\coH^r(U,f;\QQ)$ and $\coH^r_\rc(U,f;\QQ)$. When $f=0$, the $\QQ$-vector spaces $\coH^r(U;\QQ)$ and $\coH^r_\rc(U;\QQ)$ come naturally equip\-ped with an increasing filtration, called the \emph{weight filtration} (\cf \cite{DeligneHII}).

Given a $W$-filtered vector space $(H,W_\bbullet)$, we say that $(H,W_\bbullet)$ has weights $\geqslant w$, \resp $\leqslant w$, if $\gr^W_\ell H=0$ for $\ell<w$, \resp $\ell>w$. Furthermore, we define a filtration on the dual vector space $H^\csvee$ by $W_\ell(H^\csvee)=(W_{<-\ell}H)^\csperp$ (with ${<}-\ell=-\ell-1$), so that
\[
\gr_\ell^W(H^\csvee)\simeq(\gr_{-\ell}^WH)^\csvee.
\]
Therefore, $(H,W_\bbullet)$ has weights $\geqslant w$ if and only if its dual $(H^\csvee,W_\bbullet)$ has weights $\leqslant -w$. We also define the \emph{Tate twist} by $n\in\ZZ$ by the formula
\[
(H,W_\cbbullet)(n):=(H,W_{\cbbullet-2n}).
\]
In particular, it is known that $\coH^r(U;\QQ)$ has weights $\geqslant r$ and $\coH^r_\rc(U;\QQ)$ has weights $\leqslant r$.

One can also define a weight filtration on the $\QQ$-vector spaces $\coH^r(U,f;\QQ)$ and $\coH^r_\rc(U,f;\QQ)$ as follows (although without referring to mixed Hodge structures).\footnote{This weight filtration should be thought of as an analogue of the weight filtration of the mixed Hodge structure on the relative cohomology $\coH^r(U,f^{-1}(t);\QQ)$ with $t$ \emph{fixed}; it should not be confused with the (relative) monodromy weight filtration at the limit $t\csto\infty$.} For that purpose, we make use of the theory of mixed Hodge modules (\cf \cite{MSaito87}, and also \cite{Steenbrink22}), which endows the perverse complexes $\pR^{r-d}f_*(\pQQ_U)$ and $\pR^{r-d}f_!(\pQQ_U)$ with an increasing weight filtration $W_\bbullet$ in the abelian category of perverse complexes, the former having weights $\geqslant r$ and the latter weights $\leqslant r$ (in a sense similar to that for vector spaces). One then define, by using the expressions \eqref{eq:Ufpush},
\begin{multline*}
W_\ell\coH^r(U,f;\QQ)=\image\Bigl[\coH^0_\rc(\wAfu_\rmod,\bR\alpha_*W_\ell\pR^{r-d}f_*(\pQQ_U))\\
\csto\coH^0_\rc(\wAfu_\rmod,\bR\alpha_*\pR^{r-d}f_*(\pQQ_U))\Bigr]
\end{multline*}
and similarly for $\coH^r_\rc(U,f;\QQ)$ by replacing $f_*$ with $f_!$. This expression simplifies if we notice that the functor which associates to each perverse sheaf $\cF$ on $\Af1$ the $\QQ$\nobreakdash-vector space $\coH^0_\rc(\wAfu_\rmod,\bR\alpha_*\cF)$ is \emph{exact}. In view of Proposition \ref{prop:HUfphifiltration}, this is similar to the exactness of the functor $\phip_{t-c}$. It follows that, for each $\ell\in\ZZ$, we have
\begin{align*}
W_\ell\coH^r(U,f;\QQ)&\simeq\coH^0_\rc(\wAfu_\rmod,\bR\alpha_*W_\ell\pR^{r-d}f_*(\pQQ_U)),\\
\gr^W_\ell\coH^r(U,f;\QQ)&\simeq\coH^0_\rc(\wAfu_\rmod,\bR\alpha_*\gr^W_\ell\pR^{r-d}f_*(\pQQ_U)),
\end{align*}
and similarly for $\coH^r_\rc(U,f;\QQ)$ by replacing $f_*$ with $f_!$.

\begin{corollary}\label{cor:weights}
The vector spaces $\coH^r(U,f;\QQ)$ and $\coH^r_\rc(U,f;\QQ)$ have respective weights $\geqslant r$ and $\leqslant r$, and the vector space (middle cohomology)
\[
\coH^d_\mathrm{mid}(U,f;\QQ):=\image\Bigl[\coH^d_\rc(U,f;\QQ)\csto\coH^d(U,f;\QQ)\Bigr]
\]
is pure of weight $d$.
\end{corollary}

An ambiguity could occur in the second statement: which filtration do we put on $\coH^d_\mathrm{mid}(U,f;\QQ)$? Is it $\image(W_\bbullet\coH^d_\rc(U,f;\QQ))$ or $W_\bbullet\coH^d(U,f;\QQ))\cap\coH^d_\mathrm{mid}(U,f;\QQ)$? Fortunately, both coincide: firstly, it follows from the theory of mixed Hodge modules that, for any morphism $\cF\csto\cG$ between $W$-filtered perverse sheaves underlying a morphism of mixed Hodge modules, the equality $\image(W_\bbullet\cF)=W_\bbullet\cG\cap\image(\cF)$ holds; secondly, by exactness of the functor $\cF\mto\coH^0_\rc(\wAfu_\rmod,\bR\alpha_*\cF)$ for $\cF$ perverse, the previous equality passes through this functor.

Duality is also compatible with the \index{filtration!weight --}weight filtration:

\begin{theorem}
The Poincaré-Verdier duality pairing of Proposition \ref{prop:PDUf} induces an isomorphism of $W$-filtered vector spaces:
\[
(\coH^r_{\rc}(U,f;\QQ),W_\cbbullet)\simeq(\coH^{2d-r}(U,-f;\QQ),W_\cbbullet)^\csvee(d),
\]
so that, for each $\ell\in\ZZ$, we have a nondegenerate pairing
\[
\gr^W_\ell\coH^r_{\rc}(U,f;\QQ)\otimes\gr_{2d-\ell}^W\coH^{2d-r}(U,-f;\QQ)\csto\QQ.
\]
\end{theorem}

\begin{proof}[Sketch]
The second proof of Proposition \ref{prop:PDUf} would yield that, for a perverse sheaf $\cF$ on $\Af1$ with Verdier dual $\cF^\csvee$, the natural pairing
\[
\coH^0_\rc(\wAfu_\rmod(t);\bR\alpha_*\cF)\otimes\coH^0_\rc(\wAfu_\rmod(-t);\bR\alpha_*\cF^\csvee)\csto\CC
\]
is nondegenerate, where $\alpha$ denotes any of the open inclusions $\Af1\hto\wAfu_\rmod(t)$ and $\Af1\hto\wAfu_\rmod(-t)$. The theory of mixed Hodge modules expresses the weight filtration $W_\bbullet(\cF^\csvee)$ in terms of $W_\bbullet\cF$, and this leads to the formulas of the theorem.
\end{proof}

\subsection{Singular homology with growth conditions}
We consider the real blow-up space $\cswt X$ with its open subsets $\cswt U_\rrd\subset \cswt U_\rmod$ (\cf Diagram \eqref{eq:diagrealblowup}). Recall that, over $\Af1$, we have
\[
\cswt U_\rrd|_{\cswt f^{-1}(\Af1)}=U,\quad \cswt U_\rmod|_{\cswt f^{-1}(\Af1)}=\cswt X|_{\cswt f^{-1}(\Af1)},
\]
while $\cswt U_\rrd\cap \cswt P=\cswt U_\rmod\cap \cswt P$. The boundary of these open subsets are
\[
\partial\cswt U_\rrd=\cswt U_\rrd\moins U=\cswt U_\rrd\cap\partial\cswt X,\quad\partial\cswt U_\rmod=\cswt U_\rmod\moins U=\cswt U_\rmod\cap\partial\cswt X.
\]
The analogues of the homology $\coH_r(U;\csring)$ and the homology with closed support (Borel-Moore) $\coH_r^\BM(U;\csring)$ are the following relative homology groups (\cf \cite[App.]{F-S-Y20b}):
\[
\coH_r(U,f;\csring):=\coH_r(\cswt U_\rmod,\partial\cswt U_\rmod;\csring),\quad \coH_r^\BM(U,f;\csring):=\coH_r(\cswt U_\rrd,\partial\cswt U_\rrd;\csring).
\]

\begin{proposition}[Intersection and period pairings, {\cf\cite[\S2]{F-S-Y20a}}]
Assume that~$U$ is connected. For each $r=0,\dots,2d$, the \index{pairing!intersection --}intersection pairing
\[
\coH_r(U,f;\QQ)\otimes \coH_{2d-r}^\BM(U,-f;\QQ)\csto\coH_0(U;\QQ)\simeq\QQ
\]
is nondegenerate. Moreover, the \index{pairing!period --}period pairing
\[
\coH^r(U,f;\QQ)\otimes \coH_r(U,f;\QQ)\csto\QQ,\quad \coH^r_\rc(U,f;\QQ)\otimes \coH_r^\BM(U,f;\QQ)\csto\QQ
\]
is nondegenerate and relates the intersection pairing with the Poincaré-Verdier pairing.
\end{proposition}

\subsection{Algebraic de~Rham cohomology}\label{subsec:algDR}

\subsubsection*{Algebraic de~Rham cohomology of the pair \texorpdfstring{$(U,f)$}{Uf}}
We now consider $U$ with its Zariski topology. The algebraic de~Rham cohomology $\coH^r_\dR(U)$ of the variety $U$ is by definition the hypercohomology on $U$ of the \index{complex!algebraic de~Rham --}algebraic de~Rham complex $(\Omega_U^\cbbullet,\rd)$, whose terms are the sheaves of algebraic differential forms on $U$. In a similar way, we introduce the \index{cohomology!algebraic de~Rham --}algebraic de~Rham cohomology $\coH^r_\dR(U,f)$ of the pair $(U,f)$ as being the hypercohomology of the \index{complex!twisted algebraic de~Rham --}\emph{twisted algebraic de~Rham complex}
\[
(\Omega_U^\cbbullet,\rd+\rd f).
\]
(The differential $\rd+\rd f$ can be regarded as the twisted differential \hbox{$\rme^{-f}\circ\rd\circ\rme^f$}, although this expression is only meaningful in the analytic context.)

When $f=0$, it is helpful to consider a good projectivization $\csov X$ of $U$ and to introduce the \index{complex!logarithmic de~Rham --}logarithmic de~Rham complex $(\Omega_{\csov X}^\cbbullet(\log D),\rd)$, where $\Omega^k_{\csov X}(\log D)$ is the sheaf of logarithmic differential $k$\nobreakdash-forms (these are the rational $k$-forms~$\omega$ on~$\csov X$ with poles along $D$ at most such that~$\omega$ and $\rd\omega$ have at most simple poles along $D$, \cf\cite{DeligneHII}). Then, for each integer $r$, the cohomology space $\coH^r_\dR(U)$ can be computed as the $r$\nobreakdash-th hypercohomology space of the complex $(\Omega_{\csov X}^\cbbullet(\log D),\rd)$. One advantage is that each term $\Omega^k_{\csov X}(\log D)$ is a locally free sheaf of finite rank on the smooth projective variety~$\csov X$.

Pursuing the analogy with $\coH^r_\dR(U)$, we introduce the \index{complex!Kontsevich --}\emph{Kontsevich complex} of $(U,f)$ (\cf \cite{E-S-Y13}, \cite{K-K-P14}, and the references therein): this is the complex\footnote{For the sake of simplicity, we simply denote by $f$ the morphism denoted above by $\csov f$. The notation $\Omega_{\csov X}^\cbbullet(\log D,f)$ is used in \cite{K-K-P14}.}
\begin{equation}\label{eq:Kontcompl}
(\Omega_f^\cbbullet,\rd+\rd f),
\end{equation}
where $\Omega_f^k$ is the subsheaf of $\Omega_{\csov X}^k(\log D)$ consisting of logarithmic $k$\nobreakdash-forms~$\omega$ such that $\rd f\wedge\omega$ is still a logarithmic $(k+1)$-form. In other words:
\begin{equation}\label{eq:Omegaf}
\Omega_f^k=\ker\Bigl[\rd f:\Omega_{\csov X}^k(\log D)\csto\Omega_{\csov X}^{k+1}(*D)/\Omega_{\csov X}^{k+1}(\log D)\Bigr].
\end{equation}
Away from $P$, this subsheaf coincides with $\Omega_{\csov X}^k(\log D)$. On the other hand, along~$P$, it is a strictly smaller subsheaf: writing $\rd f$ as the product $f\cdot \rd f/f$, we note that $(\rd f/f)\wedge\omega$ is also logarithmic, but the poles of $f$ possibly introduce higher order poles of $\rd f\wedge\omega$, so~that the condition required for $\omega$ is strong. However, one can compute the sheaves~$\Omega_f^k$ in local analytic coordinates where $f$ is written as the inverse of a monomial, and show that they are locally $\cO_{\csov X}$-free (\cf \cite{Yu12}). For example, we~have $\Omega_f^0=\cO_{\csov X}(-\bP)$, where~$\bP$ is the pole divisor (with multiplicities) of $f$.

In the algebraic setting, the notion of \index{cohomology!de~Rham -- with compact support}de~Rham cohomology with compact support $\coH^r_{\dR,\rc}(U)$ is also defined (\cf\cite{Hartshorne72}) and it also has an expression in terms of a logarithmic complex as the hypercohomology of the complex
\[
(\Omega_{\csov X}^\cbbullet(\log D)(-D),\rd).
\]
The termwise wedge product
\[
\Omega_{\csov X}^k(\log D)\otimes\Omega_{\csov X}^\ell(\log D)(-D)\csto\Omega_{\csov X}^{k+\ell}(\log D)(-D)\hto\Omega_{\csov X}^{k+\ell}
\]
is compatible with differentials and, for $\ell\geqslant0$, identifies $\Omega_{\csov X}^\ell(\log D)(-D)$ with the Serre dual $\omega_{\csov X}\otimes(\Omega_{\csov X}^{d-\ell}(\log D))^\csvee$. It provides thus, for each $r$, by passing to complexes and their hypercohomologies, a nondegenerate pairing (\index{duality!de~Rham --}\index{pairing!de~Rham --}de~Rham pairing)
\[
\coH^r_{\dR}(U)\otimes_\CC\coH^{2d-r}_{\dR,\rc}(U)\csto\coH^{2d}_{\dR}(\csov X)\simeq\CC.
\]

In an analogous way (\cf\cite{Yu12}), the termwise wedge product
\[
\Omega_f^k\otimes\Omega_{-f}^\ell(-D)\csto\Omega_{\csov X}^{k+\ell}(\log D)(-D)\hto\Omega_{\csov X}^{k+\ell}
\]
is compatible with the differentials $\rd+\rd f,\rd-\rd f$ and, for $\ell\geqslant0$, identifies the Serre dual $\Omega_{-f}^\ell(-D)$ with $\omega_{\csov X}\otimes(\Omega_f^{d-\ell})^\csvee$. It provides thus, for each $r$, by passing to complexes and their hypercohomologies, a nondegenerate pairing (de~Rham pairing)
\[
\coH^r\bigl(\csov X,(\Omega_f^\cbbullet,\rd+\rd f)\bigr)\otimes_\CC\coH^{2d-r}\bigl(\csov X,(\Omega_{-f}^\cbbullet(-D),\rd-\rd f)\bigr)\csto\CC.
\]
We can interpret this pairing according to the following result (\cf\cite{E-S-Y13}).

\begin{theorem}\mbox{}\label{th:kontsevichyu}
\begin{enumerate}
\item\label{th:kontsevichyu1}
The restriction to $U$ induces isomorphisms for all $r\in\NN$:
\begin{align*}
\coH^r_\dR(U,f)&\simeq \coH^r\bigl(\csov X,(\Omega_f^\cbbullet,\rd+\rd f)\bigr),
\\
\coH^r_{\dR,\rc}(U,f)&\simeq \coH^r\bigl(\csov X,(\Omega_f^\cbbullet(-D),\rd+\rd f)\bigr).
\end{align*}
\item\label{th:kontsevichyu2}
For each $r\in\NN$, the natural pairing that one deduces from these identifications is nondegenerate:
\[
\coH^r_\dR(U,f)\otimes_\CC\coH^{2d-r}_{\dR,\rc}(U,-f)\csto \coH^{2d}_{\dR}(\csov X)\simeq\CC.
\]
\end{enumerate}
\end{theorem}

\begin{remark}\label{rem:kontsevichyu}
At this step, one can avoid the use of the Kontsevich complex and simply consider the usual logarithmic de~Rham complex on $X=\csov X\moins P$ with respect to the divisor $H=D\moins P$. One can replace in the statement of Theorem~\ref{th:kontsevichyu} the Kontsevich complexes with
\[
(\Omega_X^\cbbullet(\log H),\rd+\rd f_X)\quand(\Omega_X^\cbbullet(\log H)(-H),\rd+\rd f_X)
\]
respectively (\cf\cite[Cor.\,1.4.3]{E-S-Y13} and \cite[Lem.\,2.8]{Bibi22a}). However, the role of the pole divisor will be emphasized when considering the irregular Hodge filtration.
\end{remark}

Grothendieck's comparison isomorphisms
\[
\coH^r_\dR(U)\simeq\coH^r(U^\an;\QQ)\otimes\CC,\quad\coH^r_{\dR,\rc}(U)\simeq\coH^r_\rc(U^\an;\QQ)\otimes\CC
\]
can be extended in the following way (much details are given in \cite[\S2]{F-S-Y20a}):

\begin{theorem}\label{th:isoBdR}
We have natural isomorphisms
\[
\coH^r_\dR(U,f)\simeq\coH^r(U^\an,f;\QQ)\otimes\CC,\quad \coH^r_{\dR,\rc}(U,f)\simeq\coH^r_\rc(U^\an,f;\QQ)\otimes\CC
\]
which transform the de~Rham pairing into the Poincaré-Verdier duality pairing.
\end{theorem}

\begin{remark}
One can regard Theorem \ref{th:isoBdR}, together with Proposition \ref{prop:HUfphifiltration}, as giving an algebraic formula for the total dimension of the $(r-1)$-st hypercohomology of the complexes of vanishing cycles $\phi_{f_X-c}(\bR j_*\QQ_U)$ and $\phi_{f_X-c}(\bR j_!\QQ_U)$ on~$X$ when~$c$ varies in $\Af1$. Is it possible to obtain an algebraic formula when replacing these complexes with the complex $\phi_{f-c}(\QQ_U)$? This would amount to replacing any of the former complexes with their image by the functor $\bR j_*j^{-1}$. We will find such an expression in Theorem \ref{th:lpru}.
\end{remark}

\subsection{Irregular mixed Hodge theory}\label{subsec:IrrMHS}

\subsubsection*{The irregular Hodge filtration}
Continuing the analogy between the cohomologies $\coH^r(U)$ and $\coH^r(U,f)$, let us recall, after \cite{DeligneHII}, that the filtration by ``stupid'' truncation (``filtration bête'') of the logarithmic de~Rham complex $(\Omega_{\csov X}^\cbbullet(\log D),\rd)$ is the filtration
\[\let\csto\rightarrow
F^p(\Omega_{\csov X}^\cbbullet(\log D),\rd)=\bigl\{0\csto\cdots\csto0\csto\Omega_{\csov X}^p(\log D)\csto\cdots\csto\Omega_{\csov X}^d(\log D)\csto0\bigr\},
\]
and, for each $p$ and $r$, the natural morphism
\[
\coH^r\bigl(\csov X,F^p(\Omega_{\csov X}^\cbbullet(\log D),\rd)\bigr)\csto\coH_\dR^r(U)
\]
is injective, with image defining the decreasing filtration $F^\cbbullet\coH^r_\dR(U)$. This is the \index{Hodge!filtration}\index{filtration!Hodge --}Hodge filtration of a \index{Hodge!mixed -- structure}mixed Hodge structure on $\coH^r(U^\an,\QQ)$, and we have for each~\hbox{$r\!\geqslant\!0$} a decomposition
\[
\coH^r_\dR(U)\simeq\csbigoplus_{p\geqslant0}\gr^p_F\coH^r_\dR(U)\simeq\csbigoplus_{p+q=r}\coH^q(X,\Omega_{\csov X}^p(\log D)).
\]
Similar properties hold for the cohomology with compact support
\[
\coH_{\dR,\rc}^r(U)=\coH^r\bigl(\csov X,(\Omega_{\csov X}^\cbbullet(\log D)(-D),\rd)\bigr).
\]

It is therefore tempting to consider the filtration by stupid truncation on the Kontsevich complex \eqref{eq:Kontcompl}. Before doing so, let us note a new phenomenon that appears in the context of a pair $(U,f)$. Indeed, the pole divisor of $f$ contains information that has not been exploited. Recall that $\bP$ denotes the (non-reduced) divisor $\csov f{}^*(\infty)$ with support~$P$. For any $a\in\QQ$, we~can consider the integral part $\lfloor a\bP\rfloor$, that is, if $\bP=\sum_im_iP_i$ with $P_i$ reduced, we set $\lfloor a\bP\rfloor=\sum_i\lfloor a m_i\rfloor P_i$. The family of divisors $\lfloor a\bP\rfloor$ with $a\in\QQ$ is increasing, and there exists a finite set $\cA$ of rational numbers in $[0,1)$ such that the jumps occur at most for $a\in\cA+\ZZ$ (since the jumps occur at most when the denominator of~$a$ divides some $m_i$). Multiplication by~$f$ sends $\cO_{\csov X}(\lfloor a\bP\rfloor)$ to $\cO_{\csov X}(\lfloor(a+1)\bP\rfloor)$. On noting that $\rd f=f\cdot\rd f/f$ and that both~$\rd$ and~$\rd f/f$ preserve logarithmic poles along $D$, we can consider the Kontsevich-Yu complex $(\Omega_f^\cbbullet(\alpha),\rd+\rd f)$, for each $\alpha\in\cA$, with a definition similar to \eqref{eq:Omegaf}, that is,
\begin{equation}\label{eq:Omegalogalpha}
\Omega_f^k(\alpha)
=\ker\Bigl[\rd f:\Omega_{\csov X}^k(\log D)(\lfloor\alpha\bP\rfloor)\rightarrow\Omega_{\csov X}^{k+1}(*D)/\Omega_{\csov X}^{k+1}(\log D)(\lfloor\alpha\bP\rfloor)\Bigr].
\end{equation}

The properties previously recalled for the logarithmic de~Rham complex extend to the Kontsevich complex. We define the filtration $F^p(\Omega_f^\cbbullet(\alpha),\rd+\rd f)$ as the filtration by stupid truncation.

\enlargethispage{-2\baselineskip}%
\pagebreak[2]
\begin{theorem}[\cite{E-S-Y13}]\label{th:E1degalpha}\mbox{}
\begin{enumerate}
\item
For each $\alpha\in\cA$, the inclusion
\[
(\Omega_f^\cbbullet,\rd+\rd f)\hto(\Omega_f^\cbbullet(\alpha),\rd+\rd f)
\]
is a quasi-isomorphism, leading to an identification
\[
\coH^r\bigl(\csov X,(\Omega_f^\cbbullet,\rd+\rd f)\bigr)=\coH^r\bigl(\csov X,(\Omega_f^\cbbullet(\alpha),\rd+\rd f)\bigr).
\]

\item
Furthermore, for each $\alpha\in\cA$, $p\geqslant0$ and $r\geqslant0$, the natural morphism
\[
\coH^r\bigl(\csov X,F^p(\Omega_f^\cbbullet(\alpha),\rd+\rd f)\bigr)\csto\coH_\dR^r(U,f)
\]
is injective, with image defining the decreasing filtration $F_{\irr,\alpha}^\cbbullet\coH^r_\dR(U,f)$. For each~$r$, we~have a decomposition
\[
\coH^r_\dR(U,f)\simeq\csbigoplus_{p\geqslant0}\gr^p_{F_{\irr,\alpha}}\coH^r_\dR(U,f)\simeq\csbigoplus_{p+q=r}\coH^q(X,\Omega_f^p(\alpha)).
\]

\item
A similar result holds for the Kontsevich-Yu complex ``with compact support'' $(\Omega_f^\cbbullet(\alpha)(-D),\rd+\rd f)$.
\end{enumerate}
\end{theorem}

Let us emphasize that, for each $\alpha\in\cA$, the only interesting exponents of $F_{\irr,\alpha}^\cbbullet$ belong to $\{0,\dots,d\}$ in the sense that $\gr^p_{F_{\irr,\alpha}}=0$ for $p$ not in this set.

\begin{remark}[Filtration indexed by $\QQ$]\label{rem:FirrQ}
The \index{filtration!indexed by $\QQ$}filtrations $F_{\irr,\alpha}^\cbbullet\coH^r_\dR(U,f)$ and $F_{\irr,\alpha}^\cbbullet\coH^r_{\dR,\rc}(U,f)$ also have an expression in terms of the \index{complex!twisted meromorphic de~Rham --}\emph{twisted meromorphic de~Rham complex} (\cf\cite{Yu12}), which makes clear that these filtrations increase with~$\alpha$, so that we can regard all of them as forming a decreasing filtration $F_{\irr}^\sfp\coH^r_\dR(U,f)$ indexed by $\sfp=p-\alpha\in-\cA+\ZZ\subset\QQ$ by setting
\[
F_{\irr}^\sfp\coH^r_\dR(U,f)=F_{\irr}^{p-\alpha}\coH^r_\dR(U,f):=F_{\irr,\alpha}^p\coH^r_\dR(U,f)
\]
for each $\alpha\in\cA$ and $p\in\ZZ$, and similarly with compact support. This filtration is called the \index{filtration!irregular Hodge --}\emph{irregular Hodge filtration} of $\coH^r_\dR(U,f)$, \resp $\coH^r_{\dR,\rc}(U,f)$. More precisely, J.-D.\,Yu \cite{Yu12} (\cf also \cite[Cor.\,1.4.5]{E-S-Y13}) considers the complex
\[
0\csto\cO_{\csov X}\To{\rd+\rd f}\Omega^1_{\csov X}(\log D)(\bP)\csto\cdots\csto\Omega^d_{\csov X}(\log D)(d\bP)\csto0
\]
filtered by the subcomplexes $F_\mathrm{Yu}^\sfp$ with $\sfp\in\QQ_{\geqslant0}$:
\begin{multline*}
0\csto\cO_{\csov X}(\lfloor(-\sfp)_+\bP\rfloor)\csto\Omega^1_{\csov X}(\log D)(\lfloor(1-\sfp)_+\bP\rfloor)\csto\cdots\\
\csto\Omega^d_{\csov X}(\log D)(\lfloor(d-\sfp)_+\bP\rfloor)\csto0,
\end{multline*}
where $(a)_+:=\max(a,0)$ and, for $a\in\QQ$, $\lfloor a\bP\rfloor$ has been defined above. J.\nobreakdash-D.\,Yu shows that $\coH^r_\dR(U,f)$ is the hypercohomology of the latter complex and that $F_\irr^\sfp\coH^r_\dR(U,f)$ is the filtration induced by the subcomplex $F_\mathrm{Yu}^\sfp$. As a consequence, we have $\gr^\sfp_{F_\irr}\coH^r_\dR(U,f)=0$ for $\sfp<0$ (the expression in terms of the Kontsevich complex only yields vanishing for $\sfp\leqslant-1$).
\end{remark}

Such a theorem was already envisioned by Deligne in 1984 (\cf \cite{Deligne8406}) in the case where $\dim U=1$.

As in the case of the Hodge filtration on $\coH^r_\dR(U)$ and $\coH^r_{\dR,\rc}(U)$, the irregular Hodge filtration behaves well with respect to the nondegenerate pairing of Theorem \ref{th:kontsevichyu}\eqref{th:kontsevichyu2}. If $(H,F^\cbbullet)$ is a filtered finite dimensional $\CC$-vector space (with a filtration indexed by $\sfp\in -\cA+\ZZ\subset\QQ$), we denote by $\gr^\sfp_FH$ the quotient space $F^\sfp H/F^{>\sfp}H$, where ``${>}\sfp$'' is the successor of $\sfp$ in $-\cA+\ZZ$.
We define a filtration on the dual space~$H^\csvee$ by setting $F^\sfp(H^\csvee)=(F^{>-\sfp}H)^\csperp$ (orthogonality taken with respect to the tautological pairing $H\otimes H^\csvee\csto\CC$). It follows that $\gr^\sfp_F(H^\csvee)\simeq(\gr^{-\sfp}_FH)^\csvee$. On~the other hand, we define the \emph{Tate twist}, for $n\in\ZZ$, by the formula
\[
(H,F^\cbbullet)(n):=(H,F^{\cbbullet-n}),
\]
so that, for each $\sfp$, we have $\gr^\sfp_F(H(n))=\gr^{\sfp-n}_F(H)$.

\begin{theorem}[{\cite[Th.\,2.2]{Yu12}}]\label{th:yuduality}
The \index{pairing!de~Rham --}de~Rham duality pairing of Theorem \ref{th:kontsevichyu}\eqref{th:kontsevichyu2} induces an isomorphism of filtered vector spaces (indexed by $-\cA+\ZZ\subset\QQ$):
\[
(\coH^r_{\dR,\rc}(U,f),F^\cbbullet_\irr)\simeq(\coH^{2d-r}_{\dR}(U,-f),F^\cbbullet_\irr)^\csvee(d),
\]
so that, for each $\sfp\in -\cA+\ZZ$, we have a nondegenerate pairing
\[
\gr^\sfp_{F_\irr}\coH^r_\dR(U,f)\otimes\gr^{d-\sfp}_{F_\irr}\coH^{2d-r}_{\dR,\rc}(U,-f)\csto\CC.
\]
\end{theorem}

From Remark \ref{rem:FirrQ} we deduce that $\gr^\sfp_{F_\irr}\coH^r_\dR(U,f)=0$ for $\sfp\notin[0,d]\cap\QQ$, and a similar property for $\gr^\sfp_{F_\irr}\coH^r_{\dR,\rc}(U,f)$.

\subsubsection*{Some examples}
The \index{Thom-Sebastiani}\emph{Thom-Sebastiani property} allows for simple computations: set for example $U=\Gm^d$ and $f=x_1+\cdots+x_d$; then $\coH^r_\dR(\Gm^d,f)=0$ for $r\neq d$ and $\dim \coH^d_\dR(\Gm^d,f)=1$ with irregular Hodge filtration jumping at $d$ only (\cf \cite[(A.23)]{F-S-Y18}).

On the other hand, it can happen that the jumping indices of the irregular Hodge filtration of $\coH^r_\dR(U,f)$ are integers and, even more, that $\coH^r_\dR(U,f)$ is isomorphic to the cohomology of some algebraic variety in such a way that the irregular Hodge filtration of $\coH^r_\dR(U,f)$ corresponds to that of the canonical mixed Hodge structure on the latter cohomology. In such a case, the computation of the irregular Hodge numbers can be easier than a direct computation of the Hodge numbers of this mixed Hodge structure.

For example, if $U=\Af1_t\times V$ and $f=t\cdot g$, where $g:V\csto\Af1$ is a regular function, then $\coH^r_{\dR,\rc}(U,f)\simeq\coH^{r-2}_{\dR,\rc}(g^{-1}(0))$ and the irregular Hodge filtration of the left-hand side corresponds to the Hodge filtration of the mixed Hodge structure on the right-hand side Tate-twisted by $-1$ (\cf\cite[Ex.\,A.27]{F-S-Y18}).

If $g$ is defined by means of \emph{Thom-Sebastiani sums}, the computation can be simplified. Let us mention one remarkable computation obtained in this way by Y.\,Qin \cite{Qin23}, extending that made in \cite{F-S-Y18}.

Let $k$ be an integer $\geqslant1$ and let us consider the case of the $k$-fold Thom-Sebastiani sum $g_k:V^k\csto\Af1$ of a regular function $g:V\csto\Af1$---\ie for \hbox{$(x_1,\dots,x_k)\in V$}, $g_k(x_1,\dots,x_k)=\sum_{i=1}^kg(x_i)$---and the corresponding hypersurface
\[
\cH_k(g)=\{g_k=0\}\subset V^k.
\]
The symmetric group $\gS_k$ acts on $V^k$ and preserves $\cH_k(g)$, and we are interested in computing the Hodge numbers of the isotypical component $\coH^r_\dR(\cH_k(g),g_k)^\sigma$ with respect to the signature character $\sigma$ of $\gS_k$.

Y.\,Qin considered, as in \cite{F-S-Y18}, the case of the function
\[
\begin{aligned}
g:\Gm^d&\csto\Af1\\[-3pt]
(y_1,\dots,y_d)&\mto y_1+\cdots+y_d+\frac{1}{y_1\cdots y_d}.
\end{aligned}
\]
In Arithmetic, the function $g$ is at the source of the \emph{generalized Kloosterman sums}, and in Mirror symmetry it plays the role of the \emph{mirror of $\PP^d$}. There is a supplementary action of the group $\mu_2=\{\pm1\}$ on $\cH_k(g)$ induced by $(y_1,\dots,y_d)\mto\pm(y_1,\dots,y_d)$, and we are interested in the invariant part with respect to this action. We denote by $\chi:\gS_k\times\mu_2\csto\{\pm1\}\subset\CC^*$ the corresponding character $\sigma\cdot\id$. Y.\,Qin has given a closed formula for the \index{Hodge!numbers}Hodge numbers of the $(dk+1)$-pure part of $\coH^{dk-1}_\rc\bigl(\cH_k(g)\bigr)^\chi(-1)$, which extends one obtained in \cite{F-S-Y18} when $d=1$, both proved by means of computation of \index{Hodge!irregular -- numbers}irregular Hodge numbers.

\begin{theorem}[Y.\,Qin \cite{Qin23}]\mbox{}
Assume that $\gcd(k,d+1)=1$. Then the Hodge numbers $h^{p,dk+1-p}$ of the \index{Hodge!Tate-twisted -- structure}Tate-twisted Hodge structure $\gr^W_{dk+1}\bigl[\coH^{dk-1}_\rc\bigl(\cH_k(g)\bigr)^\chi(-1)\bigr]$ are, for $p\leqslant dk/2$, the coefficients of $t^px^k$ in the power series expansion of
\[
\frac{(1-t)t^{d+1}}{(1-t^{d+1})(1-x)(1-tx)\cdots(1-t^dx)},
\]
and for $p> dk/2$, they are obtained by duality $h^{p,dk+1-p}=h^{dk+1-p,p}$.
\end{theorem}

\subsubsection*{The weight filtration}
On the de~Rham cohomology spaces $\coH^r_{\dR}(U)$, the \index{filtration!weight --}weight filtration is induced by a filtration of the de~Rham complexes: for a logarithmic form (of some degree), the weight is the number of components of $D$ along which the form as a pole. On the other hand, the weight filtration on $\coH^r_{\dR}(U,f)$ cannot be defined as simply as in the case where $f=0$ since the pole set $P$ disturbs the notion of pole of a logarithmic form. We~define it in a different way, that can also be used in the case where $f=0$. As in the topological setting, we make use of the theory of mixed Hodge modules \cite{MSaito87} and we compute $\coH^r_{\dR}(U,f)$ and $\coH^r_{\dR,\rc}(U,f)$ by means of the de~Rham complex of the regular holonomic $\cD_{\Af1}$-modules $M$ corresponding to the perverse complexes $\pR^{r-d}f_*(\pQQ_U)$ and $\pR^{r-d}f_!(\pQQ_U)$ respectively. Given such a module $M$, that we regard as a quasi-coherent $\cO_U$-module with a connection $\nabla$, one shows that the twisted de~Rham complex
\[
M\To{\nabla+\rd t}\Omega^1_{\Af1}\otimes M
\]
has nonzero cohomology in degree $1$ only, giving rise to an exact functor
\[
M\mto\coH^1_{\dR}\bigl(\Af1,(\Omega_{\Af1}^\cbbullet\otimes M,\nabla+\rd t)\bigr)
\]
from regular holonomic $\cD_{\Af1}$-modules to $\CC$-vector spaces. When $M$ underlies a mixed Hodge module, it is endowed with a weight filtration $W_\bbullet M$ by regular holonomic submodules, and we set
\[
W_\bbullet\coH^1_{\dR}\bigl(\Af1,(\Omega_{\Af1}^\cbbullet\otimes M,\rd+\rd t)\bigr)=\coH^1_{\dR}\bigl(\Af1,(\Omega_{\Af1}^\cbbullet\otimes W_\bbullet M,\rd+\rd t)\bigr),
\]
in a way similar to the topological case. In this way, we obtain a weight filtration for $\coH^r_{\dR}(U,f)$ and $\coH^r_{\dR,\rc}(U,f)$. \emph{Let us emphasize that, in general, contrary to the example above, the pair $(F_\irr^\cbbullet,W_\bbullet)$ does not form a mixed Hodge structure of $\coH^r(U,f)$ or $\coH^r_{\dR,\rc}(U,f)$.}

\section{Monodromy properties of a pair \texorpdfstring{$(U,f)$}{Uf}}

While the purpose of Section \ref{sec:cohUf} was to extend as much as possible cohomological properties of $U$ (\ie pairs, $(U,0)$) to such properties for any pair $(U,f)$, we~consider in this section properties that do not exist (\ie are trivial) when $f\equiv0$, and which are related to monodromy. A non trivial monodromy operator may occur around each critical fiber, that is, a fiber of $f$ above a typical or atypical critical value, \ie any point of the bifurcation set $B(f)$ (\cf Theorem \ref{th:Bf}), and it also may occur around the ``fiber at infinity''. However, mimicking now the local properties of critical points, we aim at defining a \emph{global vanishing space with monodromy} that takes into account all critical values at the same time, including~$\infty$, so that it is not just the direct sum over all (typical or atypical) critical values of $f$ at finite distance. Furthermore, when considering complex coefficients, we will produce an algebraic formula for these vanishing cycles, similar in spirit to the expression discovered by Brieskorn in the case of an isolated singularity of hypersurface, in terms of the \emph{Gauss-Manin system} and the \emph{Brieskorn lattice}. This algebraic formula is obtained via the consideration of the \emph{twisted de~Rham complex with parameter}.

\subsection{Monodromy on global vanishing cycles attached to \texorpdfstring{$(U,f)$}{Uf}}\label{sec:quiverperverse}

From the initial description of the \index{vanishing cycles!global --}global vanishing cycle space $\coH^r(U,f;\csring)=\coH^r(U,f^{-1}(\rho);\csring)$ for $\rho\gg0$ in Section \ref{subsec:singcoh}, one obtains that this space carries a monodromy operator produced by rotating the fiber $f^{-1}(\rho)$, that is, by considering the \index{sheaf!locally constant --}locally constant sheaf on the circle $S^1$ with stalk at $\rme^{\sfi\theta}$ equal to $\coH^r(U,f^{-1}(\rho\rme^{\sfi\theta});\csring)$ (recall the choice \ref{choice:Crho}). This description does neither make clear the relation with the Poincaré-Verdier duality pairing of Proposition \ref{prop:PDUf} nor with the supplementary structure that is produced on this locally constant sheaf by the critical values of $f$. We will thus give a family of isomorphic descriptions of this space by rotating $f$ and considering the rotation angle as a new parameter of the description.

For any $\theta\in S^1$, let us consider the pair $(U,\rme^{-\sfi\theta}f)$ and its global vanishing cycle spaces. Recalling that $\cswt\Delta_\rho$ is pictured in Figure \ref{fig:tDeltaR}, let us rotate it by setting
\begin{equation}\label{eq:Deltainf}
\cswt\Delta_\rho^{\leqslant\infty}:=\csbigsqcup_\theta\rme^{\sfi\theta}\cswt\Delta_\rho\subset S^1\times\Delta_\rho,
\end{equation}
with projections $p:\cswt\Delta_\rho^{\leqslant\infty}\csto S^1$ and $q:\cswt\Delta_\rho^{\leqslant\infty}\csto\Delta_\rho$. Due to the choices \ref{choice:Crho}, one easily proves:

\begin{lemma}
There exists a semi-analytic Whitney stratification $\cswt\cS_\rho$ of $\cswt Z_\rho:=(\id\times\cswt f)^{-1}(\cswt\Delta_\rho^{\leqslant\infty})\subset S^1\times X$ such that
\begin{enumerate}
\item
$Z_\rho:=(\id\times f)^{-1}(\cswt\Delta_\rho^{\leqslant\infty})\subset S^1\times U$ is a union of strata,
\item
and for each stratum $\cswt S$ of $\cswt\cS_\rho$, the map $p\circ(\id\times\cswt f)|_{\cswt S}:\cswt S\csto S^1$ is a submersion.
\end{enumerate}
\end{lemma}

Since the morphism $(\cswt Z_\rho,\cswt\cS)\csto (S^1,\cT)$ is stratified with $\cT$ having the only stratum $S^1$, and due to the $\cT$-constructibility of the complexes $\bR f_*\csring_{\cswt Z_\rho}$, $\bR f_!\csring_{\cswt Z_\rho}$, $\bR f_*\csring_{Z_\rho}$, $\bR f_!\csring_{Z_\rho}$ (\cf\eg\cite[Th.\,10.2.6]{M-S22}), that is, the local constancy of their cohomologies, it follows that the spaces $\coH^r(\cswt f^{-1}(\rme^{\sfi\theta}\cswt\Delta_\rho);\csring)\simeq\coH^r(U,\rme^{-\sfi\theta}f;\csring)$, \resp $\coH^r_\rc(f^{-1}(\rme^{\sfi\theta}\cswt\Delta_\rho);\csring)\simeq\coH^r_\rc(U,\rme^{-\sfi\theta}f;\csring)$, glue as a local system on~$S^1$ when $\theta$ varies. If~we make use of the correspondence between local systems of $\csring$\nobreakdash-modules of finite type on~$S^1$ and pairs consisting of an $\csring$\nobreakdash-module of finite type together with an automorphism, we are entitled to consider the previous local systems as the ``\index{vanishing cycles!at infinity}vanishing cycle spaces of $(U,f)$ at infinity''.

\begin{definition}\label{def:globalvanishingcycles}
The \index{vanishing cycles!global --}\emph{local system of global vanishing cycles} of the pair $(U,f)$ are the local systems $\cswt\Phi_\infty^r(U,f)$ and $\cswt\Phi_{\infty,\rc}^r(U,f)$ with respective fibers $\coH^r(U,\rme^{-\sfi\theta}f;\csring)$ and $\coH^r_\rc(U,\rme^{-\sfi\theta}f;\csring)$ at $\theta\in S^1$.
\end{definition}

\begin{remark}\label{rem:Phiperv}\mbox{}
\begin{enumerate}
\item\label{rem:Phiperv1}
To any perverse complex $\cF$ on $\Af1$ with singular set $C$, one can associate similarly a local system $\cswt\Phi_\infty(\cF):=p_!(q^{-1}\cF)$ on $S^1$ whose fiber at $\theta$ is $\coH^0_\rc(\rme^{\sfi\theta}\cswt\Delta_\rho;\cF)$. Note that, applying the commutativity of $\bR p_!$ with restriction to $\theta$ and Proposition \ref{prop:pervexact}, the higher direct images $R^kp_!(q^{-1}\cF)$ vanish for $k\geqslant1$. One recovers $\cswt\Phi_\infty^r(U,f)$, \resp $\cswt\Phi_{\infty,\rc}^r(U,f)$, by setting $\cF=\pR^{r-d}f_*\csring_U$, \resp $\cF=\pR^{r-d}f_!\csring_U$.
\item\label{rem:Phiperv2}
By replacing $\Delta_\rho$ by a small closed disc centered at $c\in C$, one defines similarly the local system $\cswt\pphi_c\cF$.
\end{enumerate}
\end{remark}

These local systems can be nontrivial: when expressing these spaces in terms of perverse sheaves $\cF$ as in Corollary \ref{cor:Ufpush}, what we are doing is to rotate the picture in Figure \ref{fig:tDeltaR} with respect to its center, and the support of a cohomology class $\coH^0_\rc(\cswt\Delta_\rho;\cF)$ could intersect the open interval $(\pi/2,3\pi/2)$ at the boundary $\cswt\Delta_\rho$, so that it does not remain with compact support in $\rme^{\sfi\theta}\cswt\Delta_\rho$ for some~$\theta$.

The question that remains is to relate these global vanishing cycle spaces with monodromy to the local ones with monodromy, as in Proposition \ref{prop:HUfphifiltration}. The new phenomenon that occurs is that the terms of the filtration considered in Proposition~\ref{prop:HUfphifiltration} do not glue as local systems on $S^1$. Let us explain why. When $a\in\RR$ is fixed and~$\theta$ varies, some elements of $\rme^{-\sfi\theta}C$ may enter or leave the subset \hbox{$\cswt\Delta_\rho\cap\{\reel(t)<a\}$} so that the dimension of the cohomology space
\[
\coH^0_\rc(\cswt\Delta_\rho\cap\{\reel(t)<a\};\cF_\theta),
\]
with $\cF_\theta=\pR^{r-d}(\rme^{-\sfi\theta}f)_*(\pring_U)$ or $\pR^{r-d}(\rme^{-\sfi\theta}f)_!(\pring_U)$, changes at $\theta_o$.

We need to adapt the notion of filtration in order to take into account such a phenomenon, which is very similar to that occurring in the asymptotic theory of differential equations near an irregular singular point, as explained \eg in \cite{Wasow65} and~\cite{Malgrange91}. We develop in the Section \ref{subsec:Stokesfilt} the general framework of such \emph{Stokes-filtered local systems}, that we will first introduce in the case of a pair $(U,f)$. Let us emphasize that, instead of indexing the filtration by $\RR$ with its natural order, we~index the filtration by $C$ that we equip with a partial order depending on $\theta$.

\subsection{The Stokes filtration attached to a pair \texorpdfstring{$(U,f)$}{UF}}\label{sec:constStokesfilt}
Let $(U,f)$ be as in Section~\ref{sec:cohUf}. We recognize in the expressions of Corollary \ref{cor:Ufpush} for $\coH^r(U,f)$ and $\coH^r_\rc(U,f)$ the fiber at $\theta=\pi$ of $\cswt\Phi_\infty(\cF)$ with the perverse sheaf $\cF$ being respectively $\pR^{r-d}f_*(\pring_U)$ and $\pR^{r-d}f_!(\pring_U)$ (\cf Remark \ref{rem:Phiperv}). We will define the Stokes filtration on $\cswt\Phi_\infty^r(U,f)$ and $\cswt\Phi_{\infty,\rc}^r(U,f)$ by means of such an identification. We thus treat the general case of an $\csring$-perverse sheaf $\cF$ on $\Af1$ with critical set $C$. For the application to the pair $(U,f)$, one can use expressions similar to those of Lemma \ref{lem:coHUf}, but for proving the Stokes properties, it is easier to work on $\Af1$ with a perverse sheaf~$\cF$, as we did in Proposition \ref{prop:HUfphifiltration}.

To a perverse sheaf $\cF$ of $\csring$-modules on $\Af1$ with singular set $C$ we will attach a \index{filtration!indexed by $C$}family indexed by $c\in C$ of nested subsheaves\footnote{These subsheaves are \emph{not} local systems.} $\cswt\Phi_\infty(\cF)_{<c}\subset \cswt\Phi_\infty(\cF)_{\leqslant c}$ of the local system $\cswt\Phi_\infty(\cF)$. We intend that this family of subsheaves satisfies the properties defining a \index{filtration!Stokes --}\index{Stokes!filtration}\emph{Stokes filtration} that will be analyzed with more details in the next section.

Let us first define a partial order on $C$ depending on $\theta\in S^1$ by the formula
\begin{equation}\label{eq:partialorder}
c'\leqslant_\theta c\quad\text{if}\quad c'=c\text{ or } c'\neq c\text{ and }\arg(c'-c)\in\theta+(\pi/2,3\pi/2)\bmod2\pi,
\end{equation}
and $c'<_\theta c$ means $c'\!\neq\!c$ and $c'\leqslant_\theta c$. For example, if $\theta\!=\!0$, $c'<_0 c$ means $\reel(c'-c)\!<\!0$.

The subsheaves $\cswt\Phi_\infty(\cF)_{<c}\subset \cswt\Phi_\infty(\cF)_{\leqslant c}$ should satisfy the following properties:
\begin{enumerate}
\item\label{def:PhiStokes1}
for each $\theta\in S^1$, the germs of $\cswt\Phi_\infty(\cF)_{<c},\cswt\Phi_\infty(\cF)_{\leqslant c}$ at $\theta$ satisfy
\[
c'<_\theta c\implies\cswt\Phi_\infty(\cF)_{\leqslant c',\theta}\subset\cswt\Phi_\infty(\cF)_{<c,\theta};
\]
\item\label{def:PhiStokes2}
each quotient sheaf $\gr_c\cswt\Phi_\infty(\cF):=\cswt\Phi_\infty(\cF)_{\leqslant c}/\cswt\Phi_\infty(\cF)_{<c}$ is a \index{sheaf!locally constant --}\emph{locally constant sheaf} of $\csring$-modules;
\item\label{def:PhiStokes3}
for any $\theta$ and any $c\in C$, there exists an isomorphism (possibly depending on~$\theta$) of germs at $\theta$: $\cswt\Phi_\infty(\cF)_{\leqslant c,\theta}\simeq\csbigoplus_{c'\leqslant_\theta c}\gr_{c'}\cswt\Phi_\infty(\cF)_\theta$, in a way compatible with the inclusions in Item~\ref{def:PhiStokes1}.
\end{enumerate}

These nested subsheaves are obtained from the following geometric construction. For each $\theta\in S^1$, we set $\cswt\Delta_{\rho,\theta}=\rme^{\sfi\theta}\cswt\Delta_\rho$, and for each $c\in C$, we define two nested open subsets $\cswt\Delta_{\rho,\theta}^{<c}\subset\cswt\Delta_{\rho,\theta}^{\leqslant c}$ of $\cswt\Delta_\rho$ as pictured in Figure \ref{fig:4}:
\begin{itemize}
\item
$\cswt\Delta_{\rho,\theta}^{<c}$ is the intersection of $\cswt\Delta_{\rho,\theta}$ with the \emph{open half-plane} having the line passing through $c$ and of direction $\theta\pm\pi/2$ as boundary, and which contains the point of argument $\theta+\pi$ on $\partial\cswt\Delta_{\rho,\theta}$;
\item
$\cswt\Delta_{\rho,\theta}^{\leqslant c}$ is the union of $\cswt\Delta_{\rho,\theta}^{<c}$ and $\cswt\Delta_{c,\theta}$.
\end{itemize}
Of course, this definition can be made for any $c$ in the interior of $\Delta_\rho$, but only $c\in C$ will matter. From now on, we denote by $\cswt\Delta_\rho^{<c}$ and $\cswt\Delta_\rho^{\leqslant c}$ the union over $\theta\in S^1$ of these subsets. We have a diagram
\[
\xymatrix{
\cswt\Delta_\rho^{<c}\subset\cswt\Delta_\rho^{\leqslant c}\subset\hspace*{-.93cm}&S^1\times\Delta_\rho\ar[r]^-{q}\ar@<-.5ex>[d]_p&\Delta_\rho\\
&S^1&
}
\]

\begin{figure}[htb]
\centerline{\includegraphics[scale=.4]{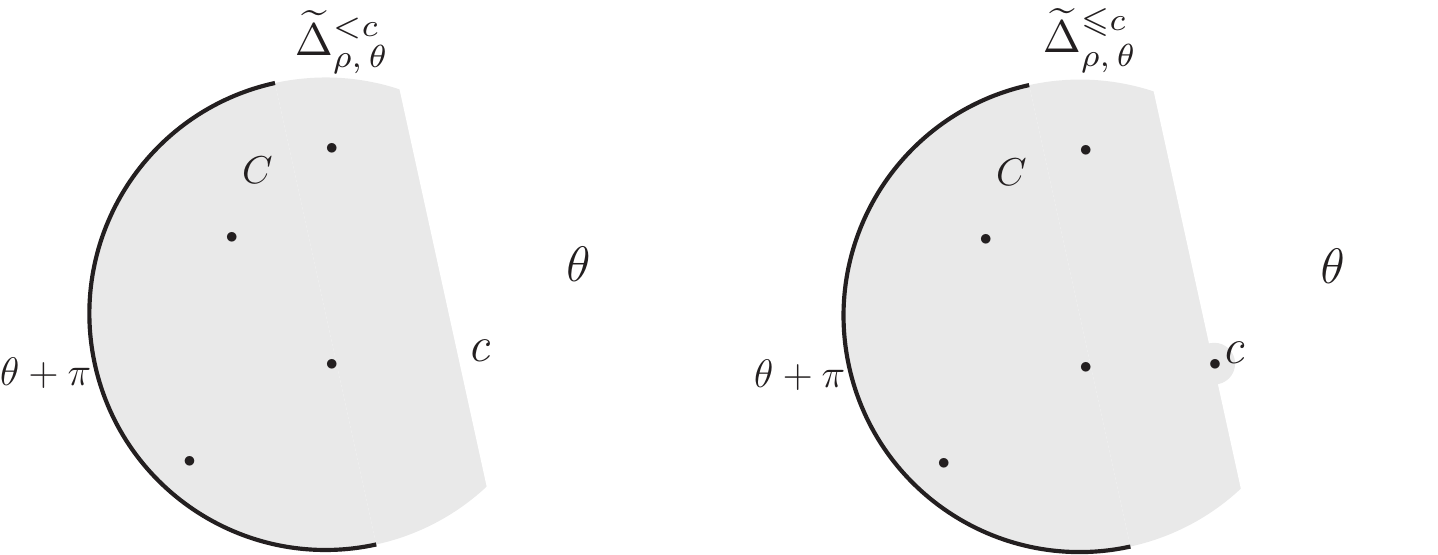}}
\caption{The nested open subsets of the semi-closed disk}\label{fig:4}
\end{figure}

We set $\cswt\cF=q^{-1}\cF$. Correspondingly, we consider the pair of nested subsheaves of $\cswt\Phi_\infty(\cF)$:
\[
\cswt\Phi_\infty(\cF)_{<c}=p_!(\cswt\cF_{\cswt\Delta_\rho^{<c}})\subset\cswt\Phi_\infty(\cF)_{\leqslant c}=p_!(\cswt\cF_{\cswt\Delta_\rho^{\leqslant c}})\subset\cswt\Phi_\infty(\cF)=p_!(\cswt\cF_{\cswt\Delta_\rho^{\leqslant \infty}}),
\]
where we recall that $\cswt\cF_{\cswt\Delta_\rho^{<c}}$ denotes the extension by zero of the sheaf-theoretic restric\-tion $\cswt\cF|_{\cswt\Delta_\rho^{<c}}$ (if $c\notin C$, we have $\cswt\Phi_\infty(\cF)_{<c}=\cswt\Phi_\infty(\cF)_{\leqslant c}$). We note that (by~using base change under a proper morphism) the derived pushforwards
\[
R^jp_!(\cswt\cF_{\cswt\Delta_\rho^{<c}})\quand R^jp_!(\cswt\cF_{\cswt\Delta_\rho^{\leqslant c}})
\]
are zero for $j\neq0$ (\cf the proof of Proposition \ref{def:Stokes-fillocsys}). We also note that
\[
\cswt\Phi_\infty(\cF)_{<c}=p_!(\csring_{\cswt\Delta_\rho^{<c}}\otimes\cswt\Phi_\infty(\cF)),\quad \cswt\Phi_\infty(\cF)_{\leqslant c}=p_!(\csring_{\cswt\Delta_\rho^{\leqslant c}}\otimes\cswt\Phi_\infty(\cF)).
\]

\begin{proposition}\label{prop:StokesF}
The family of pairs of nested subsheaves
\[
(\cswt\Phi_\infty(\cF)_{<c},\cswt\Phi_\infty(\cF)_{\leqslant c})_{c\in C}
\]
of the local system $\cswt\Phi_\infty(\cF)$ satisfy the properties \ref{def:PhiStokes1}--\ref{def:PhiStokes3} of a Stokes filtration and moreover we have an identification $\gr_c\cswt\Phi_\infty(\cF)\simeq\cswt\pphi_c\cF$ for each $c\in C$.
\end{proposition}

(\Cf Remark \ref{rem:Phiperv}\eqref{rem:Phiperv2} for the definition of $\cswt\pphi_c\cF$.)

\begin{proof}
We start with Item \ref{def:PhiStokes2} of the definition of a Stokes filtration, and we compute the sheaf $\gr_c\cswt\Phi_\infty(\cF)$. We set $\cswt\Delta'_c=\cswt\Delta_\rho^{\leqslant c}\moins \cswt\Delta_\rho^{<c}$. (\Cf Figure \ref{fig:3}.)
\begin{figure}[htb]
\centerline{\includegraphics[scale=.4]{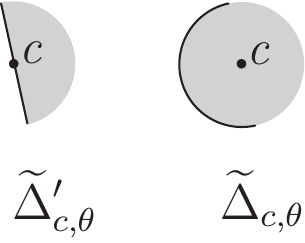}}
\caption{Stalks of $\cswt\Delta'_c$ and $\cswt\Delta_c$ at $\theta$}\label{fig:3}
\end{figure}

From the exact sequence of sheaves
\[
0\csto \csring_{\Delta_{\rho,<c}}\csto \csring_{\Delta_{\rho,\leqslant c}}\csto \csring_{\cswt\Delta'_c}\csto0
\]
we deduce that $\gr_c\cswt\Phi_\infty(\cF)\simeq p_!(\cswt\cF_{\cswt\Delta'_c})$. An easy computation on the stalks together with the vanishing lemma \ref{lem:vanishing} shows that $p_!(\csring_{\cswt\Delta_c\moins\cswt\Delta'_c}\otimes\cswt\cF)=0$, leading to the isomorphism $\gr_c\cswt\Phi_\infty(\cF)\simeq\cswt\pphi_c\cF$.

For Item \ref{def:PhiStokes1}, one notices that, for $c'\neq c$, the inequality $c'<_\theta c$ is equivalent to the inclusion $\Delta_{\rho,\leqslant c',\theta}\subset \Delta_{\rho,<c,\theta}$. Since proper pushforward commutes with base change, the exact sequence
\begin{equation}\label{eq:ZS}
0\csto \csring_{\Delta_{\rho,\leqslant c'}}\csto \csring_{\Delta_{\rho,<c}}\csto \csring_{\Delta_{\rho,<c}\moins\Delta_{\rho,\leqslant c'}}\csto0
\end{equation}
together with Lemma \ref{lem:vanishing}, leads to the exact sequence (with an obvious notation)
\begin{equation}\label{eq:ZSPhi}
0\csto\cswt\Phi_\infty(\cF)_{\leqslant c',\theta}\csto\cswt\Phi_\infty(\cF)_{<c,\theta}\csto\cswt\Phi_\infty(\cF)_{c'<\cbbullet<c,\theta}\csto0,
\end{equation}
hence the filtration property.

Verification of Item \ref{def:PhiStokes3} follows by induction on $c$ with respect to the order at $\theta$.
\end{proof}

\begin{remark}[Stokes filtration of a strongly perverse sheaf]\label{rem:sperv}
Assume that $\cF$ is \index{complex!strongly perverse --}strongly perverse (see Section~\ref{subsec:cohomtools}). Then each $\cswt\phip_c\cF$ is a local system of \emph{free} $\csring$-modules of finite rank by Lemma \ref{lem:sperv}, and thus so is $\cswt\Phi_\infty(\cF)$. Furthermore, $\cswt\Phi_\infty(\cF)_{<c}$ and $\cswt\Phi_\infty(\cF)_{\leqslant c}$ are torsion-free.
\end{remark}

\subsection{The notion of a Stokes-filtered local system}\label{subsec:Stokesfilt}

We develop the general properties of a Stokes filtration, that we can apply to that attached to $\coH^r(U,f/u;\csring)$ and $\coH^r_\rc(U,f/u;\csring)$.

The notion of a Stokes filtration as considered here has been introduced by Deligne \cite{Deligne78}; in \cite[Chap.\,XII]{Malgrange91}, Malgrange explains it in the framework of differential equations. See also \cite{Bibi10} for a more expanded version.

\subsubsection*{Stokes filtration}
This \index{filtration!Stokes --}\index{Stokes!filtration}section takes up \cite[\S\S2\,\&\,3]{H-S09}, with the difference that we work over the ring~$\csring$ and we do not make use of complex conjugation, which has to be read here as the identity in \loccit

We fix a finite set $C\subset\Af1$. Recall the partial order \eqref{eq:partialorder} on $C$. For each pair $c'\neq c\in C$, there are exactly two values of $\theta\bmod2\pi$, say $\theta_{c,c'}$ and $\theta'_{c,c'}$, such that~$c$ and $c'$ are not comparable at~$\theta$, namely $\theta_{c,c'}=\arg(c-c')-\pi/2$ and $\theta'_{c,c'}=\theta_{c,c'}+\pi$. These values are called the \index{Stokes!arguments}\emph{Stokes arguments} (or~directions) of the pair $(c,c')$. For any~$\theta$ in one component of $S^1\moins\{\theta_{c,c'},\theta'_{c,c'}\}$, we~have $c<_\theta c'$, and the reverse inequality for any $\theta$ in the other component. We denote these open intervals in $S^1$ by $S^1_{c<c'}$ and $S^1_{c'<c}$ respectively.

Let $\scL$ be a sheaf of $\csring$-modules on~$S^1$ and, for each $c\in C$, let $\scL_{<c}\subset\scL_{\leqslant c}\subset\scL$ be of a pair of nested subsheaves of $\csring$-modules. We set $\gr_c\scL=\scL_{\leqslant c}/\scL_{<c}$ and $\gr\scL=\csbigoplus_{c\in C}\gr_c\scL$. We say that $\gr\scL$ underlies a \index{Stokes!graded -- -filtered local system}\emph{graded Stokes-filtered local system} if each $\gr_c\scL$ is a \index{sheaf!locally constant --}locally constant sheaf of $\csring$-modules. The graded Stokes filtration on $\gr\scL$ is the family of pairs $(\gr\scL)_{<c}\subset(\gr\scL)_{\leqslant c}\subset\gr\scL$ of nested subsheaves of $\csring$-modules defined by (\cf Notation \ref{nota}\eqref{nota2})
\begin{equation}\label{eq:grnested}
(\gr\scL)_{<c}=\csbigoplus_{c'\in C}(\gr_{c'}\scL)_{S^1_{c'<c}},\quad (\gr\scL)_{\leqslant c}=(\gr\scL)_{<c}\oplus\gr_c\scL.
\end{equation}

\begin{definition}[Stokes filtration and Stokes-filtered local system]\label{def:Stokes-fillocsys}
A \index{filtration!Stokes --}\index{Stokes!filtration}Stokes filtration indexed by $C$ on a sheaf $\scL$ of $\csring$-modules consists of a family of pairs $\scL_{<c}\subset\scL_{\leqslant c}\subset\scL$ ($c\in C$) of nested subsheaves of $\csring$-modules such that
\begin{enumerate}
\item\label{def:Stokes1}
for each $\theta\in S^1$, the germs of $\scL_{<c},\scL_{\leqslant c}$ at $\theta$ satisfy
\[
c<_\theta c'\implies\scL_{\leqslant c,\theta}\subset\scL_{<c',\theta};
\]
\item\label{def:Stokes2}
each quotient sheaf $\gr_c\scL=\scL_{\leqslant c}/\scL_{<c}$ is a \index{sheaf!locally constant --}\emph{locally constant sheaf} of $\csring$\nobreakdash-modules;
\item\label{def:Stokes3}
in the neighborhood of each $\theta\in S^1$, there exists an isomorphism $\scL\simeq\gr\scL$ compatible with the family of nested subsheaves, where $\gr\scL$ is equipped with that defined by \eqref{eq:grnested}.
\end{enumerate}
\end{definition}

We denote these data as a pair $(\scL,\scL_\bbullet)$, which is called a \index{Stokes!-filtered local system}\index{filtration!indexed by $C$}\emph{Stokes-filtered local system indexed by $C$}, and the family $(\scL_{<c},\scL_{\leqslant c})_{c\in C}$ a \index{filtration!Stokes --}\emph{Stokes filtration} of $\scL$. Indeed, as a consequence of Items \ref{def:Stokes2} and \ref{def:Stokes3}, $\scL$ is a \index{sheaf!locally constant --}\emph{locally constant sheaf} of $\csring$-modules.

If $C_1\supset C$ is a finite subset containing $C$, one can naturally extend the Stokes filtration $\scL_\bbullet$ indexed by $C$ to one indexed by $C_1$ in such a way that, for any $c\in C_1\moins C$, the local system $\gr_c\scL$ is zero.

A morphism $\lambda:(\scL,\scL_\bbullet)\csto(\scL',\scL'_\bbullet)$ of Stokes-filtered local systems is a morphism of local systems compatible with the families of nested subsheaves (due to the remark above, one can assume that both Stokes filtrations are indexed by the same set $C$).

Let us state the main properties of Stokes-filtered local systems that we use. By~a \emph{$C$-good open interval} $I\subset S^1$, we mean an open interval containing \emph{exactly one} Stokes argument for each pair $c\neq c'$ in $C$. Such an interval is thus of length $>\pi$. As an example of a $C$-good open interval we can take the image in $S^1$ of an interval $(\theta_o-\varepsilon, \theta_o+\pi+\varepsilon)$, for $\theta_o$ arbitrary and $\varepsilon>0$ small enough, and one can enlarge it by pushing on the left and on the right the boundary points to the next Stokes arguments.

\begin{proposition}[{\cf\cite[\S5]{Malgrange83bb} and \cite[Chap.\,3]{Bibi10}}]\label{prop:strict}
Let $(\scL,\scL_\bbullet)$ be a Stokes-filtered local system indexed by $C$.
\begin{enumerate}
\item\label{prop:strict1}
On any $C$-good open interval $I\subset S^1$, there exists a unique splitting $\scL_{|I}\simeq\csbigoplus_c\gr_c\scL_{|I}$ compatible with the Stokes filtrations.
\item\label{prop:strict2}
Let $\lambda:(\scL,\scL_\bbullet)\csto(\scL',\scL'_\bbullet)$ be a morphism of Stokes-filtered local systems indexed by $C$. Then, for any $C$-good open interval $I\subset S^1$, the morphism~$\lambda_{|I}$ is graded with respect to the splittings in \eqref{prop:strict1}.
\item\label{prop:strict3}
The category of Stokes-filtered local systems $(\scL,\scL_\bbullet)$ is abelian.
\end{enumerate}
\end{proposition}

\begin{proof}[Sketch]
Let us start with Item~\ref{prop:strict1}. Let $I\neq S^1$ be any strict open interval in $S^1$ and assume that a splitting of $\scL|_I$ exists, \ie a morphism $\lambda:\csbigoplus_c\gr_c\scL|_I\isom\scL|_I$ compatible with the Stokes filtrations and whose associated graded morphism is the identity. Let us set $C=\{c_1,\dots,c_n\}$ and let $\lambda_i:\gr_{c_i}\scL|_I\csto\scL_{\leqslant c_i}|_I$ be the restriction of $\lambda$ to $\gr_{c_i}\scL|_I$. Another lifting can be obtained as follows. For $j\in\{1,\dots,n\}$ we denote $j<_Ii$ if $c_j<_\theta c_i$ for any $\theta\in I$. For any $i,j$ with $j<_Ii$, let $\psi_{ji}:\gr_{c_i}\scL|_I\csto\gr_{c_j}\scL|_I$ be any morphism of (constant) local systems of $\csring$\nobreakdash-modules. Then $\cswt\lambda$ with restriction for each $i$ given by $\cswt\lambda_i=\lambda_i+\sum_{j<_Ii}\lambda_j\circ\psi_{ji}$ defines another splitting on~$I$, and any other splitting on~$I$ is obtained like this. The condition $j<_Ii$ implies that there is no Stokes argument for the pair $(c_i,c_j)$ in $I$. It~follows that, if~$I$ contains one Stokes argument for each pair $(c_i,c_j)$, a splitting on~$I$, if it exists, is unique.

We now show that a splitting exists on any interval $I$ which satisfies:
\par\noindent$(*)$ $I$ contains at most one Stokes argument for each pair $c\neq c'\in C$.
\par\noindent
The proof is by induction on the number of such Stokes arguments in $I$. One first notes that if a splitting exists on~$I$, it can be extended to any $J\supset I$ provided that~$J$ contains the same Stokes argument as~$I$ does. Furthermore, the case of one Stokes argument follows from Item~\ref{def:Stokes3} of Definition \ref{def:Stokes-fillocsys}. Assume then that a splitting~$\lambda$ exists on $I$ satisfying $(*)$ with a boundary point being a Stokes argument $\theta$ such that, if~$I'$ is a small open neighbourhood of $\theta$, $I\cup I'$ still satisfies $(*)$. We will prove that a splitting exists on $I\cup I'$. Let~$\lambda'$ be a splitting of $\scL_\bbullet$ on $I'$ and set $J=I\cap I'$.

Given any $i$ and any $j<_Ji$, we can find $\eta_{ji},\eta'_{ji}:\gr_{c_i}\scL_{|J}\csto\gr_{c_j}\scL_{|J}$ such that
\[
\lambda'_i|_J=\lambda_i|_J+\sum_{j<_Ji}\lambda_j|_J\circ\eta_{ji}\quand\lambda_i|_J=\lambda'_i|_J+\sum_{j<_Ji}\lambda'_j|_J\circ\eta'_{ji}.
\]
If $j<_Ji$ and $j\not<_Ii$, then we must have $j<_{I'}i$. We can write
\[
\lambda'_i|_J-\sum_{j<_{I'}i}\lambda_j|_J\circ\eta_{ji}=\lambda_i|_J+\sum_{j<_Ii}\lambda_j|_J\circ\eta_{ji}.
\]
The right-hand side is a new splitting on $I$. In the left-hand side, we use the equality $\lambda_j|_J\circ\eta_{ji}=\lambda'_j\circ\eta_{ji}+\sum_{k<_Jj}\lambda'_k|_J\circ\eta'_{kj}\circ\eta_{ji}$ and we find a new splitting on the left-hand side (but possibly not on the right-hand side)
\begin{multline*}
\lambda'_i|_J-\sum_{j<_{I'}i}\biggl(\lambda'_j|_J\circ\eta_{ji}+\sum_{k<_{I'}j}\lambda'_k|_J\circ\eta'_{kj}\circ\eta_{ji} \biggr)\\
=\lambda_i|_J+\sum_{j<_Ii} \biggl(\lambda_j|_J\circ\eta_{ji}-\sum_{k<_Ij}\lambda'_k|_J\circ\eta'_{kj}\circ\eta_{ji} \biggr).
\end{multline*}
By iterating this process, we get new splittings both on $I'$ and $I$ that coincide on $J$, defining thus a splitting on $I\cup I'$.

For Item \ref{prop:strict2}, we can assume that $(\scL,\scL_\bbullet)|_I$ and $\scL',\scL'_\bbullet)|_I$ are graded, according to Item~\ref{prop:strict1}, and write $\lambda=(\lambda_{ji})$ with $\lambda_{ji}:\gr_{c_i}\scL|_I\csto\gr_{c_j}\scL'|_I$ being constant and zero at any $\theta\in I$ such that $c_j<_\theta c_i$. Since such a $\theta$ exists by assumption on $I$ for any pair $c_i\neq c_j$, it follows that $\lambda|_I$ is diagonal. Then, Item \ref{prop:strict3} follows easily.
\end{proof}

\begin{corollary}\label{cor:isoStokes}
Let $\lambda:(\scL,\scL_\bbullet)\csto(\scL',\scL'_\bbullet)$ be a morphism of Stokes-filtered $\csring$\nobreakdash-local systems indexed by $C$. Assume that $\lambda:\scL\csto\scL'$ is an isomorphism. Then~$\lambda$ is an isomorphism of Stokes-filtered $\csring$-local systems.
\end{corollary}

\begin{proof}
By Proposition \ref{prop:strict}\eqref{prop:strict2}, $\lambda$ is graded on each good interval, hence each $\gr_c\lambda$ is also an isomorphism. The assertion follows from Property \ref{def:Stokes-fillocsys}\eqref{def:Stokes3}.
\end{proof}

\subsubsection*{Stokes data}
In the same way that the datum of a local system $\scL$ on $S^1$ is equivalent, when fixing a base point $\theta_o\in S^1$, to the data of the $\csring$-module $\scL_{\theta_o}$ together with an automor\-phism~$T_{\theta_o}$, we can represent in an equivalent way a Stokes-filtered local system on~$S^1$ by a set of linear data, called the \emph{Stokes data}. Below, by an increasing filtration of an $\csring$-module $L$, we mean a \emph{finite exhaustive} increasing filtration $F_\bbullet L$ indexed by~$\ZZ$ by $\csring$-submodules, \ie such that $F_kL=0$ for $k\ll0$ and $F_kL=L$ for $k\gg0$. Similar convention for a decreasing filtration.

\begin{definition}[Stokes data]
By \index{Stokes!data}\emph{Stokes data} we mean the data of
\begin{itemize}
\item
a pair of $\csring$-modules $L_o,L'_o$,
\item
an increasing filtration $F_\bbullet L_o$, and a decreasing filtration $F^\bbullet L'_o$,
\item
a pair of isomorphisms $S_o^\pm:L_o\isom L'_o$ (called the Stokes isomorphisms),
\end{itemize}
subject to the following condition:
\begin{itemize}
\item
The image increasing filtration $S_o^+(F_\bbullet L_o)$ is \emph{opposite} to the decreasing filtration $F^\bbullet L'_o$, and the image decreasing filtration $(S_o^-)^{-1}(F^\bbullet L'_o)$ is opposite to the increasing filtration $F_\bbullet L_o$.
\end{itemize}
\end{definition}

Recall that, for an $\csring$-module $E$, an increasing filtration $F_\bbullet E$ is \index{filtration!opposite --}\emph{opposite} to a decreasing filtration $G^\cbbullet E$ if any of the equivalent properties is satisfied:
\begin{itemize}
\item
for each $k\in\ZZ$, $E=F_kE\oplus G^{k+1}E$;
\item
the $\csring$-module $E$ decomposes as $E=\csbigoplus_k(F_kE\cap G^kE)$ and the filtrations are those obtained from this grading.
\end{itemize}

We now attach Stokes data to a Stokes-filtered local system $(\scL,\scL_\bbullet)$. In order to do so, we \emph{choose} $\theta_o$ general, that is, not a Stokes argument, and we set $\theta'_o=\theta_o+\pi$. The Stokes data associated to $((\scL,\scL_\bbullet),\theta_o)$ consist of
\begin{itemize}
\item
the $\csring$-modules $\scL_{\theta_o},\scL_{\theta'_o}$ (the stalks of~$\scL$ at~$\theta_o,\theta'_o$),
\item
the Stokes isomorphisms \hbox{$S^+_{\theta_o},S^-_{\theta_o}:\scL_{\theta_o}\!\isom\!\scL_{\theta'_o}$} defined as follows: setting
\[
I^+=(\theta_o-\varepsilon,\theta'_o+\varepsilon),\quad I^-=(\theta'_o-\varepsilon,\theta_o+\varepsilon), \quand L^\pm=\Gamma(I_\pm,\scL),
\]
and considering the diagram of restriction isomorphisms
\[
\xymatrix@=.5cm{
&L^+\ar[dl]_{a'_+}\ar[dr]^{a_+}&\\
\scL_{\theta'_o}&&\scL_{\theta_o}\\
&L^-\ar[ul]^{a'_-}\ar[ur]_{a_-}&
}
\]
we define
\[
\left.
\begin{aligned}
S^+_{\theta_o}&=a'_+a_+^{-1}\\
S^-_{\theta_o}&=a'_-a_-^{-1}
\end{aligned}
\right\}
:\scL_{\theta_o}\isom \scL'_{\theta_o}.
\]

\item
In order to define the filtrations, we write $C=\{c_1,\dots,c_n\}$ where the numbering respects the order at $\theta_o$, and thus the reverse order at $\theta'_o$. By Proposition \ref{prop:strict}, there exist \emph{unique decompositions}
\[
L^+=\csbigoplus_{i=1}^nG_{c_i}^+,\quad L^-=\csbigoplus_{i=1}^nG_{c_i}^-
\]
from which we deduce decompositions
\begin{equation}\label{eq:decthetao}
\scL_{\theta_o}=\csbigoplus_{i=1}^nG_{c_i,\theta_o},\quad \scL_{\theta'_o}=\csbigoplus_{i=1}^nG_{c_i,\theta'_o},
\end{equation}
pushed by $a'_+$, \resp $a_-$, from the decompositions of $L^+$, \resp $L^-$. With respect to these decompositions, $a_-^{-1}a_+$ is compatible with the associated increasing filtrations, while $a^{\prime-1}_-a'_+$ is compatible with the associated decreasing filtrations, and their associated graded morphisms are isomorphisms. It follows that, with respect to the decompositions \eqref{eq:decthetao}, $S^+_{\theta_o}$ preserves the increasing filtrations, while $S^-_{\theta_o}$ preserves the decreasing filtrations. The oppositeness condition is thus clearly satisfied.
\end{itemize}

\begin{remark}
Stokes data allow one to recover in an equivalent way the Stokes-filtered local system $(\scL,\scL_\bbullet)$. For example, the monodromy $T_{\theta_o}:\scL_{\theta_o}\isom \scL_{\theta_o}$ is given by $T_{\theta_o}=(S^-_{\theta_o})^{-1}S^+_{\theta_o}$.
\end{remark}

\subsection{Pairings and Stokes matrices}

The pairing \eqref{eq:dualHUf} can be regarded as the fiber at $\theta=1$ of a pairing of local systems
\begin{equation}\label{eq:dualHUfS1}
\cswt\Phi_\infty^r(U,f)\otimes\cswt\Phi_{\infty,\rc}^{2d-r}(U,-f)\csto \csring_{S^1},
\end{equation}
for each $r$, where $\csring_{S^1}$ is seen as $R^{2d}p_!\csring_{S^1\times U}$ and $p$ is the \hbox{projection $S^1\times U\!\csto\!S^1$}. In a way similar to that of Proposition \ref{prop:PDUf}, we obtain:

\begin{proposition}
If $\csring$ is a field, for each $r$ the pairing \eqref{eq:dualHUfS1} is nondegenerate.
\end{proposition}

\begin{proof}
One can apply Proposition \ref{prop:PDUf} for each $\rme^{-\sfi\theta}f$, by using that $\bR p_!$ commutes with restriction to $\theta$.
\end{proof}

The main result of this section is:

\begin{theorem}\label{th:dualHUfS1}
The pairing \eqref{eq:dualHUfS1} is compatible with the Stokes filtrations.
\end{theorem}

We first explain the notion of compatibility and give some consequences.

\subsubsection*{Pairings of Stokes-filtered local systems}
As usual for pairings between filtered objects with values in the constant sheaf with trivial filtration jumping only at zero, we pair a term of index $a$ with a term of index $-a$. In the case of Stokes-filtered local systems, we pair an object indexed by $C$ with an object indexed by $-C$.

In order to implement the change of $C$ into $-C$, we consider the involution $\iota:\theta\mto\theta+\pi$ on $S^1$ (this is reminiscent of the sign occurring in Theorem \ref{th:kontsevichyu}\eqref{th:kontsevichyu2}). If $(\scL',\scL'_\bbullet)$ is indexed by $C$, one naturally defines the Stokes-filtered local system $\iota^{-1}(\scL',\scL'_\bbullet)$ indexed by $-C$: the underlying local system is $\iota^{-1}\scL'$ (isomorphic to~$\scL'$) and the Stokes filtration is $(\iota^{-1}\scL')_{\leqslant-c}=\iota^{-1}(\scL'_{\leqslant c})$. The corresponding Stokes data are obtained by exchanging~$\theta_o$ and $\theta'_o$ and inverting the Stokes isomorphisms. For any $c\in C$, we have $\gr_{-c}(\iota^{-1}\scL')=\iota^{-1}\gr_c\scL'\simeq\gr_c\scL'$.

Assume that we are given Stokes-filtered local systems $(\scL,\scL_\bbullet)$ and $(\scL',\scL'_\bbullet)$ indexed by $C$ and a pairing between the local systems $\scL,\iota^{-1}\scL'$:
\[
\varh:\scL\otimes_\csring\iota^{-1}\scL'\csto \csring_{S^1}.
\]
We say that this \index{pairing!of Stokes-filtered local systems}pairing is compatible with the Stokes filtrations (or is a pairing of Stokes-filtered local systems) if, for any $c\in C$, $\varh$ induces the zero morphism on
\[
\scL_{\leqslant c}\otimes_\csring(\iota^{-1}\scL')_{<-c}\quand\scL_{<c}\otimes_\csring(\iota^{-1}\scL')_{\leqslant -c}.
\]
It follows that $\varh$ induces, for each $c\in C$, a pairing
\begin{equation}\label{eq:Pc}
\varh_c:\gr_c\scL\otimes_\csring\iota^{-1}\gr_c\scL'\csto \csring_{S^1}.
\end{equation}
In order to conveniently deal with duality, \emph{we now assume that either~$\csring$ is a field (\eg $\QQ,\CC$) or $\csring=\ZZ$ and the Stokes-filtered local systems we consider are \emph{torsion-free of finite rank} in the sense that each $\gr_c\scL$ (hence also $\scL$) locally $\csring$-free of finite rank.}

The notion of a dual Stokes-filtered local system $(\scL,\scL_\bbullet)^\csvee$ is naturally defined: the underlying local system is the dual local system $\scL^\csvee:=\cHom_\csring(\scL,\csring)$; the Stokes filtration is indexed by $-C$ and is defined by $(\scL^\csvee)_{\leqslant-c}=(\scL_{<c})^\csperp$; it satisfies $\gr_{-c}(\scL^\csvee)\simeq(\gr_c\scL)^\csvee$. It is then clear that the conditions for a Stokes filtration are fulfilled.

The Stokes-filtered local system $\iota^{-1}(\scL',\scL'_\bbullet)^\csvee$ is thus indexed by $C$ and it is natural to consider it as the suitable dual object of $(\scL',\scL'_\bbullet)$. A pairing $\varh$ as above is then nothing but a morphism $(\scL,\scL_\bbullet)\csto\iota^{-1}(\scL',\scL'_\bbullet)^\csvee$ compatible with the Stokes filtrations.

The Stokes filtration associated to $\iota^{-1}\scL^{\prime\csvee}$ gives rise to the Stokes data and decom\-positions
\[
\Bigl(\csbigoplus_{i=1}^nG^{\prime\csvee}_{c_i,\theta'_o},\;\;\csbigoplus_{i=1}^nG^{\prime\csvee}_{c_i,\theta_o},\;\tS^+_{\theta_o},\;\tS^-_{\theta_o}\Bigr)
\]
(\cf\cite[\S3]{H-S09}).

The data of the pairing $\varh$ translates into the data of a pair of morphisms\enlargethispage{\baselineskip}
\[
\varh_{\theta_o,\theta'_o}:\csbigoplus_{i=1}^nG_{c_i,\theta_o}\csto \csbigoplus_{i=1}^nG^{\prime\csvee}_{c_i,\theta'_o},\quad\varh_{\theta'_o,\theta_o}:\csbigoplus_{i=1}^nG_{c_i,\theta'_o}\csto\csbigoplus_{i=1}^nG^{\prime\csvee}_{c_i,\theta_o}
\]
which, due to the compatibility with the Stokes filtrations, are block-diagonal, with blocks $\varh_{\theta_o,\theta'_o}^{(i)},\varh_{\theta'_o,\theta_o}^{(i)}$, according to Proposition \ref{prop:strict}\eqref{prop:strict2}, and which satisfy
\[
\tS^+_{\theta_o}\circ\varh_{\theta_o,\theta'_o}=\varh_{\theta'_o,\theta_o}\circ S^+_{\theta_o},\quad\tS^-_{\theta_o}\circ\varh_{\theta_o,\theta'_o}=\varh_{\theta'_o,\theta_o}\circ S^-_{\theta_o}.
\]
In other words, $\varh_{\theta_o,\theta'_o}$, \resp $\varh_{\theta'_o,\theta_o}$, is the direct sum of pairings
\[
\varh_{\theta_o,\theta'_o}^{(i)}:G_{c_i,\theta_o}\otimes G'_{c_i,\theta'_o}\csto \csring,\quad\resp \varh_{\theta'_o,\theta_o}^{(i)}:G_{c_i,\theta'_o}\otimes G'_{c_i,\theta_o}\csto \csring,
\]
which satisfy, for $g\in \scL_{\theta_o}=\csbigoplus_{i=1}^nG_{c_i,\theta_o}$ and $g'\in \scL'_{\theta_o}=\csbigoplus_{i=1}^nG_{c_i,\theta_o}$,
\[
\varh_{\theta_o,\theta'_o}(g,S^+_{\theta_o}g')=\varh_{\theta'_o,\theta_o}(S^+_{\theta_o}g,g')
\]
and a similar equality with $S^-_{\theta_o}$. This defines a pairing $\varh^+_{\theta_o}$ between $\scL_{\theta_o}$ and $\scL'_{\theta_o}$ by~the formula
\[
\varh^+_{\theta_o}(g,g'):=\varh_{\theta_o,\theta'_o}(g,S^+_{\theta_o}g')=\varh_{\theta'_o,\theta_o}(S^+_{\theta_o}g,g'),\quad \begin{cases}
g\in\csbigoplus_{i=1}^nG_{c_i,\theta_o},\\[3pt]
g'\in\csbigoplus_{i=1}^nG'_{c_i,\theta_o}.
\end{cases}
\]
We similarly have a pairing $\varh^-_{\theta_o}$ between $\scL_{\theta_o}$ and $\scL'_{\theta_o}$ by replacing $S^+_{\theta_o}$ with $S^-_{\theta_o}$ in the above formulas. Since $S^-_{\theta_o}=S^+_{\theta_o}T_{\theta_o}$, we find
\[
\varh'^-_{\theta_o}(g,g')=\varh'^+_{\theta_o}(T_{\theta_o}g,g')=\varh'^+_{\theta_o}(g,T_{\theta_o}g').
\]

\begin{definition}[Seifert pairing]
If $\scL=\scL'$, we call the pairing $\varh^+_{\theta_o}$ the \index{pairing!Seifert --}\emph{Seifert pairing} associated with the pairing $\varh:\scL\otimes\iota^{-1}\scL\csto \csring_{S^1}$.
\end{definition}

\begin{lemma}\label{lem:nondegP}
The following properties are equivalent for a pairing $\varh$ compatible with the Stokes filtrations:
\begin{enumerate}
\item\label{lem:nondegP1}
$\varh$ induces an isomorphism $(\scL,\scL_\bbullet)\isom\iota^{-1}(\scL',\scL'_\bbullet)^\csvee$,
\item\label{lem:nondegP2}
$\varh$ induces an isomorphism $\scL\isom\scL^{\prime\csvee}$,
\item\label{lem:nondegP3}
for each $c\in C$, $\varh_c:\gr_c\scL\otimes\iota^{-1}\gr_c\scL'\csto \csring_{S^1}$ is nondegenerate.
\end{enumerate}
\end{lemma}

\begin{proof}
The equivalence of Items \ref{lem:nondegP1} and \ref{lem:nondegP2} follows from Corollary \ref{cor:isoStokes}. For that of Items \ref{lem:nondegP2} and~\ref{lem:nondegP3}, one uses that, on any good open interval $I$, the morphism $\varh|_I:\scL|_I\csto\iota^{-1}\scL^{\prime\csvee}|_I$ is graded.
\end{proof}

The following proposition gives an interpretation of the matrix of the Stokes isomorphism $S^+_{\theta_o}$ in terms of matrix of the pairing $\varh^+_{\theta_o}$, called the \emph{Seifert matrix} when $\scL=\scL'$.

\begin{proposition}\label{prop:hS}
Assume that $\varh:\scL\otimes\iota^{-1}\scL'\csto \csring_{S_1}$ is nondegenerate. For each $i=1,\dots,n$, let $\bme_i,\bme'_i$ be $\csring$-bases of $G_{c_i,\theta_o}, G'_{c_i,\theta'_o}$ in which the matrix of $\varh_{\theta_o,\theta'_o}^{(i)}$ is the identity. Then the Seifert matrix of $\varh^+_{\theta_o}$ in the basis $\bme\!=\!\csbigoplus_i\bme_i$ is equal to the Stokes matrix of~$S^+_{\theta_o}$ in the bases $\bme,\bme'$.\qed
\end{proposition}

\subsubsection*{Pairings of perverse sheaves on \texorpdfstring{$\Af1$}{A1}}
Let $\cF,\cF'$ be $\csring$-perverse sheaves on $\Af1$ and let $\pairing:\cF\otimes^L_\csring\cF'\csto \csring_{\Af1}[2]$ be a \index{pairing!of perverse sheaves}pairing (an $\csring$-bilinear morphism in $\catD^\rb_{\cc}(\csring_{\Af1})$). Giving~$\pairing$ is equivalent to giving a morphism $\cF\csto\bD\cF'$ in $\catD^\rb_{\cc}(\csring_{\Af1})$, where $\bD\cF'$ is the Poincaré-Verdier dual of~$\cF'$. We say that~$\pairing$ is \emph{nondegenerate} if the latter morphism is an isomorphism. If a nondegenerate pairing exists on $\cF$, then $\bD\cF'$ is also perverse, and so $\cF'$ (and similarly $\cF$) is strongly perverse (\cf Section \ref{subsec:cohomtools}).

We now explain how such a pairing $\pairing$ induces a pairing between the Stokes-filtered local system associated to $\cF$ and $\cF'$ by Proposition \ref{prop:StokesF}, from which we keep the notation. Any pairing $\pairing$ defines, by pullback, a pairing $\cswt\cF\otimes^L\cswt\cF'\csto \csring_{S^1\times\Af1}[2]$, where $\cswt\cF=q^{-1}\cF$ and $q:S^1\times\Af1\csto\Af1$ is the projection.

Let $\cswt\iota=\iota\times\id:S^1\times\Af1 \csto S^1\times\Af1$ be the involution induced by $\iota$. Then there exists a natural morphism
\[
\csring_{\cswt\Delta_\rho^{\leqslant\infty}}\otimes \csring_{\cswt\iota(\cswt\Delta_\rho^{\leqslant\infty})}\csto \csring_{\cswt\Delta_\rho^{\leqslant\infty}\cap\cswt\iota(\cswt\Delta_\rho^{\leqslant\infty})},\quad \bun_{\cswt\Delta_\rho^{\leqslant\infty}}\otimes\bun_{\cswt\iota(\cswt\Delta_\rho^{\leqslant\infty})}\mto\bun_{\cswt\Delta_\rho^{\leqslant\infty}\cap\cswt\iota(\cswt\Delta_\rho^{\leqslant\infty})}.
\]
The intersection $\cswt\Delta_\rho^{\leqslant\infty}\cap\cswt\iota(\cswt\Delta_\rho^{\leqslant\infty})$ is the product of $S^1$ by the open disc $\rond\Delta_\rho$. Together with~$\pairing$, we~can thus consider the pairing
\[
\cswt\pairing:(\csring_{\cswt\Delta_\rho^{\leqslant\infty}}\otimes\cswt\cF)\otimes^L_\csring(\csring_{\cswt\iota(\cswt\Delta_\rho^{\leqslant\infty})}\otimes\cswt\cF')\csto \csring_{S^1\times\rond\Delta_\rho}[2],
\]
and, by applying $\bR p_!$ and taking cohomology in degree zero, the pairing
\[
\varh:\cswt\Phi_\infty(\cF)\otimes\iota^{-1}\cswt\Phi_\infty(\cF')\csto \csring_{S^1}.
\]
Replacing $\cswt\Delta_\rho^{\leqslant\infty}$ with $\cswt\Delta_c$ in the formulas above, we obtain similarly for each $c\in C$ a pairing
\[
\varh_c:\cswt\pphi_c(\cF)\otimes\iota^{-1}\cswt\pphi_c(\cF')\csto \csring_{S^1}.
\]
\enlargethispage{-\baselineskip}

\pagebreak[2]
\begin{proposition}\label{prop:pairingF}
For $\csring$-perverse sheaves $\cF,\cF'$ with a pairing $\pairing$,
\begin{enumerate}
\item\label{prop:pairingF1}
the pairing $\varh:\cswt\Phi_\infty(\cF)\otimes\iota^{-1}\cswt\Phi_\infty(\cF')\csto \csring_{S^1}$ is compatible with the Stokes filtrations, and the induced pairing \eqref{eq:Pc} is equal to $\varh_c$ defined above;

\item\label{prop:pairingF2}
If $\varh_c$ is nondegenerate for each $c\in C$, then each $\cswt\pphi_c(\cF),\cswt\pphi_c(\cF')$, and thus $\cswt\Phi_\infty(\cF),\cswt\Phi_\infty(\cF')$, are $\csring$-free, and $\varh$ is nondegenerate;

\item\label{prop:pairingF3}
if $\pairing$ is nondegenerate, then $\cF$ and $\cF'$ are strongly perverse, $\cswt\Phi_\infty(\cF),\cswt\Phi_\infty(\cF')$ and each $\cswt\pphi_c(\cF),\cswt\pphi_c(\cF')$ are $\csring$-free, $\varh$ and each $\varh_c$ are nondegenerate.
\end{enumerate}
\end{proposition}

\begin{proof}
For the first point, we have to show that, for each $c,\theta$, the pairing $\varh$ induces zero on $\cswt\Phi_\infty(\cF)_{\leqslant c,\theta}\otimes\cswt\Phi_\infty(\cF')_{<c,\theta'}$ and on $\cswt\Phi_\infty(\cF)_{<c,\theta}\otimes\cswt\Phi_\infty(\cF')_{\leqslant c,\theta'}$. But this is clear since, in the first case for example,
\begin{itemize}
\item
we can compute $\cswt\Phi_\infty(\cF)_{<c,\theta'}$ with the set $\cswt\Delta_{\rho,\theta'}^{<c+\varepsilon\rme^{\sfi\theta}}$, according to Lemma \ref{lem:vanishing},
\item
the sets $\cswt\Delta_{\rho,\theta}^{\leqslant c}$ and $\cswt\Delta_{\rho,\theta'}^{<c+\varepsilon\rme^{\sfi\theta}}$ do not intersect.
\end{itemize}
The identification of both definitions of $\varh_c$ are then clear.

The second and third points are consequences of the results already obtained.
\end{proof}

\begin{proof}[of Theorem \ref{th:dualHUfS1}]
We apply Proposition \ref{prop:pairingF} to $\cF=\pR^{r-d}f_*\pring_U$ and $\cF'=\pR^{2d-r}f_!\pring_U$, and we identify the Stokes-filtered local systems
\[
\iota^{-1}(\cswt\Phi_\infty(\cF'),\cswt\Phi_\infty(\cF')_\bbullet)\simeq(\cswt\Phi_{\infty,\rc}^{2d-r}(U,-f),\cswt\Phi_{\infty,\rc}^{2d-r}(U,-f)_\bbullet).\eqno{\hbox{\rlap{$\sqcap$}$\sqcup$}}
\]
\end{proof}

\subsection{The twisted de~Rham complex with an algebraic parameter}\label{subsec:dRalgparam}
We revisit the results of Section \ref{subsec:algDR} by introducing a rescaling parameter for the function $f$. We~thus consider the rescaling $f\rightsquigarrow f/u$. In contrast with the definition of $\cswt\Phi_\infty(U,f)$, we do not restrict to $|u|=1$ but consider $u\in\Gm$, so that we can analyze the behavior at the limit $u=0$. We will also consider in the next subsection a~formal parametrization by $\CC\lpr u\rpr$. The purpose is to obtain in an algebraic (or formal) way some expressions related to the monodromy operators occurring Section \ref{sec:constStokesfilt}.

\subsubsection*{The twisted de~Rham complex with parameter in \texorpdfstring{$\Gm$}{Gm}}
In algebraic terms, we consider the \index{complex!twisted algebraic de~Rham -- with parameter}twisted de~Rham complex
\[
(\Omega_U^\cbbullet[u,u^{-1}],\rd+\rd f/u),
\]
where the sections of $\Omega_U^k[u,u^{-1}]:=\Omega_U^k\otimes_\CC\CC[u,u^{-1}]$ are Laurent polynomials in~$u$ with coefficients being sections of $\Omega_U^k$, and the differential $\rd$ only concerns the variables in~$U$. In other words, the notation $\rd+\rd f/u$ is a shorthand for $\rd\otimes\id+\rd f\otimes u^{-1}$. As~a consequence, the twisted de~Rham complex is a complex of $\CC[u,u^{-1}]$-modules, and its cohomologies $\coH^r_\dR(U,f/u)$ consist of $\CC[u,u^{-1}]$-modules.

On the other hand, we make the differential operator $\partial_u$ acts on each term of this complex by setting
\[
\partial_u(\omega\otimes u^\ell)=\ell\omega\otimes u^{\ell-1}-f\omega\otimes u^{\ell-2}.
\]
This operator commutes with the differential (this is easily checked if one interprets $\rd+\rd f/u$ as $\rme^{-f/u}\circ\rd\circ\rme^{f/u}$ and the above action of $\partial_u$ as that $\rme^{-f/u}\circ \partial_u\circ\rme^{f/u}$). It~induces therefore a differential operator on each cohomology module $\coH^r_\dR(U,f/u)$. The structure of these cohomology modules are described by the next theorem.

\begin{theorem}
For each $r$, the $\CC[u,u^{-1}]$-module $\coH^r_\dR(U,f/u)$ is free of finite rank equal to $\dim_\CC\coH^r_\dR(U,f)$, and the action of $\partial_u$ defines on it an algebraic connection with a regular singularity at $u=\infty$ and a possibly irregular singularity at $u=0$.
\end{theorem}

\begin{proof}[Sketch]
Setting $v=u^{-1}$, it is convenient to first consider the complex $(\Omega^\cbbullet_U[v],\rd+v\rd f)$, which also comes equipped with an action of $\partial_v$ defined as $\rme^{-vf}\circ \partial_v\circ\rme^{vf}$. Then $\coH^r_\dR(U,vf)$ is a $\CC[v]$-module equipped with a compatible action of $\partial_v$, \ie a module over the Weyl algebra $\CC[v]\langle\partial_v\rangle$ of differential operators with polynomial coefficients in the variable $v$. Let us regard $v$ as the derivation $\partial_t$ with respect to the variable $t=-\partial_v$. Regarded as a $\CC[t]\langle\partial_t\rangle$-module, $\coH^r_\dR(U,vf)$ is interpreted as the $(d-r)$-Gauss-Manin system attached to $f$. The regularity theorem of Griffiths implies that it has a regular singularities at any of its singularity \hbox{$c\in\Af1$} and also at infinity. A standard theorem of differential equations implies that the $\CC[v]\langle\partial_v\rangle$-module $\coH^r_\dR(U,vf)$, regarded as the Laplace transform of a $\CC[t]\langle\partial_t\rangle$-module with regular singularities, has a regular singularity at $v=0$, an possibly irregular one at $v=\infty$, and no other singularity. The statement of the theorem is obtained by tensoring $\coH^r_\dR(U,vf)$ with $\CC[u,u^{-1}]=\CC[v^{-1},v]$ over $\CC[v]$. The assertion on the rank is obtained by identifying the restriction to $u=1$ of the free $\CC[u,u^{-1}]$-module $\coH^r_\dR(U,f/u)$ with $\coH^r_\dR(U,f)$.
\end{proof}

\begin{remark}[Comparison with the topological data]
The sheaf of horizontal sections
\[
\ker\bigl[\partial_u:\coH^r_\dR(U,f/u)^\an\csto\coH^r_\dR(U,f/u)^\an\bigr]
\]
is \index{sheaf!locally constant --}locally constant on $\CC^*$ with stalk isomorphic to the $\CC$-vector space $\coH^r_\dR(U,f)\simeq\coH^r(U,f;\CC)$. When restricting to $|u|=1$, we recover the local system $\cswt\Phi_\infty^r(U,f):=\cswt\Phi_\infty(\pR^{r-d}f_*\CC_U)$ considered in Section \ref{sec:quiverperverse}.

Furthermore, the asymptotic analysis of the meromorphic connection
\[
(\coH^r_\dR(U,f/u)^\an,\partial_u)
\]
near $u=0$ provides Stokes data, that can be expressed as a Stokes-filtered local system with underlying local system being that considered above. This Stokes-filtered local system can be identified with that considered in Section \ref{sec:constStokesfilt} (\cf\cite[Chap.\,XII]{Malgrange91} and~\cite{D-H-M-S17}).
\end{remark}

\subsubsection*{Duality pairing}
The theory of $\cD$-modules enables one to define the cohomology with compact support $\coH^r_{\dR,\rc}(U,f/u)$ as a $\CC[u,u^{-1}]$-module with an action of $\partial_u$ together with a nondegenerate pairing
\[
\coH^r_{\dR}(U,f/u)\otimes_{\CC[u,u^{-1}]}\coH^{2d-r}_{\dR,\rc}(U,-f/u)\csto \coH^{2d}_{\dR,\rc}(U)[u,u^{-1}]\simeq\CC[u,u^{-1}].
\]
We will however produce such cohomology modules with compact support and such pairings with concrete de~Rham complexes by adding a parameter $u$ in those considered in Theorem \ref{th:kontsevichyu} (\cf \cite[Th.\,B]{Bibi22a}).

\begin{theorem}\mbox{}\label{th:kontsevichyuu}
\begin{enumerate}
\item\label{th:kontsevichyuu1}
Restriction to $U$ induces isomorphisms of $\CC[u,u^{-1}]$-modules with a compatible action of $\partial_u$, for all $r\in\NN$:
\begin{align*}
\coH^r_\dR(U,f/u)&\simeq \coH^r\bigl(\csov X,(\Omega_f^\cbbullet[u,u^{-1}],\rd+\rd f/u)\bigr),
\\
\coH^r_{\dR,\rc}(U,f/u)&\simeq \coH^r\bigl(\csov X,(\Omega_f^\cbbullet(-D)[u,u^{-1}],\rd+\rd f/u)\bigr).
\end{align*}
\item\label{th:kontsevichyuu2}
For each $r\in\NN$, the natural de~Rham pairing
\[
\coH^r_\dR(U,f/u)\otimes_\CC\coH^{2d-r}_{\dR,\rc}(U,-f/u)\csto \coH^{2d}_{\dR}(\csov X)[u,u^{-1}]\simeq\CC[u,u^{-1}]
\]
is compatible with the action of $\partial_u$ and is nondegenerate.
\end{enumerate}
\end{theorem}

As in Remark \ref{rem:kontsevichyu}, the first statement also holds when replacing the Kontsevich complexes with their restriction to $X$, that is:
\begin{equation}\label{eq:UlogH}
\begin{aligned}
\coH^r_\dR(U,f/u)&\simeq \coH^r\bigl(X,(\Omega_X^\cbbullet(\log H)[u,u^{-1}],\rd+\rd f_X/u)\bigr),
\\
\coH^r_{\dR,\rc}(U,f/u)&\simeq \coH^r\bigl(X,(\Omega_X^\cbbullet(\log H)(-H)[u,u^{-1}],\rd+\rd f_X/u)\bigr).
\end{aligned}
\end{equation}

\subsubsection*{The Brieskorn lattice and the Barannikov-Kontsevich theorem}
We consider polynomial coefficients instead of Laurent polynomials, and in order to take into account the action of $\rd f/u$, we consider the subcomplex
\[
(u^{-\cbbullet}\Omega_U^\cbbullet[u],\rd+\rd f/u)\subset(\Omega_U^\cbbullet[u,u^{-1}],\rd+\rd f/u),
\]
where the terms in degree $k$ have a pole of order at most $k$. By multiplying the degree~$k$ term by $u^k$ we obtain the isomorphic complex
\[
(\Omega_U^\cbbullet[u],u\rd+\rd f).
\]
It is natural to try to ``set $u=0$'' in this complex and to compare the result with the complex $(\Omega_U^\cbbullet,\rd f)$, which is a complex in the category of coherent $\cO_U$-modules and whose cohomology sheaves are $\cO_U$-coherent and supported on the critical set of~$f$. However, without any other assumption on $f$, none of these complexes present finiteness properties. For example, if $U=\Af1$ with coordinate $t$ and $f=0$, then
\[
\coH^1\bigl(U,(\Omega_U^\cbbullet[u],u\rd+\rd f)\bigr)\simeq\coker\bigl(u\partial_t:\CC[t,u]\csto\CC[t,u]\bigr)
\]
is identified with the $\CC$-vector space $\CC[t]$ on which $u$ acts by zero. Although one could consider such objects with duality pairing as defined in \cite[Th.\,3.1]{Hartshorne72}, we~will instead consider the logarithmic expressions like in Theorem \ref{th:kontsevichyuu}\eqref{th:kontsevichyuu1} or in \eqref{eq:UlogH}. This construction is analogous to that of the \index{lattice!Brieskorn --}\emph{Brieskorn lattice} in Singularity theory.

\begin{theorem}[The Brieskorn lattices, {\cf\cite[\S8]{S-Y14} and \cite[Th.\,A]{Bibi22a}}]\label{th:A}\mbox{}
\begin{enumerate}
\item\label{th:A1}
The $\CC[u]$-modules
\begin{align*}
&\coH^r\bigl(X,(\Omega_X^\cbbullet(\log H)[u],u\rd+\rd f_X)\bigr),\\
&\coH^r\bigl(X,(\Omega_X^\cbbullet(\log H)(-H)[u],u\rd+\rd f_X)\bigr)
\end{align*}
are $\CC[u]$-free of finite rank, and equipped with an action of $\partial_u$ having a pole of order at most~$2$ at $u=0$, a regular singularity at infinity and no other pole. After tensoring by $\CC[u,u^{-1}]$ these free $\CC[u]$-modules, one recovers $\coH^r_\dR(U,f/u)$ and $\coH^r_{\dR,\rc}(U,f/u)$ respectively.

\item\label{th:A2}
Furthermore, there is a natural perfect pairing compatible with the actions of~$\partial_u$:
\begin{multline*}\label{eq:pairX}
\coH^r\bigl(X,(\Omega_X^\cbbullet(\log H)[u],u\rd+\rd f_X)\bigr)\\\otimes_{\CC[u]} \coH^{2d-r}\bigl(X,(\Omega_X^\cbbullet(\log H)(-H)[u],u\rd-\rd f_X)\bigr)
\csto \CC[u].
\end{multline*}

\item\label{th:A3}
By restricting this pairing modulo $u\CC[u]$, we obtain a perfect pairing
\[
\coH^r\bigl(X,(\Omega_X^\cbbullet(\log H),\rd f_X)\bigr)\\\otimes\coH^{2d-r}\bigl(X,(\Omega_X^\cbbullet(\log H)(-H),-\rd f_X)\bigr)
\csto\CC.
\]
\end{enumerate}
All these objects are independent of the choice of the good projectivization $(X,f)$ of $(U,f)$.
\end{theorem}

The freeness property of $\coH^r\bigl(X,(\Omega_X^\cbbullet(\log H)[u],u\rd+\rd f_X)\bigr)$ also follows from a variant of the Barannikov-Kontsevich theorem \cite[\S0.6]{Bibi97b}. The independence on the choice of the good projectivization follows from \cite{Yu12} and \cite[Prop.\,2.3]{C-Y16}. Let us also note that one has a meromorphic (instead of logarithmic) expression of the Brieskorn lattice (\cf \cite[(8.12)]{S-Y14}).

\begin{proof}[Sketch]
The results of \cite{E-S-Y13} allow one to replace the logarithmic complexes on~$X$ with the Kontsevich complexes on $\csov X$, as in Theorem \ref{th:kontsevichyuu}. In \loccit, it is also proved that the dimension of the cohomologies when restricted to any \hbox{$u_o\in\CC$} does not depend on $u_o$. This leads to the $\CC[u]$-freeness property. The duality property has been considered in \cite{C-Y16}.
\end{proof}

\subsection{The case of a formal parameter}
The approach of the previous section \ref{subsec:dRalgparam}, extending with an algebraic parameter~$u$ the results of Section \ref{subsec:algDR}, yields algebraic expressions for differential equations (connections) whose horizontal sections provide the \index{vanishing cycles!global --}\index{sheaf!of global vanishing cycles}global vanishing cycle sheaves $\cswt\Phi_\infty^r(U,f;\CC)$, $\cswt\Phi_{\infty,\rc}^r(U,f;\CC)$ on $S^1$. Furthermore, from the classical theory of differential equations, one can recover algebraically each vector space with monodromy
\[
\bH^r(f_X^{-1}(c);\phi_{f_X-c}(\bR j_*\CC_U)) \quand\bH^r_\rc(f_X^{-1}(c);\phi_{f_X-c}(\bR j_!\CC_U))
\]
by tensoring $\coH^r_\dR(U,f/u)$, respectively $\coH^r_{\dR,\rc}(U,f/u)$ with $\CC\lcr u\rcr$ and extracting a suitable direct summand of them.

As $f$ is possibly non proper, the vector spaces with monodromy
\[
\bH^r(f^{-1}(c);\phi_{f-c}\CC_U)\quand\bH^r_\rc(f^{-1}(c);\phi_{f-c}\CC_U)
\]
differ in general respectively from the latter.

\begin{example}\label{ex:HUUan}
Let $f\in\CC[t]$ be a non-constant polynomial in one variable and let $U$ be the Zariski open set of $\Af1$ complementary to $\{f'=0\}$, so that $\cO(U)=\CC[t,1/f']$. Then $\phip_{f-c}\pCC_U=0$ for any $c\in\CC$ since $f$ has no critical point in $U$, while the complex
\[
\CC[t,1/f']\To{\partial_t-f'}\CC[t,1/f']
\]
has cohomology in degree one only, of dimension \hbox{$\deg f\cdot\#\{f(t)\mid f'(t)=0\}$}.
\end{example}

The question remains to give an algebraic formula for the latter spaces with monodromy. This is the content of Theorem \ref{th:lpru} below.

For a coherent $\cO_U$-module $\cF$, we set
\[
\cF[u]:=\CC[u]\otimes_\CC\cF\quand\cF\lcr u\rcr=\varprojlim_\ell(\cF[u]/u^\ell\cF[u]).
\]
The latter module is in general not equal to $\CC\lcr u\rcr\otimes_\CC\cF$ and for $x\in U$ we have a strict inclusion $\cF\lcr u\rcr_x\subsetneq \cF_x\lcr u\rcr$: a germ of section of $\cF\lcr u\rcr$ at $x\in U$ consists of a formal power series $\sum_nf_nu^n$ where $f_n$ are sections of $\cF$ defined on a fixed neighbourhood of~$x$, while for $\cF_x\lcr u\rcr$ we allow the neighbourhood to be shrunk when $n\csto\infty$. In~par\-ticular, there is a natural morphism $\CC\lcr u\rcr\otimes_\CC\cF\csto\cF\lcr u\rcr$. Letting $\CC\lpr u\rpr$ denote the field of formal Laurent series, we then set $\cF\lpr u\rpr:=\cF\lcr u\rcr[u^{-1}]=\CC\lpr u\rpr\otimes_{\CC\lcr u\rcr}\cF\lcr u\rcr$. Taking cohomology does note commute in general with applying the functor $\lpr u\rpr$.

Let us consider the formally twisted de~Rham complex
\[
(\Omega_U^\cbbullet\lpr u\rpr,\rd+\rd f/u).
\]
For each $r\geqslant0$, the $\CC\lpr u\rpr$-vector space $\bH^r\bigl(U,(\Omega_U^\cbbullet\lpr u\rpr,\rd+\rd f/u)\bigr)$ is equipped with a compatible action of $\partial_u$.

\begin{example}
Let us illustrate the difference between the $\CC\lpr u\rpr$-vector spaces
\[
\bH^r\bigl(U,(\Omega_U^\cbbullet\lpr u\rpr,\rd+\rd f/u)\bigr)\quand\CC\lpr u\rpr\otimes_{\CC[u,u^{-1}]}\bH^r\bigl(U,(\Omega_U^\cbbullet[u,u^{-1}],\rd+\rd f/u)\bigr).
\]
Let us take up the notation of Example \ref{ex:HUUan}. Then the formally twisted de~Rham complex
\[
\cO(U)\lpr u\rpr\To{u\partial_t-f'}\cO(U)\lpr u\rpr
\]
has cohomology equal to zero. Indeed, let us show for example that the differential is onto. This amounts to showing that, given $\psi_{k_o},\psi_{k_o+1},\dots$ in $\CC[t,1/f']$ (with $k_o\in\ZZ$), we can find $\varphi_{k_o},\varphi_{k_o+1},\dots$ in $\CC[t,1/f']$ such that
\[
\psi_{k_o}=-f'\varphi_{k_o},\quad\psi_{k_o+1}=\partial_t\varphi_{k_o}-f'\varphi_{k_o+1},\dots,\psi_{k+1}=\partial_t\varphi_{k}-f'\varphi_{k+1},\dots,
\]
a system which can be solved inductively because $f'$ is invertible in $\CC[t,1/f']$.

On the other hand, the complex
\[
\CC[t,1/f'][u,u^{-1}]\To{u\partial_t-f'}\CC[t,1/f'][u,u^{-1}]
\]
has cohomology in degree one only, and this cohomology is a free $\CC[u,u^{-1}]$-module of rank equal to $\deg f\cdot\#\{f(t)\mid f'(t)=0\}$.
\end{example}

In order to state the next theorem, let us introduce a notation. Let $E$ be a finite dimensional $\CC$-vector space equipped with an automorphism $T$. Given a choice of a logarithm of $T$, that is, writing $T=\exp(-2\pi i M)$ for some $M:E\csto E$, we denote by $\cswh{\mathrm{RH}}{}^{-1}(E,T)$ the $\CC\lpr u\rpr$-vector space $E\lpr u\rpr$ equipped with the connection $\rd+M \rd u/u$. Given $c\in\CC$, we set $\cswh\cE^{-c/u}=(\CC\lpr u\rpr,\rd+c\rd u/u^2)$.

\begin{theorem}[{\cf\cite[Th.\,1]{S-MS12}}]\label{th:lpru}
Each hypercohomology $\CC\lpr u\rpr$-vector space
\[
\bH^r\bigl(U,(\Omega_U^\cbbullet\lpr u\rpr,\rd+\rd f/u)\bigr)
\]
is finite dimensional and, with respect to the action of $\partial_u$, it decomposes as the direct sum of $\CC\lpr u\rpr$-vector spaces with a compatible action of $\partial_u$
\[
\csbigoplus_{c\in\CC}\bH^r\bigl(U,(\Omega_U^\cbbullet\lpr u\rpr,\rd+\rd f/u)\bigr)_c
\]
such that, for each $c\in\CC$, the component $\bigl(\bH^r\bigl(U,(\Omega_U^\cbbullet\lpr u\rpr,\rd+\rd f/u)\bigr)_c, \partial_u\bigr)$ is isomorphic to
\[
\cswh\cE^{-c/u}\otimes_{\CC\lpr u\rpr}\cswh{\mathrm{RH}}{}^{-1}\bigl(\bH^{r-d}(f^{-1}(c),\phip_{f-c}\pCC_U),T\bigr).
\]
\end{theorem}

\begin{remark}
The case of cohomology with compact support is not treated in \cite{S-MS12} nor in \cite{Bibi10b}, but should be obtained in the same way. See however \cite{Schefers23} for a similar identification at the level of complexes.
\end{remark}

\subsection{Irregular mixed Hodge theory}\label{subsec:irrHodge}
We can consider the dependence of the irregular Hodge filtration $F_{\irr,\alpha}^\cbbullet\coH(U,f/u)$ ($\alpha$~belonging to a finite subset $\cA\subset\QQ\cap[0,1)$) with respect to the rescaling parameter~$u$ from two distinct point of views:
\begin{enumerate}
\item
From the point of view of Hodge theory, we are led to prove \emph{Griffiths' transversality property} when $u$ varies in $\CC^*$ with respect to the connection (\ie the action of~$\partial_u$) and to understand the limiting behavior when $u\csto\infty$ or $u\csto0$.
\item
From the point of view of Singularity theory, we emphasize the Brieskorn lattice as a free $\CC[u]$-module and we produce the family of irregular Hodge filtrations directly from the Brieskorn lattice and the Deligne canonical extensions at $u=\nobreak\infty$, due to the regular singularity of the action of $\partial_u$ at $u=\infty$. This approach is analogous to that of Varchenko, Pham and Scherk-Steenbrink, M.\,Saito for an isolated singularity of a complex hypersurface (\cf \cite{Varchenko82, Pham83b, S-S85,MSaito89}): in that case, the action of $\partial_u$ on the Brieskorn lattice, while having a pole of order two, has nevertheless a regular singularity there, and one produces the Hodge filtration on the vanishing cycles from the Brieskorn lattice and the Deligne canonical lattices at $u=0$.
\end{enumerate}

We will give an overview of these two aspects and of the way they are related. We mainly follow \cite{S-Y14,K-K-P14,C-Y16,Mochizuki15a} and \cite[Chap.\,3]{Bibi15}, which we refer to for detailed proofs.

\subsubsection*{The variational point of view: the Kontsevich bundles}
Let us fix $(U,f)$ as above and $\alpha\in\cA$. Recall the definition \eqref{eq:Omegalogalpha} of $\Omega_f^\cbbullet(\alpha)$.

\begin{theorem}
The dimension of the twisted de~Rham complex
\[
\coH^r\bigl(\csov X,(\Omega_f^\cbbullet(\alpha),u_o\rd+v_o\rd f)\bigr)
\]
is independent of $(u_o,v_o)\in\CC^2$.
\end{theorem}

It follows that we can define a vector bundle $\cK^r_\alpha$ on $\PP^1$, called the \index{bundle!Kontsevich --}\emph{Kontsevich bundle with index $\alpha$}, by decomposing $\PP^1=\Af1_v\cup\Af1_u$ with $u=v^{-1}$ on $\Gm: \Af1_v\cap\Af1_u $, and setting
\begin{align*}
\cK^r_\alpha|_{\Af1_v}&=\coH^r\bigl(\csov X,(\Omega_f^\cbbullet(\alpha)[v],\rd+v\rd f)\bigr),\\
\cK^r_\alpha|_{\Af1_u}&=\coH^r\bigl(\csov X,(\Omega_f^\cbbullet(\alpha)[u],u\rd+\rd f)\bigr)\\
&\simeq\coH^r\bigl(\csov X,(u^{-\cbbullet}\Omega_f^\cbbullet(\alpha)[u],\rd+\rd f/u)\bigr),\\
\cK^r_\alpha|_{\Gm}&=\coH^r\bigl(\csov X,(\Omega_f^\cbbullet(\alpha)[u,u^{-1}],\rd+\rd f/u)\bigr)\\
&\simeq\coH^r\bigl(\csov X,(\Omega_f^\cbbullet(\alpha)[v,v^{-1}],\rd+v\rd f)\bigr),
\end{align*}

In a way analogous to that of Theorem \ref{th:E1degalpha}, the filtration by the stupid truncation of each of these twisted de~Rham complexes degenerates at $E^1$ and induces a filtration $F^\cbbullet_{\irr,\alpha}\cK^r_\alpha$ indexed by $\NN$. Its restriction at $u=1$ is the filtration considered in that theorem. Furthermore, the Kontsevich bundle $\cK^r_\alpha$ is naturally equipped with a meromorphic connection $\nabla$, corresponding to the action of $\partial_u$ in the chart $\Af1_u$, which has a double pole at $u=0$, a simple pole at $v=0$, and no other pole.

\begin{proposition}\label{prop:HNK}
The filtration $F^\cbbullet_{\irr,\alpha}\cK^r_\alpha|_{\Gm}$ satisfies the Griffiths transversality property $\nabla(F^p_{\irr,\alpha}\cK^r_\alpha|_{\Gm})\subset\Omega^1_{\Gm}\otimes F^{p-1}_{\irr,\alpha}\cK^r_\alpha|_{\Gm}$ for all $p\in\NN$, and
\[
\forall p\in\NN,\;\exists h_\alpha^p\in\NN,\quad \gr^p_{F_{\irr,\alpha}}\cK^r_\alpha\simeq\cO_{\PP^1}(p)^{h_\alpha^p}.
\]
\end{proposition}

In other words, the filtration $F^\cbbullet_{\irr,\alpha}\cK^r_\alpha$ can be recovered by the mere datum of the \index{bundle!Kontsevich --}Kontsevich bundle $\cK^r_\alpha$, as being its \index{filtration!Harder-Narasimhan --}\emph{Harder-Narasimhan filtration}.

\subsubsection*{The singularity theory point of view: the Brieskorn-Deligne bundles}
Let us consider the free $\CC[u,u^{-1}]$-module
\[
\coH^r_\dR(U,f/u)\simeq \coH^r\bigl(\csov X,(\Omega_f^\cbbullet[u,u^{-1}],\rd+\rd f/u)\bigr)
\]
with its action of $\partial_u$, that we now regard as a connection. On the one hand, the \index{lattice!Brieskorn --}Brieskorn lattice
\begin{equation}\label{eq:Brieskornlattice}
G_r(U,f):=\coH^r\bigl(\csov X,(u^{-\cbbullet}\Omega_f^\cbbullet[u],\rd+\rd f/u)\bigr)
\end{equation}
is a free $\CC[u]$-module generating $\coH^r_\dR(U,f/u)$ as a $\CC[u,u^{-1}]$-module. On the other hand, since $\nabla$ has a regular singularity at $v=0$, there exists for each $\alpha\in\cA$ a free $\CC[v]$-submodule $V_\alpha(U,f/u)$ which generates it over $\CC[v,v^{-1}]=\CC[u^{-1},u]$, such that the connection $\nabla$ has a simple pole on it and its residue has eigenvalues contained in $[-\alpha,-\alpha+1)$ (it is called a \index{lattice!Deligne canonical --}\emph{Deligne canonical lattice} of $\coH^r_\dR(U,f/u)$). Since
\[
V_\alpha(U,f)|_{\Gm}=\coH^r_\dR(U,f/u)=G_r(U,f)|_{\Gm},
\]
we can glue $V_\alpha(U,f)$ with $G_r(U,f)$ to produce the \index{bundle!Brieskorn-Deligne --}\emph{Brieskorn-Deligne} vector bundle~$\cH^r_\alpha$ equipped with a connection having a pole of order one at $v=0$ and of order two at $u=0$.

\begin{theorem}
For each $\alpha\in\cA$, the bundles $\cK^r_\alpha$ and $\cH^r_\alpha$ are isomorphic.
\end{theorem}

It follows that the irregular Hodge filtration can be recovered from the Brieskorn-Deligne bundle as its \index{filtration!Harder-Narasimhan --}Harder-Narasimhan filtration.

\subsubsection*{Behavior on \texorpdfstring{$\Gm$}{Gm}}
One checks, in a way analogous to that of Theorem \ref{th:E1degalpha}, that the natural morphisms $\cK^r_\beta|_{\Gm}\csto\cK^r_\alpha|_{\Gm}\csto\coH_\dR^r(U,f/u)$ are isomorphisms for $\beta\leqslant\alpha\in\cA$, so that, by restricting to $\Gm$ the Harder-Narasimhan filtration of each $\cK^r_\alpha$, we obtain a family of filtrations $F_{\irr,\alpha}^\cbbullet\coH_\dR^r(U,f/u)$ indexed by $\alpha\in\cA$. Furthermore, for $\beta\leqslant\alpha\in\cA$, the quasi-isomorphism
\[
(\Omega_f^\cbbullet(\beta)[u,u^{-1}],\rd+\rd f/u)\csto(\Omega_f^\cbbullet(\alpha)[u,u^{-1}],\rd+\rd f/u)
\]
is filtered with respect to the filtration by stupid truncation, so that we obtain, for all $p\in\NN$, the inclusion
\[
F_{\irr,\beta}^p\coH_\dR^r(U,f/u)\subset F_{\irr,\alpha}^p\coH_\dR^r(U,f/u).
\]
We can thus regard the irregular Hodge filtration as indexed by $-\cA+\ZZ\subset\QQ$ by setting, for $\sfp=p-\alpha\in-\cA+\ZZ$,
\[
F_{\irr}^\sfp\coH_\dR^r(U,f/u):=F_{\irr,\alpha}^p\coH_\dR^r(U,f/u).
\]

\begin{proposition}\label{prop:grFirr}
The irregular Hodge filtration $F_{\irr}^\sfp\coH_\dR^r(U,f/u)$ is a filtration by subbundles, that is, each $\CC[u,u^{-1}]$-module
\[
\gr_{F_{\irr}}^\sfp\coH_\dR^r(U,f/u):=F_{\irr}^\sfp\coH_\dR^r(U,f/u)/F_{\irr}^{>\sfp}\coH_\dR^r(U,f/u)
\]
is free, where ${>}\sfp$ denotes the successor of~$\sfp$ in $-\cA+\ZZ$. 
\end{proposition}

\begin{proof}
From Proposition \ref{prop:HNK} one deduces that for each $\alpha\in\cA$, the filtration $F_{\irr,\alpha}^\cbbullet\coH_\dR^r(U,f/u)$ is a filtration by subbundles. For any $u_o\in\CC^*$, let $i_{u_o}:\{u_o\}\hto\nobreak\Gm$ denote the inclusion. Since $F_{\irr}^{>\sfp}\coH_\dR^r(U,f/u)$ is a subbundle of $F_{\irr}^{>\sfp-1}\coH_\dR^r(U,f/u)$, the composed morphism
\[
i_{u_o}^*F_{\irr}^{>\sfp}\coH_\dR^r(U,f/u)\csto i_{u_o}^*F_{\irr}^\sfp\coH_\dR^r(U,f/u)\csto i_{u_o}^*F_{\irr}^{>\sfp-1}\coH_\dR^r(U,f/u)
\]
is injective. Therefore, $i_{u_o}^*F_{\irr}^{>\sfp}\coH_\dR^r(U,f/u)\csto i_{u_o}^*F_{\irr}^\sfp\coH_\dR^r(U,f/u)$ is injective and the dimension of $i_{u_o}^*\gr_{F_{\irr}}^\sfp\coH_\dR^r(U,f/u)$ is the difference between the ranks of the bundles $F_{\irr}^\sfp\coH_\dR^r(U,f/u)$ and $F_{\irr}^{>\sfp}\coH_\dR^r(U,f/u)$, so is independent of $u_o$. It~follows that $\gr_{F_{\irr}}^\sfp\coH_\dR^r(U,f/u)$ is $\CC[u,u^{-1}]$-free.
\end{proof}

\subsubsection*{Limiting properties}
At the limit $u\csto\infty$ (\ie $v\csto0$) where the connection $\nabla$ has a regular singularity, the filtration $F_{\irr}^\cbbullet\coH_\dR^r(U,f/u)$ behaves in the same way as a the Hodge filtration of a polarizable Hodge module in the sense of \cite{MSaito86}. Furthermore, one can show that, for each $\alpha\in\cA$, the limit filtration of $F_{\irr,\alpha}^\cbbullet\coH_\dR^r(U,f/u)$ is the Hodge filtration of a \index{Hodge!mixed -- structure}mixed Hodge structure, isomorphic to the \index{filtration!limiting Hodge --}limiting Hodge filtration of the mixed Hodge structure on $\coH^r(U,f^{-1}(t);\CC)$ when $t\csto\infty$ on the generalized eigenspace corresponding the eigenvalue $\exp(-2\pi\sfi\alpha)$ of the monodromy.

On the other hand, the limit at $u=0$ is much less understood.

\subsubsection*{The spectrum at infinity}
An important Hodge invariant of a germ of holomorphic function is its spectrum, which is a set of pairs consisting of a rational number and a multiplicity, that one can encode by a polynomial with rational exponents and integral coefficients (\cf\eg\cite[\S9.7]{Steenbrink22}). We will provide a similar invariant for regular functions $f:U\csto\Af1$ and we will relate it with the irregular Hodge filtration.

The construction of $\cH^r_\alpha$ can be extended with indices $\alpha+q\in\cA+\ZZ$: we set
\[
\cH^r_{\alpha+q}=\cO_{\PP^1}(-q\cdot\{0\})\otimes\cH^r_\alpha\simeq \cO_{\PP^1}(-q)\otimes\cH^r_\alpha.
\]
In other words, $\cH^r_{\alpha+q}$ is obtained by gluing $G_r(U,f)$ with $V_{\alpha+q}(U,f)$. The family $(\cH^r_{\beta})_{\beta\in\cA+\ZZ}$ is an increasing family of locally free sheaves. By definition, we have
\[
\Gamma(\PP^1,\cH^r_{\beta})=V_{\beta}(U,f)\cap G_r(U,f)\subset\coH^r_\dR(U,f/u),
\]
and, for $\beta=\alpha+q$, multiplication by $u^{-q}$ induces an isomorphism
\[
V_{\beta}(U,f)\cap G_r(U,f)\simeq V_{\alpha}(U,f)\cap u^{-q}G_r(U,f).
\]
The following definition is analogous to that for germs of functions (\cf\eg\cite[(7.3)]{S-S85}, with a shift however):

\begin{definition}\label{def:spectrum}
The \index{spectrum at infinity}\emph{spectrum at infinity in degree $r$} of $(U,f)$ is the set of pairs $(\beta,\nu_\beta^r)$ with $\beta\in\cA+\ZZ$ and
\[
\nu_\beta^r=\dim\frac{V_{\beta}\cap G_r}{[V_{<\beta}\cap G_r]+[V_{\beta}\cap uG_r]}.
\]
It can be encoded in the generating polynomial $\sum_\beta\nu_\beta^r t^\beta$.
\end{definition}

\begin{proposition}\label{prop:spectrumFirr}
With the notation of Proposition \ref{prop:grFirr}, we have, for each \hbox{$\beta\in\QQ$}:
\[
\nu_\beta^r=\rk\gr_{F_\irr}^{-\beta}\coH^r_\dR(U,f/u).
\]
\end{proposition}

\subsection{Other cohomology theories}

We shortly introduce other cohomological structures that can be attached to a pair $(U,f)$ and indicate the relations, when they exist, with the cohomological structures considered in this section.

\subsubsection*{Exponential mixed Hodge structures}
This theory explains the notion of weight that we have introduced for $\coH^r(U,f;\QQ)$ and $\coH^r_\rc(U,f;\QQ)$. It makes use of the category $\MHM(\Af1)$ of \index{Hodge!mixed -- module}\emph{mixed Hodge modules} on the affine line $\Af1$ \cite{MSaito87} which is briefly summarized in \cite{Steenbrink22}. Any mixed Hodge module~$M$ on $\Af1$ has an associated $\QQ$-perverse sheaf~$\cF$. The category $\EMHS$ of \index{Hodge!exponential mixed -- structure}\emph{exponential mixed Hodge structures}, introduced by Kontsevich and Soibelman \cite{K-S10}, is the full subcategory of $\MHM(\Af1)$ whose objects~$M$ have an associated perverse sheaf $\cF$ with vanishing global hypercohomology. Furthermore, given any mixed Hodge module~$M$ on~$\Af1$, one can functorially associate with it a mixed Hodge module belonging to $\EMHS$. Each object of $\EMHS$ has a weight filtration, which can be distinct from the weight filtration as a mixed Hodge module. To each object is associated a ``Betti fiber'', which is nothing but that defined by Corollary \ref{cor:pervexact}, and inherits a weight filtration.

To the pair $(U,f)$ one associates the mixed Hodge modules whose associated \index{sheaf!perverse --}perverse sheaves are $\pR^{r-d}f_*(\pring_U)$ and $\pR^{r-d}f_!(\pring_U)$, and the corresponding objects in $\EMHS$, that one can write $\coH^r_{\sEMHS}(U,f)$ and $\coH^r_{\sEMHS,\rc}(U,f)$. The Betti fibers are nothing but $\coH^r(U,f;\QQ)$ and $\coH^r_\rc(U,f;\QQ)$ respectively, and the inherited weight filtration is that considered at the end of Section \ref{subsec:singcoh}.

Furthermore, to each object of $\EMHS$ is associated a de~Rham fiber, and for $\coH^r_{\sEMHS}(U,f)$ and $\coH^r_{\sEMHS,\rc}(U,f)$, the de~Rham fibers are nothing but $\coH^r_{\dR}(U,f)$ and $\coH^r_{\dR,\rc}(U,f)$ respectively. We can regard the irregular Hodge filtration as the Hodge part of an object of $\EMHS$.

The category $\EMHS$ is equipped with a tensor product, making it a $\QQ$-linear neutral \index{Tannakian category}Tannakian category. For $\coH^r_{\sEMHS}(U,f)$ and $\coH^r_{\sEMHS,\rc}(U,f)$, this reflects the \index{Thom-Sebastiani}Thom-Sebastiani property for the Betti and de~Rham fibers. Furthermore, the \index{filtration!weight --}weight and \index{filtration!irregular Hodge --}irregular Hodge filtrations on these respective fibers behave well under tensor product. A particular case was already observed in \cite{S-S85}.

The category of exponential Nori motives constructed in \cite{F-J17} regards the expo\-nential mixed Hodge structures $\coH^r_{\sEMHS}(U,f)$ and $\coH^r_{\sEMHS,\rc}(U,f)$ respectively as a realization of an object $[U,\csemptyset,f,r,0]$ and $[X,H,f_X,r,0]$ of a category $\mathbf{M}^{\exp}(\mathbbm{k})$ if~$U$ and~$f$ are defined over a subfield $\mathbbm{k}$ of $\CC$.

\subsubsection*{The motivic Milnor fiber at infinity}
Following the general formalism of Denef and Loeser for the motivic Milnor fiber, a \index{Milnor!motivic -- fiber}motivic Milnor fiber $S_{f,\infty}$ at infinity for $(U,f)$ has been constructed by Matsui-Takeuchi \cite[\S4]{M-T09} and by Raibaut \cite{Raibaut10, Raibaut12}. It is an object of the Grothendieck ring of varieties over $\Gm$ with a $\Gm$ action. It is constructed by means of a projectivization of $f$, but Raibaut has shown the independence of such a choice.

This object $S_{f,\infty}$ gives a quick access to the \index{spectrum at infinity}spectrum of $f$ at infinity and allows for a simple computation of it when, for example, $f$ is a \index{Laurent polynomial}Laurent polynomial which is nondegenerate with respect to its \index{Newton!polyhedron}\index{polyhedron!Newton --}Newton polyhedron (\cf Section \ref{subsec:Laurent} for this notion and \cite{Takeuchi23} for a survey on this topic).

\subsubsection*{Exponentially twisted cyclic homology}
From a categorical point of view, it is a standard question to ask how much a quasi-projective variety $U$ is determined from the bounded derived category $\catD^\rb(\mathrm{coh}(U))$ simply denoted by $\catD^\rb(U)$. In particular, let us assume that $U$ and $U'$ have Fourier-Mukai equivalent derived categories. This means that there exists a kernel, that is, an object $\cK$ of $\catD^\rb(U\times U')$, such that the integral transformation functor with kernel~$\cK$:
\[
\bR p'_*\Bigl(\bL p^*(\cbbullet)\otimes^L\cK\Bigl):\catD^\rb(U)\csto\catD^\rb(U')
\]
is an equivalence of categories. It follows from results by Keller \cite{Keller98} that, under such an assumption, the odd, \resp even, de~Rham cohomologies of $U$ and $U'$ are isomorphic.

Shklyarov \cite{Shklyarov17} has extended this kind of results to pairs $(U,f)$ with $U$ quasi-projective. Namely, given pairs $(U,f)$ and $(U',f')$, assume that the kernel $\cK$ is an object of $\catD^\rb(U\times_{\Af1} U')$ and that the associated Fourier-Mukai transformation is an equivalence of categories $\catD^\rb(U)\isom\catD^\rb(U')$. Then there exist isomorphisms between the odd \resp even $\CC[v]\langle\partial_v\rangle$-modules
\[
\csbigoplus_{r \mathrm{odd}/\mathrm{even}}\coH^r_\dR(U,vf)\simeq\csbigoplus_{r \mathrm{odd}/\mathrm{even}}\coH^r_\dR(U',vf').
\]

In both cases ($f=0$ with Keller and $f:U\csto\Af1$ with Shklyarov), the main step is an identification between the de~Rham cohomology and the \index{periodic cyclic homology}periodic cyclic homology associated to the pair \resp triple consisting of the category of perfect complexes on $U$ (bounded complexes quasi-isomorphic to bounded complexes of locally free $\cO_U$-modules) and its subcategory of acyclic such complexes \resp and the regular function~$f$.

\section{Tameness on smooth affine varieties}\label{sec:tameness}

\subsubsection*{Definitions of tameness}
In the global setting, \index{tameness}\emph{tameness} of $f:U\csto\Af1$ is a property that is intended to be the analogue of having an isolated critical point for a germ of holomorphic function. It~consists of two conditions:
\begin{itemize}
\item
each critical point of $f$ is isolated (since $f$ is a regular function on $U$ quasi-projective, this implies that the set of critical points of $f$ is finite),
\item
``infinity is not a critical point for $f$''.
\end{itemize}
The \emph{global Milnor number $\mu(f)$} is the sum, over all critical points $x\in U$ of $f$, of the corresponding Milnor numbers $\mu(f,x)$.

One should not expect a unique definition of the second property, but instead various criteria adapted to each particular class of examples.

Let us list some of them. We start ``from inside''.
\begin{enumerate}
\item\label{ex:broughton}
$f:\Af{d}\csto \Af1$ is tame in the sense of Broughton (\cf\cite[Lem.\,4.3]{Broughton88}) if there exists $\varepsilon>0$ and a compact $K\subset U$ such that $\|\partial f\|\geqslant \varepsilon$ out of $K$.

\item\label{ex:malgrange}
Tameness in the sense of Malgrange means that, when $f(x)$ remains bounded, there exists $\varepsilon>0$ such that $\|x\|\|\partial f(x)\|\geqslant \varepsilon$ for $\|x\|$ sufficiently large.

\item\label{ex:nemethi}
In \cite{N-Z92} is introduced the notion of M-tameness for such a polynomial, which gives global Milnor fibrations in big balls in $\Af{d}$, that covers both previous situations (the notion of quasi-tameness was also previously considered in \cite{Nemethi86}). One can extend this notion to any function $f$ on a smooth affine variety by embedding this variety in an affine space and by considering the induced metric on it (\cf\cite{D-S02a}).

\item\label{ex:KouchGLN}
A specific class of examples of Item \ref{ex:broughton} consists of \index{convenient (Laurent) polynomial}convenient and \index{nondegenerate (Laurent) polynomial}nondegenerate polynomials with respect to their Newton polyhedron at infinity, in a given coordinate system of $\Af{n}$, as defined by Kouchnirenko \cite{Kouchnirenko76}. Their toric analogues (Laurent polynomials) will be considered in Section \ref{subsec:Laurent}. On the other hand, García López and Némethi \cite{G-N95} have defined a nice class of polynomials that are quasi-tame.
\end{enumerate}

For our purposes, tameness will be considered from a cohomological point of view. We fix the ring $\csring$ to be $\ZZ$ or to be one of the fields $\QQ,\RR,\CC$.

\begin{enumerate}\setcounter{enumi}{4}
\item\label{ex:Katz}
A cohomological condition was introduced by Katz in the arithmetic $\ell$\nobreakdash-adic setting \cite[Prop.\,14.13.3(2)]{Katz90}: let $C^\cbbullet$ denote the cone of the natural morphism of complexes $\bR f_!\csring_U\csto\bR f_*\csring_U$ on the affine line~$\Af1$; by definition, we have a long exact sequence
\[
\cdots\csto R^kf_!\csring_U\csto R^kf_*\csring_U\csto\cH^k(C^\cbbullet)\csto R^{k+1}f_!\csring_U\csto\cdots;
\]
in general, $C^\cbbullet$ has constructible cohomology; the condition is that all cohomology sheaves $\cH^k(C^\cbbullet)$ are constant sheaves of $\csring$-modules of finite type on the affine line $\Af1$.

It an be shown that the geometric tameness conditions considered in \eqref{ex:broughton}--\eqref{ex:KouchGLN} all imply the Katz tameness condition with $\csring=\ZZ$.
\end{enumerate}

Tameness ``from outside'' is defined for a suitable choice of projectivization $f_X:X\csto\Af1$ of $f$. For a polynomial $f$, the standard partial compactification of the graph $U\subset\PP^d\times \Af1$ of $f$ is natural to consider, but others can also prove useful.

\begin{enumerate}\setcounter{enumi}{5}
\item\label{ex:cohtame}
We say (\cf \cite{Bibi96bb}) that $f:U\csto\Af1$ is \emph{cohomologically $\csring$-tame} if there exists a projectivization of $f$ given by a commutative diagram
\[
\xymatrix{
U\ar@{^{ (}->}[r]^-{j}\ar[rd]_(.35){f}& X\ar[d]^-{f_X}\\
&\Af1
}
\]
where $X$ is quasi-projective and $f_X$ is projective, such that both complexes $\bR j_*\csring_U$ and $\bR j_!\csring_U$ have no \index{vanishing cycles!at infinity}vanishing cycles at infinity, namely, for all $c\in\Af1$, the vanishing cycle complexes $\phi_{f_X-c}(\bR j_*\csring_U)$ and $\phi_{f_X-c}(\bR j_!\csring_U)$ are supported in at most a finite number of points, all at finite distance, \ie contained in $U$. Note that, if $\csring$ is a field, Verdier duality yields that the condition on one complex implies the condition on the other one.

\item
An example of a cohomologically $\QQ$-tame function on an affine open subset of~$\Af3$ which is not a torus $(\Gm)^3$ is the rational function
\[
f(x,y,z)=y(2+z)+\frac{x^2}{(xy-1)(1+z)}
\]
which is regular on $U=\Af3\moins\{(xy-1)(1+z)=0\}$. Its tameness is proved in \cite[App.]{G-S15}, where the authors propose such a function as a Landau-Ginzburg model for the mirror of a $3$-dimensional quadric.
\end{enumerate}

For polynomial functions on the affine space, a relation between ``inside'' and ``outside'' tameness does exist. In fact, Parusi\'nski has shown \cite{Parusinski95} that $f:\Af{d}\csto \Af1$ is cohomologically tame with respect to the standard partial compactification of the graph $U\subset\PP^d\times \Af1$ of $f$ if and only if it satisfies the Malgrange condition.

\subsubsection*{Purpose of this section}
Tame functions provide a framework where the general properties developed in Section \ref{sec:cohUf} acquire a simpler expression and emphasize purity (in the sense of Hodge theory). This is parallel to properties of isolated singularities of holomorphic functions. However, the algebraic setting allows more specific properties, like those related to positivity.

\subsection{Cohomological properties of tame functions}

\emph{For the sake of simplicity, we assume in the remaining part of this section \ref{sec:tameness} that~$U$ is affine of dimension $d\geqslant2$.}

We will make use of the $\ZZ$-tameness condition of Katz, so we first make precise an equivalent definition when $U$ is affine, and its relation with other tameness conditions.

\begin{proposition}\label{propimdirperv}
The function $f:U\csto\Af1$ is $\csring$-tame in the sense of Katz if and only if the following properties are satisfied:
\begin{enumerate}
\item\label{propimdirperv1}
the perverse complex $\pR^{r-d}\!f_*(\pring_U)$ is zero for $r>d$ and, if $r<d$, is the constant sheaf shifted by one with fiber isomorphic to $\coH^{r-1}(U,\csring)$ on the affine line~$\Af1$;
\item\label{propimdirperv1b}
the perverse complex $\pR^{r-d}\!f_!(\pring_U)$ is zero for $r<d$ and, if $r>d$, is the constant sheaf shifted by one with fiber isomorphic to $\coH^{r-1}_\rc(U,\csring)$ on the affine line $\Af1$;
\item\label{propimdirperv2}
The kernel and the cokernel (in the perverse sense) of the natural morphism $\pR^0\!f_!\pring_U\csto \pR^0\!f_*\pring_U$ are constant sheaves shifted by one (with the same rank) on~$\Af1$.
\end{enumerate}
Furthermore, if $f$ is cohomologically $\csring$-tame, or if $f$ is M-tame, then $f$ is $\csring$-tame in the sense of Katz.
\end{proposition}

\begin{proof}[Sketch]
The notion of micro-support is useful for explaining the latter property \hbox{(\cf \cite{K-S90,Dimca04,M-S22})}: Katz' condition is equivalent to the property that the micro-support of the cone $\cC$ of the natural morphism $\bR f_!\pring_U\csto\bR f_*\pring_U$ is reduced to the zero section of the cotangent bundle of $\Af1$, and this, in turn, is equivalent to the property that its perverse cohomology sheaves are constant. The long exact sequence of perverse cohomology attached to the corresponding distinguished triangle, together with Artin's vanishing result (\cf \cite[Prop.\,10.3.17]{K-S90} or \cite[Prop.\,5.2.13]{Dimca04}), implies that the assertions of Proposition \ref{propimdirperv} are satisfied, and the converse holds in the same way.

It remains to identify the constant local systems. This is obtained by analyzing the perverse Leray spectral sequence
\[
E_2^{p,q}=\coH^p(\Af1,\pR^qf_*\pring_U)\implies \coH^{p+q}(U,\pring_U),
\]
and the compact support analogue with $\pR^qf_!\pring_U$. Due to the vanishing and constancy statements in Items \ref{propimdirperv1} and \ref{propimdirperv2}, one obtains that it degenerates at $E_2$. Then for $r<d$, the fiber of the shifted constant sheaf $\pR^{r-d}f_*\pring_U$, which is isomorphic to $\coH^{-1}(\Af1,\pR^{r-d}f_*\pring_U)$, is thus isomorphic to $\coH^{r-1}(U,\csring_U)$, and we have a similar argument for Item \ref{propimdirperv2}.

A proof for the last assertions can be found in \cite{D-S02a} for M-tameness and in \cite{Bibi97b} for cohomological tameness.
\end{proof}

This is an example where manipulating perverse cohomology instead of usual cohomology leads to simpler and clearer results. Let us mention a consequence of the theory of Hodge modules (\cf\cite{MSaito86,MSaito87}), due to purity of the image $\pR^0\!f_{{!}{*}}\pQQ_U$ of the natural morphism $\pR^0\!f_!\pQQ_U\csto \pR^0\!f_*\pQQ_U$.

\begin{corollary}\label{cor:semisimple}
The perverse complex $\pR^0\!f_{{!}{*}}\pQQ_U$ is semi-simple as such; in particular, the shifted local system on the complement of its singularities is a direct sum of irreducible local systems.
\end{corollary}

Let $C(f)$ be the set of critical values of $f$ (which coincides with the set of typical and atypical values in the tame case, at least for what concerns cohomology). For $t\notin C(f)$, the cohomology $\coH^{d-1}_\mathrm{mid}(f^{-1}(t);\QQ)$ (image of $\coH^{d-1}_\rc$ in $\coH^{d-1}$) is the stalk at~$t$ of a \index{sheaf!locally constant --}locally constant sheaf on $\Af1\moins C(f)$, and corresponds to a linear representation of $\pi_1(\Af1\moins C(f),\star)$ for some base point $\star$. The second part of Corollary \ref{cor:semisimple} asserts that this representation is a direct sum of \emph{irreducible representations}.

We now consider the \index{vanishing cycles!global --}global vanishing cycles.

\begin{theorem}\label{th:tame}
Let $f:U\csto\Af1$ be an $\csring$-tame function (in the sense of Katz, say). Then the local systems of global vanishing cycle $\cswt\Phi_\infty^r(U,f;\csring)$ and $\cswt\Phi_{\infty,\rc}^r(U,f;\csring)$ (\cf Definition \ref{def:globalvanishingcycles}) are zero for $r\neq d$ and are $\csring$-free of rank $\mu(f)$ if $r=d$. Furthermore, the natural morphism $\cswt\Phi_{\infty,\rc}^d(U,f;\csring)\csto\cswt\Phi_\infty^d(U,f;\csring)$ forgetting supports is an isomorphism. If~$\csring=\QQ$, their fibers at any $\theta$ are pure of weight $d$.
\end{theorem}

\begin{proof}
The last assertion follows from Corollary \ref{cor:weights}. The vanishing property and the isomorphism property are consequences of properties of the perverse sheaves described in Proposition \ref{propimdirperv}, according to the expressions \eqref{eq:Ufpush} and the vanishing property $\coH^k(\wAfu_\rmod;\cF)=0$ for all $k$ if $\cF$ is constant (Lemma \ref{lem:vanishing}).

The freeness property is local on $S^1$, so we can identify $\cswt\Phi_\infty^d(U,f;\csring)$ with $\csbigoplus_{c\in C}\cswt\phip_{f-c}\pring_U$. Since the Milnor fiber of $f$ at a critical point $x\in f^{-1}(c)$ is a bouquet of $\mu(f,x)$ spheres, each $\cswt\phip_{f-c}\pring_U$ is locally $\csring$-free of rank $\sum_{x\in f^{-1}(c)}\mu(f,x)$.
\end{proof}

The pairing \eqref{eq:dualHUfS1} is non-trivial in degree $d$ only and reads
\[
\varh:\cswt\Phi_\infty^d(U,f)\otimes\cswt\Phi_{\infty}^d(U,-f)\csto \csring_{S^1}.
\]
In the local situation, each $\varh_c$ is known to be nondegenerate (it is isomorphic to the Seifert pairing, if one identifies $\cswt\Phi_{\infty}^d(U,-f)$ with $\iota^{-1}\cswt\Phi_{\infty}^d(U,f)$). Then Proposition \ref{prop:pairingF}\eqref{prop:pairingF2} implies that $\varh$ is nondegenerate. It follows from Proposition \ref{prop:hS} that, in suitable bases of $\coH^d(U,f;\csring)$ and $\coH^d(U,-f;\csring)$, the matrix of $\varh$ is equal to the Stokes matrix (\cf Proposition \ref{prop:hS}).

If Poincaré duality holds with $\csring$-coefficients on $U^\an$ (\eg for $\csring=\ZZ$, $U^\an$ can be the product $\Af{k}\times\Gm^{d-k}$), then from the long exact sequence of relative cohomology of the pair $(U,f^{-1}(t))$ for $t\notin C$ one concludes that the cohomology of $f^{-1}(t)$ is $\csring$-free and, from Lemma \ref{lem:sperv}, that $\pR^0\!f_*\pring_U$ is strongly perverse.

\subsection{The spectrum at infinity}

\subsubsection*{The twisted de~Rham complex and the Brieskorn lattice}

By the comparison theorem \ref{th:isoBdR} and Theorem \ref{th:tame}, we find that, for $f$ tame in the sense of Katz, $\coH^r_{\dR}(U,f)$ and $\coH^r_{\dR,\rc}(U,f)$ are zero for $r\neq d$ and that both coincide if $r=d$, with $\dim\coH^d_{\dR}(U,f)=\mu(f)$.

The Brieskorn lattices (\cf Theorem \ref{th:A}) thus vanish for $r\neq d$ and the Brieskorn lattice in degree $d$, that we denote by $G_d(U,f)\subset G(U,f):=\coH^d_{\dR}(U,f/u)=\coH^d_{\dR,\rc}(U,f/u)$, is~a free $\CC[u]$-module of rank $\mu(f)$.

When $f$ is tame, the \index{lattice!Brieskorn --}Brieskorn lattice can be computed with a twisted de~Rham complex on $U$ (instead of a logarithmic de~Rham complex on $X$). Furthermore, as~$U$ is affine, one can compute the hypercohomology of the twisted de~Rham complex (over $\CC[u]$ or $\CC[u,u^{-1}]$) as the cohomology of the complex of global sections on $U$. In~such a case, it is more convenient to deal with the lattice
\[
G_0(U,f):=u^dG_d(U,f),
\]
that is usually called the \index{lattice!Brieskorn --}\emph{Brieskorn lattice of $f$}. We~have (\cf \cite[Prop.\,8.13]{S-Y14}):

\begin{theorem}
If $f:U\csto\Af1$ is tame (in the sense of Katz, say), then the inclusion $G_0(U,f)\subset G(U,f)$ reads
\[
G_0(U,f)=\frac{\Omega^d(U)[u]}{(u\rd+\rd f)\Omega^{d-1}(U)[u]}\hto \frac{\Omega^d(U)[u,u^{-1}]}{(u\rd+\rd f)\Omega^{d-1}(U)[u,u^{-1}]}.
\]
\end{theorem}

\subsubsection*{The spectrum at infinity}
The \index{spectrum at infinity}spectrum of $G_0(U,f)$ is the set of pairs $(\beta,\nu_\beta)$ computed with $G_0(U,f)$ instead of $G_d(U,f)$, so that $\nu_{\beta}=\nu_{\beta-d}^d$ and therefore, according to Proposition \ref{prop:spectrumFirr}, the relation with the irregular Hodge filtration of $\coH^d_\dR(U,f)$ is
\begin{equation}\label{eq:spFirr}
\nu_\beta=\rk\gr_{F_\irr}^{d-\beta}\coH^d_\dR(U,f).
\end{equation}
The perfect pairing compatible with $\partial_u$ of Theorem \ref{th:A}\eqref{th:A2} reads (due to the definition above of $G_0(U,f)$)
\[
G_0(U,f)\otimes_{\CC[u]}G_0(U,-f)\csto u^d\CC[u],
\]
and we can identify $G_0(U,-f)$ with the pullback $\iota^*G_0(U,f)$ by the involution $\iota:u\mto -u$, that we also denote by $\csov G_0(U,f)$. This is $G_0(U,f)$ as a $\CC$-vector space, and the action of $\CC[u]\langle u^2\partial_u\rangle$ is defined as follows: if $g\in G_0(U,f)$, we denote it by $\csov g$ when considering it as an element of $\csov G_0(U,f)$; we then set $P(u,u^2\partial_u)\cdot\csov g:=\csov{P(-u,-u^2\partial_u)\cdot g}$ in $\csov G_0(U,f)$. The spectrum of $\csov G_0(U,f)$ coincides with that of $G_0(U,f)$.

\begin{corollary}\label{cor:symspectrum}
If $f:U\csto\Af1$ is tame, the spectrum of $f$ at infinity is symmetric with respect to $d/2$ and is contained in $[0,d]$.
\end{corollary}

\begin{proof}
The first assertion follows from the behavior of spectrum with respect to duality (\cf\eg\cite[Prop.\,3.3]{Bibi97b}) or also from the behavior of the irregular Hodge filtration with respect to duality provided by Theorem \ref{th:yuduality} and the identification \eqref{eq:spFirr}. The second assertion follows similarly from the remark after Theorem \ref{th:yuduality} (another argument has been given in \cite{Bibi97b}).
\end{proof}

\subsection{An example: Laurent polynomials}\label{subsec:Laurent}
In this section, we consider the case where the affine variety $U$ is the torus $(\Gm)^d$. If~we choose coordinates $x_1,\dots,x_d$, then~$f$ is a Laurent polynomial in $\CC[x_1^{\pm1},\dots,x_d^{\pm1}]$. Let us write $f=\sum_{m\in\ZZ^d}a_mx^m$. The \index{Newton!polyhedron}\index{polyhedron!Newton --}\emph{Newton polyhedron} $\Delta(f)$ is the convex hull in~$\RR^d$ of the set $\{0\}\cup\{m\in\ZZ^d\mid\nobreak a_m\neq\nobreak0\}$. We say that $f$ is \index{nondegenerate (Laurent) polynomial}\emph{nondegenerate with respect to $\Delta(f)$} if for each face $\sigma$ of $\Delta(f)$ not containing $0$, setting $f_\sigma=\sum_{m\in\sigma}a_mx^m$, the functions
\[
f_ \sigma,\frac{\partial f_ \sigma}{\partial x_1},\dots, \frac{\partial f_ \sigma}{\partial x_d}
\]
do not vanish simultaneously on $U$. We say that $f$ is \index{convenient (Laurent) polynomial}convenient if $0$ belongs to the interior of $\Delta(f)$. In such a case, an argument similar to that of Broughton \cite{Broughton88} shows that $f$ is M-tame, and therefore tame in the sense of Katz. It follows from \cite{Kouchnirenko76} that
\[
\mu(f)=\dim\coH^d(U,f;\QQ)=d!\mathrm{Vol}_d(\Delta(f)).
\]

\subsubsection*{The Newton filtration and the Brieskorn lattice}

Let us assume that $f$ convenient and nondegenerate with respect to $\Delta(f)$. For each face $\sigma$ of dimension $d-1$ of the boundary $\partial\Delta(f)$, we denote by $L_\sigma$ the linear form with coefficients in~$\QQ$ such that $L_\sigma\equiv1$ on $\sigma$. For $g\in\CC[x,x^{-1}]$, we set $\deg_\sigma(g)=\max_mL_\sigma(m)$, where the max is taken on the exponents $m$ of monomials appearing in $g$, and $\deg(g)=\max_\sigma\deg_\sigma(g)$.

\begin{remark}\label{rem:Newton}\mbox{}
\begin{enumerate}
\item\label{rem:Newton2}
For $g,h\in\CC[x,x^{-1}]$, we have $\deg(gh)\leqslant\deg(g)+\deg(h)$ with equality if and only if there exists a face $\sigma$ such that $\deg(g)=\deg_\sigma(g)$ and $\deg(h)=\deg_\sigma(h)$.
\item\label{rem:Newton3}
As $0$ belongs to the interior of $\Delta(f)$, we have $\deg(g)\geqslant0$ for any $g\in\CC[x,x^{-1}]$ and $\deg(g)=0$ if and only if $g\in\CC$. This would not remain true without this assumption of convenience.
\item
Let us set $\sfrac{\rd x}{x}=\Psfrac{\rd x_1}{x_1}\wedge\cdots\wedge\Psfrac{\rd x_d}{x_d}$. If $\omega\in\Omega^d(U)$, we write $\omega=g\rd x/x$ and we define $\delta(\omega)$ to be $\delta(g)$.
\end{enumerate}
\end{remark}

The increasing \index{filtration!Newton}\index{Newton!filtration}\emph{Newton filtration} $\cN_\bbullet\Omega^d(U)$ indexed by $\QQ$ is defined by\vspace*{-3pt}
\[
\cN_\beta\Omega^d(U):=\{g\rd x/x\in\Omega^d(U)\mid\deg(g)\leqslant\beta\}.
\]
The previous remark shows that $\cN_\beta\Omega^d(U)\!=\!0$ for $\beta\!<\!0$ and \hbox{$\cN_0\Omega^d(U)\!=\!\CC \rd x/x$}. We~extend this filtration to $\Omega^d(U)[u]$ by setting\vspace*{-3pt}
\[
\cN_\beta\Omega^d(U)[u]:=\cN_\beta\Omega^d(U)+u\cN_{\beta-1}\Omega^d(U)+\cdots+u^k\cN_{\beta-k}\Omega^d(U)+\cdots
\]
and we induce it on the \index{lattice!Brieskorn --}Brieskorn lattice $G_0(U,f)$:

\begin{definition}
The Newton filtration of the Brieskorn lattice is defined by
\begin{align*}
\cN_\beta G_0&= \image\bigl[\cN_\beta\Omega^d(U)[u]\hookrightarrow\Omega^d(U)[u] \rightarrow G_0\bigr]\\
&=\cN_\beta\Omega^d(U)[u]\big/\bigl((u\rd+\rd f)\Omega^{d-1}(U)[u]\cap \cN_\beta\Omega^d(U)[u]\bigr).
\end{align*}
\end{definition}

Recall the family $(V_\beta(U,f/u))_{\beta\in-\cA+\ZZ}$ of $\CC[u^{-1}]$-submodules of $\coH^d_\dR(U,f)$ such that the residue of the connection at the origin has eigenvalues in $[-\beta,-\beta+1)$ (\cf Section~\ref{subsec:irrHodge}). The following theorem allows a combinatorial computation of the spectrum of~$f$ at infinity.

\begin{theorem}[{\cite[Th.\,4.5]{D-S02a}}]\label{th:Newton}
Assume that $f$ is convenient and nondegenerate with respect to its Newton polyhedron. Then the Newton filtration $\cN_\bbullet G_0$ of the Brieskorn lattice coincides with the filtration $V_\bbullet(U,f/u)\cap G_0$.
\end{theorem}

If we define the \index{Newton!spectrum}\emph{Newton spectrum} in a way similar to that of Definition \ref{def:spectrum} by using the Brieskorn lattice and the Newton filtration, we conclude that the Newton spectrum satisfies the symmetry and bound properties like in Corollary \ref{cor:symspectrum}, a~property that is far from obvious a priori.

\begin{remark}
In \cite{Douai20}, one finds various combinatorial properties of the Hodge numbers and the spectrum attached to a convenient nondegenerate Laurent polynomial.
\end{remark}

\subsubsection*{Application to a conjecture of Katzarkov-Kontsevich-Pantev \texorpdfstring{\cite{K-K-P14}}{cite}}

(See also \cite{Shamoto17, Harder17, Przhiyalkovski18}.)
Let $\Delta\subset\RR^d$ be a simplicial polyhedron with vertices in $\ZZ^d$, having $0$ in its interior, and which is \index{polyhedron!reflexive --}\emph{reflexive} in the sense of \cite[\S4.1]{Batyrev94}, that is, any linear form $L_\sigma$ considered above has coefficients in $\ZZ$ which are coprime. It is known (\cf\loccit) that $0$ is then the only integral point in the interior of $\Delta$ and that the dual polyhedron $\Delta^\csvee=\{\ell\in(\RR^d)^\csvee\mid\ell(m)\geqslant-1,\,\forall m\in\Delta\}$ is also reflexive. Let~$\Sigma$ be the cone on $\Delta^\csvee$ with apex $0$, which is also the dual fan of $\Delta$. It determines a toric variety $X$. Let us assume that $X$ is non singular (\ie each $(d-1)$-dimensional face of $\Delta^\csvee$ is a simplex, the vertices of which form a basis of $\ZZ^d$.

Let $V(\Delta^\csvee)$ be the set of vertices of $\Delta^\csvee$ and, for $\bma=(a_v)_{v\in V(\Delta^\csvee)}\in(\CC^*)^{V(\Delta^\csvee)}$, let us consider the Laurent polynomial $f_{\bma}=\sum_{v\in V(\Delta^\csvee)}a_vx^v\in\CC[x_1^{\pm1},\cdots,x_d^{\pm1}]$, that we consider as a regular function on $U=(\Gm)^d$. We have $\Delta(f)=\Delta^\csvee$ and $f_{\bma}$ is convenient and nondegenerate. Reflexivity implies that the jumps of the Newton filtration $\cN_\bbullet\Omega^d(U)$ are integer, hence so are the spectral numbers and the irregular Hodge numbers, as well as the jumps of the $V$-filtration $V_\bbullet(U,f_{\bma})\subset\coH^d_\dR(U,f_{\bma}/u)$ considered in Section \ref{subsec:irrHodge}. One also interprets this integrality property as the unipotency of the monodromy operator on the cohomology operator on $\coH^d(U,f_{\bma}^{-1}(t);\QQ)$ with respect to $t\mto\rme^{\sfi\theta}t$ for $t\gg0$ and $\rme^{\sfi\theta}\in S^1$. We denote by $W_\bbullet\coH^d(U,f_{\bma}^{-1}(t);\QQ)$ the monodromy weight filtration centered at $d$ associated to the nilpotent part of the monodromy (\cf\cite[p.\,661]{Steenbrink22}).\footnote{It should not be confused with the weight filtration in Corollary \ref{cor:weights}, from which it is in fact a limit.}

A toric special case of a conjecture of Katzarkov-Kontsevich-Pantev \cite{K-K-P14} (still not solved in general, \cf \cite{L-P16} for an overview) asserts:\vspace*{-3pt}
\begin{equation}\label{eq:KKP}
\forall p\geqslant0,\quad\dim\gr^p_{F_\irr}\coH^d(U,f_{\bma})=\gr_{2p}^W\coH^d(U,f_{\bma}^{-1}(t);\QQ).
\end{equation}

\begin{theorem}[\cf\cite{Bibi17}]
The previous equality holds for any $\bma\in(\CC^*)^{V(\Delta^\csvee)}$.
\end{theorem}

\begin{proof}[Sketch]
We only indicate the main ideas of the proof, showing the interplay between algebraic geometry and singularity theory.
\begin{enumerate}
\item
From Kouchnirenko's formula, $\mu(f_{\bma})$ is independent of $\bma$.
\item
A first step is to show that, separately, the left-hand side and the right-hand side of \eqref{eq:KKP} are independent of $\bma$ as long as no entry of $\bma$ vanishes. For the left-hand side, according to \eqref{eq:spFirr}, one can argue with the spectrum, and one shows that the spectrum at infinity of $f_{\bma}$ is independent of $\bma$, due to the semi-continuity property of \cite{N-S97} and the constancy of $\mu(f_{\bma})$. For the right-hand side, one shows that the family $\coH^d_{\dR}(U,f_{\bma}/u)$ is an algebraic vector bundle on $\Gm\times(\Gm)^{V(\Delta^\csvee)}$ with a flat connection having regular singularities along $\{0\}\times(\Gm)^{V(\Delta^\csvee)}$.
\item
The second step is to prove the result when $a_v=1$ for all $v\in V(\Delta^*)$. We set $f=f_\mathbf{1}$. One knows from \cite{Batyrev94} that $X$ is a smooth Fano projective variety (\ie the canonical bundle $K_X$ is anti-ample), and (\cf\cite{Fulton93}), one can interpret $f$ as the first Chern class $c_1(TX)$ of $X$. In particular, by anti-ampleness of $K_X$, the cup product with $c_1(TX)$ satisfies the Hard Lefschetz property on the $\QQ$-Chow ring of $X$. Due to a formula of Borisov-Chen-Smith \cite{B-C-S05}, this Chow ring is identified with the graded ring with respect to the Newton filtration:
\[
\gr_\bbullet^\cN(\QQ[x,x^{-1}]/(\partial f).
\]
Then one identifies multiplication by $c_1(TX)=f$ on this ring tensored by $\CC$ with the action of the monodromy on $\coH^d_\dR(U,f)$, by an argument similar to that used by Varchenko \cite{Varchenko81}. The Hard Lefschetz property implies then the desired equality~\eqref{eq:KKP}.
\end{enumerate}
\end{proof}

\begin{acknowledgement}
The author wishes to thank José Seade for inviting them to contribute to this volume and the anonymous referee for his careful reading of the manuscript and for his comments.
\end{acknowledgement}

\providecommand{\sortnoop}[1]{}\providecommand{\eprint}[1]{\href{http://arxiv.org/abs/#1}{\texttt{arXiv\string:\allowbreak#1}}}

\begin{theindex}
\mtaddtocont{\protect\contentsline{mtchap}{\indexname}{\thepage}\hyperhrefextend}

  \item bundle
    \subitem Brieskorn-Deligne --, \hyperpage{48}
    \subitem Kontsevich --, \hyperpage{47, 48}

  \indexspace

  \item cohomology
    \subitem algebraic de~Rham --, \hyperpage{21}
    \subitem de~Rham -- with compact support, \hyperpage{22}
  \item complex
    \subitem algebraic de~Rham --, \hyperpage{21}
    \subitem algebraically constructible --, \hyperpage{6}
    \subitem Kontsevich --, \hyperpage{22}
    \subitem logarithmic de~Rham --, \hyperpage{21}
    \subitem of nearby cycles, \hyperpage{7}
    \subitem of vanishing cycles, \hyperpage{7}
    \subitem perverse --, \hyperpage{6}
    \subitem strongly perverse --, \hyperpage{7}, \hyperpage{33}
    \subitem twisted algebraic de~Rham --, \hyperpage{21}
    \subitem twisted algebraic de~Rham -- with parameter, 
		\hyperpage{41}
    \subitem twisted meromorphic de~Rham --, \hyperpage{25}
  \item convenient (Laurent) polynomial, \hyperpage{53}, \hyperpage{58}

  \indexspace

  \item duality
    \subitem de~Rham --, \hyperpage{22}
    \subitem Poincaré-Verdier --, \hyperpage{6}

  \indexspace

  \item exponential growth, \hyperpage{15}

  \indexspace

  \item filtration
    \subitem Harder-Narasimhan --, \hyperpage{48}
    \subitem Hodge --, \hyperpage{24}
    \subitem indexed by $\QQ$, \hyperpage{25}
    \subitem indexed by $C$, \hyperpage{30}, \hyperpage{34}
    \subitem irregular Hodge --, \hyperpage{25}, \hyperpage{51}
    \subitem limiting Hodge --, \hyperpage{50}
    \subitem Newton, \hyperpage{59}
    \subitem opposite --, \hyperpage{36}
    \subitem Stokes --, \hyperpage{30}, \hyperpage{33, 34}
    \subitem weight --, \hyperpage{3}, \hyperpage{19, 20}, 
		\hyperpage{27}, \hyperpage{51}

  \indexspace

  \item Hodge
    \subitem exponential mixed -- structure, \hyperpage{51}
    \subitem filtration, \hyperpage{24}
    \subitem irregular -- numbers, \hyperpage{27}
    \subitem irregular -- structure, \hyperpage{4}, \hyperpage{10}
    \subitem mixed -- module, \hyperpage{51}
    \subitem mixed -- structure, \hyperpage{3}, \hyperpage{9}, 
		\hyperpage{24}, \hyperpage{50}
    \subitem numbers, \hyperpage{27}
    \subitem Tate-twisted -- structure, \hyperpage{27}

  \indexspace

  \item lattice
    \subitem Brieskorn --, \hyperpage{3}, \hyperpage{44}, 
		\hyperpage{48}, \hyperpage{57}, \hyperpage{59}
    \subitem Deligne canonical --, \hyperpage{48}
  \item Laurent polynomial, \hyperpage{52}

  \indexspace

  \item Milnor
    \subitem ball, \hyperpage{2, 3}, \hyperpage{17}
    \subitem motivic -- fiber, \hyperpage{51}
  \item moderate growth, \hyperpage{15}

  \indexspace

  \item Newton
    \subitem filtration, \hyperpage{59}
    \subitem polyhedron, \hyperpage{52}, \hyperpage{58}
    \subitem spectrum, \hyperpage{59}
  \item nondegenerate (Laurent) polynomial, \hyperpage{53}, 
		\hyperpage{58}

  \indexspace

  \item pairing
    \subitem de~Rham --, \hyperpage{22}, \hyperpage{26}
    \subitem intersection --, \hyperpage{21}
    \subitem of perverse sheaves, \hyperpage{40}
    \subitem of Stokes-filtered local systems, \hyperpage{38}
    \subitem period --, \hyperpage{21}
    \subitem Poincaré --, \hyperpage{18}
    \subitem Seifert --, \hyperpage{39}
  \item period, \hyperpage{10}
  \item periodic cyclic homology, \hyperpage{52}
  \item polyhedron
    \subitem Newton --, \hyperpage{52}, \hyperpage{58}
    \subitem reflexive --, \hyperpage{59}
  \item projectivization, \hyperpage{5}
  \item pushforward, \hyperpage{6}
    \subitem with proper support, \hyperpage{6}

  \indexspace

  \item rapid decay, \hyperpage{15}
  \item real oriented blow-up, \hyperpage{12}, \hyperpage{14}

  \indexspace

  \item set
    \subitem bifurcation --, \hyperpage{5}
    \subitem critical --, \hyperpage{4}
    \subitem of atypical critical values, \hyperpage{5}
    \subitem of critical values, \hyperpage{4}
  \item sheaf
    \subitem locally constant --, \hyperpage{6}, \hyperpage{8}, 
		\hyperpage{11}, \hyperpage{29}, \hyperpage{31}, 
		\hyperpage{33, 34}, \hyperpage{42}, \hyperpage{56}
    \subitem of global vanishing cycles, \hyperpage{45}
    \subitem perverse --, \hyperpage{6, 7}, \hyperpage{13}, 
		\hyperpage{51}
    \subitem perverse cohomology --, \hyperpage{9}, \hyperpage{11}, 
		\hyperpage{13, 14}
  \item spectrum at infinity, \hyperpage{50}, \hyperpage{52}, 
		\hyperpage{57}
  \item Stokes
    \subitem -filtered local system, \hyperpage{34}
    \subitem arguments, \hyperpage{33}
    \subitem data, \hyperpage{36}
    \subitem filtration, \hyperpage{30}, \hyperpage{33, 34}
    \subitem graded -- -filtered local system, \hyperpage{33}

  \indexspace

  \item t-structure, \hyperpage{6}
  \item tameness, \hyperpage{5}, \hyperpage{53}
  \item Tannakian category, \hyperpage{51}
  \item Thom-Sebastiani, \hyperpage{26}, \hyperpage{51}

  \indexspace

  \item vanishing cycles, \hyperpage{3}
    \subitem at infinity, \hyperpage{17}, \hyperpage{29}, 
		\hyperpage{54}
    \subitem functor, \hyperpage{7--9}, \hyperpage{18}
    \subitem global --, \hyperpage{6}, \hyperpage{17}, \hyperpage{29}, 
		\hyperpage{45}, \hyperpage{56}

\end{theindex}

\begin{thebibliography}{10}
\providecommand{\url}[1]{{#1}}
\providecommand{\urlprefix}{URL }
\expandafter\ifx\csname urlstyle\endcsname\relax
  \providecommand{\doi}[1]{DOI~\discretionary{}{}{}#1}\else
  \providecommand{\doi}{DOI~\discretionary{}{}{}\begingroup
  \urlstyle{rm}\Url}\fi

\bibitem{A-L-M00b}
Artal~Bartolo, E., Luengo, I., Melle-Hern\'{a}ndez, A.: Milnor number at
  infinity, topology and {N}ewton boundary of a polynomial function.
\newblock Math.~Z. \textbf{233}(4), 679--696 (2000)

\bibitem{Batyrev94}
Batyrev, V.V.: {Dual polyhedra and mirror symmetry for Calabi-Yau hypersurfaces
  in toric varieties}.
\newblock J.~Algebraic Geom. \textbf{3}, 493--535 (1994)

\bibitem{B-C-S05}
Borisov, L.A., Chen, L., Smith, G.G.: The orbifold {C}how ring of toric
  {D}eligne-{M}umford stacks.
\newblock J.~Amer. Math. Soc. \textbf{18}(1), 193--215 (2005)

\bibitem{Brieskorn70}
Brieskorn, E.: {Die Monodromie der isolierten Singularit{\"a}ten von
  Hyperfl{\"a}chen}.
\newblock Manuscripta Math. \textbf{2}, 103--161 (1970)

\bibitem{Broughton88}
Broughton, S.A.: Milnor number and the topology of polynomial hypersurfaces.
\newblock Invent. Math. \textbf{92}, 217--241 (1988)

\bibitem{C-Y16}
Chen, K.C., Yu, J.D.: {The K\"unneth formula for the twisted de Rham and Higgs
  cohomologies}.
\newblock SIGMA Symmetry Integrability Geom. Methods Appl. \textbf{14}, 055
  (2018)

\bibitem{D-H-M-S17}
D'Agnolo, A., Hien, M., Morando, G., Sabbah, C.: {Topological computation of
  some Stokes phenomena on the affine line}.
\newblock Ann. Inst. Fourier (Grenoble) \textbf{70}(2), 739--808 (2020)

\bibitem{DeligneHII}
Deligne, P.: {Th{\'e}orie de {Hodge} II}.
\newblock Publ. Math. Inst. Hautes {\'E}tudes Sci. \textbf{40}, 5--57 (1971)

\bibitem{Deligne78}
Deligne, P.: {Letter to B.~Malgrange dated 19/4/1978}.
\newblock In: {Singularit{\'e}s irr{\'e}gulières, Correspondance et
  documents}, \emph{Documents math{\'e}matiques}, vol.~5, pp. 25--26.
  Soci{\'e}t{\'e} Math{\'e}matique de France, Paris (2007)

\bibitem{Deligne8406}
Deligne, P.: {Th{\'e}orie de Hodge irr{\'e}gulière (mars 1984 \& ao{{\^u}}t
  2006)}.
\newblock In: {Singularit{\'e}s irr{\'e}gulières, Correspondance et
  documents}, \emph{Documents math{\'e}matiques}, vol.~5, pp. 109--114 \&
  115--128. Soci{\'e}t{\'e} Math{\'e}matique de France, Paris (2007)

\bibitem{D-L01}
Denef, J., Loeser, F.: Geometry on arc spaces of algebraic varieties.
\newblock In: European {C}ongress of {M}athematics, {V}ol.\,{I} ({B}arcelona,
  2000), \emph{Progress in Math.}, vol. 201, pp. 327--348. Birkh\"{a}user,
  Basel (2001)

\bibitem{Dimca04}
Dimca, A.: Sheaves in topology.
\newblock Universitext. Springer-Verlag, Berlin, New York (2004)

\bibitem{Douai07}
Douai, A.: A canonical {F}robenius structure.
\newblock Math.~Z. \textbf{261}(3), 625--648 (2009)

\bibitem{Douai20}
Douai, A.: {From Hodge theory for tame functions to Ehrhart theory for
  polytopes} (2023).
\newblock \eprint{2006.13669v5}

\bibitem{D-S02a}
Douai, A., Sabbah, C.: {Gauss-Manin systems, Brieskorn lattices and Frobenius
  structures~(I)}.
\newblock Ann. Inst. Fourier (Grenoble) \textbf{53}(4), 1055--1116 (2003)

\bibitem{D-S02b}
Douai, A., Sabbah, C.: {Gauss-Manin systems, Brieskorn lattices and Frobenius
  structures~(II)}.
\newblock In: C.~Hertling, M.~Marcolli (eds.) {Frobenius manifolds (Quantum
  cohomology and singularities)}, \emph{Aspects of Mathematics}, vol. E~36, pp.
  1--18. Vieweg (2004)

\bibitem{Dubrovin96}
Dubrovin, B.: {Geometry of 2D topological field theory}.
\newblock In: M.~Francaviglia, S.~Greco (eds.) Integrable systems and quantum
  groups, \emph{Lect. Notes in Math.}, vol. 1620, pp. 120--348. Springer-Verlag
  (1996)

\bibitem{Ebeling20}
Ebeling, W.: Distinguished bases and monodromy of complex hypersurface
  singularities.
\newblock In: Handbook of geometry and topology of singularities.~{I}, pp.
  449--490. Springer, Cham (2020)

\bibitem{E-S-Y13}
Esnault, H., Sabbah, C., Yu, J.D.: {$E_1$-degeneration of the irregular Hodge
  filtration (with an appendix by M.\,Saito)}.
\newblock J.~reine angew. Math. \textbf{729}, 171--227 (2017)

\bibitem{F-J17}
Fres{\'a}n, J., Jossen, P.: Exponential motives (2017).
\newblock Available at \url{http://javier.fresan.perso.math.cnrs.fr/expmot.pdf}

\bibitem{F-S-Y18}
Fres{\'a}n, J., Sabbah, C., Yu, J.D.: {Hodge theory of Kloosterman
  connections}.
\newblock Duke Math.~J. \textbf{171}(8), 1649--1747 (2022)

\bibitem{F-S-Y20a}
Fres{\'a}n, J., Sabbah, C., Yu, J.D.: {\sortnoop{1}Quadratic relations between
  periods of connections}.
\newblock Tohoku Math.~J.~(2) \textbf{75}(2), 175--213 (2023)

\bibitem{F-S-Y20b}
Fres{\'a}n, J., Sabbah, C., Yu, J.D.: {\sortnoop{2}Quadratic relations between
  Bessel moments}.
\newblock Algebra Number Theory \textbf{17}(3), 541--602 (2023)

\bibitem{Fulton93}
Fulton, W.: {Introduction to toric varieties}, \emph{Ann. of Math. Studies},
  vol. 131.
\newblock Princeton University Press, Princeton, NJ (1993)

\bibitem{G-N95}
Garc{\'\i}a~L{\'o}pez, R., N{\'e}methi, A.: On the monodromy at infinity of a
  polynomial map.
\newblock Compositio Math. \textbf{100}, 205--231 (1996)

\bibitem{G-S15}
Gorbounov, V., Smirnov, M.: Some remarks on {L}andau-{G}inzburg potentials for
  odd-dimensional quadrics.
\newblock Glasgow Math.~J. \textbf{57}(3), 481--507 (2015)

\bibitem{H-L84}
{H{\`a} H. V.}, {Lê D. T.}: Sur la topologie des polyn\^{o}mes complexes.
\newblock Acta Math. Vietnam. \textbf{9}(1), 21--32 (1984)

\bibitem{Harder17}
Harder, A.: Hodge numbers of {L}andau-{G}inzburg models.
\newblock Adv. in Math. \textbf{378}, 107436 (2021)

\bibitem{Hartshorne72}
Hartshorne, R.: Cohomology with compact supports for coherent sheaves on an
  algebraic variety.
\newblock Math. Ann. \textbf{195}, 199--207 (1972)

\bibitem{Hertling05}
Hertling, C.: {Formes bilin{\'e}aires et hermitiennes pour des
  singularit{\'e}s: un aper\-{\c c}u}.
\newblock In: {Singularit{\'e}s}, vol.~18, pp. 1--17. Institut {\'E}lie Cartan,
  Nancy (2005).
\newblock English transl.: \eprint{2011.10099}

\bibitem{H-S09}
Hertling, C., Sabbah, C.: {Examples of non-commutative Hodge structures}.
\newblock J.~Inst. Math. Jussieu \textbf{10}(3), 635--674 (2011)

\bibitem{Hironaka77}
Hironaka, H.: Stratifications and flatness.
\newblock In: P.~Holm (ed.) Real and Complex Singularities (Oslo, 1976), pp.
  199--265. Sijthoff and Noordhoff, Alphen aan den Rijn (1977)

\bibitem{H-MS17}
Huber, A., M{\"u}ller-Stach, S.: {Periods and Nori motives}, \emph{Ergeb. Math.
  Grenzgeb.~(3)}, vol.~65.
\newblock Springer, Cham (2017)

\bibitem{K-S90}
Kashiwara, M., Schapira, P.: {Sheaves on Manifolds}, \emph{Grundlehren Math.
  Wissen.}, vol. 292.
\newblock Springer-Verlag, Berlin, Heidelberg (1990)

\bibitem{Katz90}
Katz, N.: Exponential sums and differential equations, \emph{Ann. of Math.
  studies}, vol. 124.
\newblock Princeton University Press, Princeton, NJ (1990)

\bibitem{K-K-P08}
Katzarkov, L., Kontsevich, M., Pantev, T.: {Hodge theoretic aspects of mirror
  symmetry}.
\newblock In: R.~Donagi, K.~Wendland (eds.) {From Hodge theory to integrability
  and TQFT: tt*-geometry}, \emph{Proc. Symposia in Pure Math.}, vol.~78, pp.
  87--174. American Mathematical Society, Providence, R.I. (2008)

\bibitem{K-K-P14}
Katzarkov, L., Kontsevich, M., Pantev, T.: {Bogomolov-Tian-Todorov theorems for
  Landau-Ginzburg models}.
\newblock J.~Differential Geometry \textbf{105}(1), 55--117 (2017)

\bibitem{Keller98}
Keller, B.: On the cyclic homology of ringed spaces and schemes.
\newblock Doc. Math. \textbf{3}, 231--259 (1998)

\bibitem{K-S10}
Kontsevich, M., Soibelman, Y.: {Cohomological Hall algebra, exponential Hodge
  structures and motivic Donaldson-Thomas invariants}.
\newblock Commun. Number Theory Phys. \textbf{5}(2), 231--352 (2011)

\bibitem{Kouchnirenko76}
Kouchnirenko, A.G.: {Polyèdres de Newton et nombres de Milnor}.
\newblock Invent. Math. \textbf{32}, 1--31 (1976)

\bibitem{L-P16}
Lunts, V., Przyjalkowski, V.: {Landau-Ginzburg Hodge numbers for mirrors of Del
  Pezzo surfaces}.
\newblock Adv. in Math. \textbf{329}, 189--216 (2018)

\bibitem{L-NB-S20}
{Lê D. T.}, Nu\~{n}o Ballesteros, J.J., Seade, J.: The topology of the
  {M}ilnor fibration.
\newblock In: Handbook of geometry and topology of singularities.~{I}, pp.
  321--388. Springer, Cham (2020)

\bibitem{Malgrange83bb}
Malgrange, B.: {La classification des connexions irr{\'e}gulières {\`a} une
  variable}.
\newblock In: L.~Boutet~{de Monvel}, A.~Douady, J.L. Verdier (eds.)
  {S{\'e}minaire E.N.S. Math{\'e}matique et Physique}, \emph{Progress in
  Math.}, vol.~37, pp. 381--399. Birkh{\"a}user, Basel, Boston (1983)

\bibitem{Malgrange91}
Malgrange, B.: {\'E}quations diff{\'e}rentielles {\`a} coefficients
  polynomiaux, \emph{Progress in Math.}, vol.~96.
\newblock Birkh{\"a}user, Basel, Boston (1991)

\bibitem{Massey67}
Massey, {\relax W.H}.: Algebraic topology, an introduction, \emph{Graduate
  Texts in Math.}, vol.~56.
\newblock Springer-Verlag (1967)

\bibitem{M-T09}
Matsui, Y., Takeuchi, K.: Monodromy at infinity of polynomial maps and {N}ewton
  polyhedra (with an appendix by {C}. {S}abbah).
\newblock Internat. Math. Res. Notices (8), 1691--1746 (2013)

\bibitem{M-S22}
Maxim, L.G., Sch\"{u}rmann, J.: Constructible sheaf complexes in complex
  geometry and applications.
\newblock In: Handbook of geometry and topology of singularities~{III}, pp.
  679--791. Springer, Cham (2022)

\bibitem{Milnor68}
Milnor, J.: Singular points of complex hypersurfaces, \emph{Ann. of Math.
  studies}, vol.~61.
\newblock Princeton University Press (1968)

\bibitem{Mochizuki15a}
Mochizuki, T.: {A twistor approach to the Kontsevich complexes}.
\newblock Manuscripta Math. \textbf{157}, 193--231 (2018)

\bibitem{Nemethi86}
N{\'e}methi, A.: Th{\'e}orie de {L}efschetz pour les vari{\'e}t{\'e}s
  alg{\'e}briques affines.
\newblock C.~R. Acad. Sci. Paris S{\'e}r. I Math. \textbf{303}(12), 567--570
  (1986)

\bibitem{N-S97}
N{\'e}methi, A., Sabbah, C.: Semicontinuity of the spectrum at infinity.
\newblock Abh. Math. Sem. Univ. Hamburg \textbf{69}, 25--35 (1999)

\bibitem{N-Z92}
N{\'e}methi, A., Zaharia, A.: {Milnor fibration at infinity}.
\newblock Indag. Math. \textbf{3}, 323--335 (1992)

\bibitem{Parusinski95}
Parusi\'nski, A.: A note on singularities at infinity of complex polynomials.
\newblock In: Symplectic singularities and geometry of gauge fields, 1995,
  vol.~39, pp. 131--141. Banach Center Publications (1997)

\bibitem{Pham83b}
Pham, F.: Structure de {Hodge} mixte associ{\'e}e {\`a} un germe de fonction
  {\`a} point critique isol{\'e}.
\newblock In: B.~Teissier, J.L. Verdier (eds.) {Analyse et topologie sur les
  espaces singuliers (Luminy, 1981)}, \emph{Ast{\'e}risque}, vol. 101-102, pp.
  268--285. Soci{\'e}t{\'e} Math{\'e}matique de France, Paris (1983)

\bibitem{Pham85b}
Pham, F.: {La descente des cols par les onglets de Lefschetz avec vues sur
  Gauss-Manin}.
\newblock In: A.~Galligo, J.M. Granger, P.~Maisonobe (eds.) {Systèmes
  diff{\'e}rentiels et singularit{\'e}s (Luminy, 1983)}, \emph{Ast{\'e}risque},
  vol. 130, pp. 11--47. Soci{\'e}t{\'e} Math{\'e}matique de France (1985)

\bibitem{Przhiyalkovski18}
Przhiyalkovski\u{\i}, V.V.: Toric {L}andau-{G}inzburg models.
\newblock Uspehi Mat. Nauk \textbf{73}(6(444)), 95--190 (2018).
\newblock \doi{10.4213/rm9852}

\bibitem{Qin23}
Qin, Y.: {Hodge number of motives attached to Kloosterman and Airy moments}
  (2023).
\newblock \eprint{2302.05365}

\bibitem{Raibaut10}
Raibaut, M.: Fibre de {M}ilnor motivique {\`a} l'infini.
\newblock Comptes Rendus Math{\'e}matique \textbf{348}(7-8), 419--422 (2010)

\bibitem{Raibaut12}
Raibaut, M.: Singularities at infinity and motivic integration.
\newblock Bull. Soc. math. France \textbf{140}(1), 51--100 (2012)

\bibitem{Bibi97b}
Sabbah, C.: {On a twisted de~Rham complex}.
\newblock T\^ohoku Math.~J. \textbf{51}, 125--140 (1999)

\bibitem{Bibi96bb}
Sabbah, C.: Hypergeometric periods for a tame polynomial.
\newblock Portugal. Math. \textbf{63}(2), 173--226 (2006)

\bibitem{Bibi10b}
Sabbah, C.: {On a twisted de Rham complex, II} (2010).
\newblock \eprint{1012.3818}

\bibitem{Bibi10}
Sabbah, C.: {Introduction to Stokes structures}, \emph{Lect. Notes in Math.},
  vol. 2060.
\newblock Springer-Verlag (2013)

\bibitem{Bibi15}
Sabbah, C.: {Irregular Hodge theory}, \emph{M{\'e}m. Soc. Math. France (N.S.)},
  vol. 156.
\newblock Soci{\'e}t{\'e} Math{\'e}matique de France, Paris (2018).
\newblock Chap.\,3 in collaboration with Jeng-Daw Yu

\bibitem{Bibi17}
Sabbah, C.: {Some properties and applications of Brieskorn lattices}.
\newblock J.~Singul. \textbf{18}, 239--248 (2018)

\bibitem{Bibi22a}
Sabbah, C.: {Duality for Landau-Ginzburg models} (2022).
\newblock \eprint{2212.07745}

\bibitem{S-MS12}
Sabbah, C., Saito, M.: {Kontsevich's conjecture on an algebraic formula for
  vanishing cycles of local systems}.
\newblock Algebraic Geom. \textbf{1}(1), 107--130 (2014)

\bibitem{S-Y14}
Sabbah, C., Yu, J.D.: {On the irregular Hodge filtration of exponentially
  twisted mixed Hodge modules}.
\newblock Forum Math. Sigma \textbf{3}, e9 (2015)

\bibitem{KSaito83}
Saito, K.: The higher residue pairings {$K_F^{(k)}$} for a family of
  hypersurfaces singular points.
\newblock In: Singularities, \emph{Proc. of Symposia in Pure Math.}, vol.~40,
  pp. 441--463. American Mathematical Society (1983)

\bibitem{MSaito86}
Saito, M.: {Modules de Hodge polarisables}.
\newblock Publ. RIMS, Kyoto Univ. \textbf{24}, 849--995 (1988)

\bibitem{MSaito89}
Saito, M.: {On the structure of Brieskorn lattices}.
\newblock Ann. Inst. Fourier (Grenoble) \textbf{39}, 27--72 (1989)

\bibitem{MSaito87}
Saito, M.: {Mixed Hodge modules}.
\newblock Publ. RIMS, Kyoto Univ. \textbf{26}, 221--333 (1990)

\bibitem{Schefers23}
Schefers, K.: {Derived $V$-filtrations and the Kontsevich-Sabbah-Saito theorem}
  (2023).
\newblock \eprint{2310.09979}

\bibitem{S-S85}
Scherk, J., Steenbrink, J.H.M.: {On the mixed Hodge structure on the cohomology
  of the Milnor fiber}.
\newblock Math. Ann. \textbf{271}, 641--655 (1985)

\bibitem{Shamoto17}
Shamoto, Y.: Hodge-{T}ate conditions for {L}andau-{G}inzburg models.
\newblock Publ. RIMS, Kyoto Univ. \textbf{54}(3), 469--515 (2018)

\bibitem{Shklyarov17}
Shklyarov, D.: Exponentially twisted cyclic homology.
\newblock Internat. Math. Res. Notices (2), 566--582 (2017)

\bibitem{Steenbrink22}
Steenbrink, J.H.M.: Mixed {H}odge structures applied to singularities.
\newblock In: Handbook of geometry and topology of singularities~{III}, pp.
  645--678. Springer, Cham (2022)

\bibitem{Takeuchi23}
Takeuchi, K.: {Geometric monodromies, mixed Hodge numbers of motivic Milnor
  fibers and Newton polyhedra} (2023).
\newblock \eprint{2308.09418}

\bibitem{Tibar07}
Tib{\u{a}}r, M.: Polynomials and vanishing cycles, \emph{Cambridge Tracts in
  Mathematics}, vol. 170.
\newblock Cambridge University Press, Cambridge (2007)

\bibitem{Trotman20}
Trotman, D.: Stratification theory.
\newblock In: Handbook of geometry and topology of singularities.~{I}, pp.
  243--273. Springer, Cham (2020)

\bibitem{Varchenko81}
Varchenko, A.N.: On the monodromy operator in vanishing cohomology and the
  operator of multiplication by {$f$} in the local ring.
\newblock Soviet Math. Dokl. \textbf{24}, 248--252 (1981)

\bibitem{Varchenko82}
Varchenko, A.N.: {Asymptotic Hodge structure on the cohomology of the Milnor
  fiber}.
\newblock Izv. Akad. Nauk SSSR Ser. Mat. \textbf{18}, 469--512 (1982)

\bibitem{Wasow65}
Wasow, W.: Asymptotic Expansions for Ordinary Differential Equations.
\newblock Interscience, New York (1965)

\bibitem{Yu12}
Yu, J.D.: {Irregular Hodge filtration on twisted de Rham cohomology}.
\newblock Manuscripta Math. \textbf{144}(1--2), 99--133 (2014)

\end{thebibliography}
\end{document}